\documentclass[12pt,leqno,twoside]{article}
\usepackage{amssymb}
\usepackage{amsmath}
\usepackage{amsthm,amscd}
\usepackage{a4,indentfirst,latexsym}
\usepackage{graphics}
\usepackage{mathrsfs}
\usepackage[titletoc]{appendix}
\usepackage{cite,enumitem,graphicx}
\usepackage{color}
\usepackage[colorlinks=true,urlcolor=blue,
citecolor=blue,linkcolor=blue,linktocpage,pdfpagelabels,
bookmarksnumbered,bookmarksopen]{hyperref}

\usepackage[colorinlistoftodos]{todonotes}

\allowdisplaybreaks 

\newtheorem{theorem}{Theorem}[section]
\newtheorem{lemma}{Lemma}[section]
\newtheorem{proposition}{\hspace{2em} Proposition}[section]
\newtheorem{df}{\hspace{2em} Definition}[section]
\newtheorem{corollary}{\hspace{2em} Corollary}[section]

\newtheorem{remark}{\it Remark}[section]

\newtheorem{claim}{\it {Claim}}[section]

\def\F{\mathbb{F}}

\def\N{\mathbb{N}}

\def\R{\mathbb{R}}

\def\cF{\mathcal{F}}

\def\cK{\mathcal{K}}

\newcommand{\la}{\langle}
\newcommand{\ra}{\rangle}
\newcommand{\ran}{\rangle}
\newcommand{\lan}{\langle}

\newcommand{\intsigma}{\mathring\Sigma}
\def\ol{\iota}

\newcommand{\oH}{\mathring{\mathrm{H}}^1}
\newcommand{\ii}{\mathrm{i}}

\newenvironment{altproof}[1]
{\noindent
	{\em Proof of {#1}}.}
{\nopagebreak\mbox{}\hfill $\Box$\par\addvspace{0.5cm}}
\numberwithin{equation}{section}

\title{\textbf{Blow-up solutions for the steady state of the Keller-Segel system on Riemann surfaces }}
\author{ Mohameden Ahmedou \and Thomas Bartsch \and Zhengni Hu
}
\date{\today}

\begin{document}
	\maketitle
	\begin{abstract}
		We study the following Neumann boundary problem related to the stationary solutions of the Keller-Segel system, a basic model of chemotaxis phenomena:
		\begin{equation*}
			\left\{\begin{aligned}
				-\Delta_g u +\beta u  &=\lambda\left(\frac{Ve^u}{\int_{\Sigma}  Ve^u \, dv_g}-1\right) & & \text { in } \intsigma\\
				\partial_{ \nu_g} u&=0 & & \text { on } \partial \Sigma
			\end{aligned}\right.,
		\end{equation*} 
		on a compact Riemann surface $(\Sigma, g)$ of unit area, with interior $\intsigma$ and smooth boundary $\partial \Sigma$. Here, $\Delta_g$ denote the Laplace-Beltrami operator, $\, dv_g$ the area element of $(\Sigma, g)$, and $\nu_g$ the unit outward normal to $\partial \Sigma$ and  $\lambda$ and $\beta$ are non-negative parameters, $V$  is non-negative with finite zero set.
		
		For any $m\in \N_+$ and $k,l\in \N$ with $m=2k+l$, we establish a sufficient condition on $V$ for the existence of a sequence of blow-up solutions as $\lambda$ approaches the critical values $4\pi m$, which blows up at $k$ points in the interior and $l$ points on the boundary. Moreover, the study expands to the corresponding singular problem.
	\end{abstract}
	\newpage
	\tableofcontents
	
	\section{Introduction}
	The Keller-Segel system was first introduced in~\cite{keller1970} to show the aggregation of biological species. It is a coupled parabolic system for the concentration of species $u(x,t)$ and chemical released $v(x,t)$  as the following:
	\begin{equation}~\label{ks1}
		\left\{\begin{array}{ll}
			u_t(x)=\Delta u(x)-\chi(x) \nabla(u(x) \nabla v(x)),& x \in \Omega, t>0 \\
			\Gamma v_t(x)=\Delta v(x)-\beta_0 v(x)+\delta u(x),& x \in \Omega, t>0 \\
			u(x, 0)=u_0(x), & x \in \Omega\\
			v(x, 0)=v_0(x),&  x \in \Omega \\
			\frac{\partial u(x)}{\partial \nu}=\frac{\partial v(x)}{\partial \nu}=0,&  x\in \partial \Omega
		\end{array}\right.,
	\end{equation}
	where  $\Omega\subset \R^N$ ($N\geq 1$), $\nu$ is the unit outward normal to $\partial\Omega$, $\chi, \Gamma, \beta_0$ and $\delta$ are positive parameters. The mass of $u(x, t)$ is preserved in~\eqref{ks1}, i.e. 
	$$
	\int_{\Omega} u(x, t)=\int_{\Omega} u_0(x).
	$$
	Considering the stationary solutions of~\eqref{ks1}, the problem turns out to be an elliptic system.  
	After a transformation (see~\cite{Wang2002SteadySS,jager1992explosions,gajewski1998global}, for instance), $u=C e^v$ for some constant $C$. For $v$, we obtain the following problem with the Neumann boundary condition, 
	\begin{equation}~\label{ks3}
		\left\{\begin{array}{ll}
			-\Delta v+\beta v=\lambda\left(\frac{e^v}{\int_{\Omega} e^v}-\frac{1}{|\Omega|}\right),& x \in \Omega \\
			\frac{\partial v}{\partial \nu}=0, &\text { on } \partial \Omega
		\end{array}\right.,
	\end{equation}
	where $\nu$ is the unit outward normal on $\partial\Omega$,  $\beta\text{ and }\lambda$ are parameters.
	
	In the one-dimensional case, Schaaf demonstrates the existence of non-trivial solutions using a bifurcation technique in~\cite{Schaaf1985StationarySO}. For the higher-dimensional case with  $N\geq 3$, we refer to~\cite{AgudeloPistoia2016KellerSegel,PistoiaVaira2015KellerSegel,Biler1998Chemotaxis} and references therein. 
	
	This paper specifically focuses on the case where
	$N=2$. We will now delve into the literature on this particular setting.
	
	By Struwe's technique and blow-up analysis, Wang and Wei in~\cite{Wang2002SteadySS} obtain non-constant solutions of~\eqref{ks3} for $\beta>\frac{\lambda}{|\Omega|}-\lambda_1$ and $\lambda\in (4\pi, +\infty)\setminus 4\pi \mathbb{N}_+$, where $\lambda_1$ is the first eigenvalue of $-\Delta$ with the Neumann boundary condition. Independently, Senba and Suzuki obtain the same result in~\cite{Senba2000some}. Battaglia generates their result for $\lambda\in (0,+\infty)\setminus 4\pi\mathbb{N}_+$ and  $\beta$ with any sign in~\cite{Battaglia2018}. He proves the existence of non-constant solutions of~\eqref{ks3} with some algebraic conditions involved with $\beta,\lambda$ and eigenvalues $\{\lambda_i\}_{i=1}^{+\infty}$ by the variational method and Morse theory. 
	
	However, when $\lambda$ approaches the critical value set $4\pi\N_+$, the blow-up phenomena may occur. 
	Del Pino and Wei in \cite{pino_collapsing_2006} construct positive value bubbling solutions for the Neumann boundary  problem on bounded domains $\Omega$ with $\beta>0$
	\begin{equation}
		\label{eq:espsilon}
		\begin{cases}
			-\Delta u+\beta u=\varepsilon^2 e^u & \text{ in }\Omega\\
			\partial_{\nu} u=0 & \text{ on }\partial\Omega
		\end{cases},
	\end{equation}
	by the Lyapunov-Schmidt reduction. In particular, the  sequence of bubbling solutions blows up at $k$ distinct points 
	$\xi_1, \cdots, \xi_k$ inside the domain $\Omega$ and $l$ distinct points 
	$\xi_{k+1}, \cdots, \xi_{k+l}$ on the boundary of $\Omega$, i.e. as $\varepsilon\rightarrow 0$
	\[ u_{\varepsilon} \rightarrow \sum_{i=1}^k 8\pi \delta_{\xi_i}+\sum_{i=k+1}^{k+l} 4 \pi \delta_{\xi_{i}},\]
	where $\delta_{\xi}$ is the Dirac mass. 
	Subsequently, Del Pino, Pistoia, and Vaira in~\cite{DelPino2016KellerSegel} construct solutions of~\eqref{eq:espsilon} which blow up along the whole boundary $\partial\Omega$.

	This paper studies the Neumann boundary problem on a compact Riemann surface $\Sigma$ with  smooth boundary $\partial \Sigma$:
	\begin{equation}~\label{eq:main_eq}
		\left\{\begin{aligned}
			-\Delta_g u+\beta u &=\lambda\left(\frac{V e^u}{\int_{\Sigma}  Ve^u \, dv_g}-\frac{1}{|\Sigma|_g}\right) & & \text { in } \intsigma\\
			\partial_{ \nu_g} u&=0 & & \text { on } \partial \Sigma
		\end{aligned}\right.,
	\end{equation} 
	where the parameters $\lambda,\beta\in\mathbb{R}$ and $V$ is a non-negative smooth function with a finite zero set denoted as $\{q_1,\cdots,q_{\iota}\}$ for some $\iota\in\mathbb{N},$ $\intsigma:=\Sigma\setminus \partial\Sigma$ is the interior of $\Sigma$, $\Delta_g$ is the Laplace-Beltrami operator, $\, dv_g$ is the area element in $(\Sigma, g)$, $|\Sigma|_g=\int_{\Sigma} \, dv_g$, and $\nu_g$ is the unit outward normal of $\partial \Sigma$.
	
	This paper delves into the study of the blow-up solutions of the problem~\eqref{eq:main_eq}.  For integers $k,l\in \N \text{ with } 2k+l=m$, we establish a sufficient condition for blow-up solutions. Moreover, the precise locations of blow-up points are explicitly characterized by the ``stable" critical point of a reduced function  $\cF^V_{k,l}$.

	The non-linear equation in \eqref{ks3} with $\beta \equiv 0$ is a mean field equation. This equation arises in various branches of mathematics and physics, such as statistical mechanics  \cite{Caglioti1992EulerFlows, Caglioti1995EulerFlowsII,Kiessling1993LogInteractions}, Abelian Chern-Simons gauge theory  \cite{Nolasco1999ChernSimonsHiggs,Yang2001Solitons,Caffarelli1995VortexCondensation,Tarantello1996MultipleCondensates}, and conformal geometry  \cite{ChangYang1987GaussianCurvature,Tarantello2010MeanFieldEquations,KazdanWarner1974CurvatureFunctions,ChangGurskyYang1993ScalarCurvature,ChangYang1988ConformalDeformation,ChingQuan1993NirenbergProblem}. When it is equipped with Dirichlet boundary conditions, by  Lyapunov-Schmidt reduction the blow-up solutions of the mean field equations are well-studied both in domains of  Euclidean spaces $\mathbb{R}^2$ (refer to \cite{ del_pino_singular_2005, pino_collapsing_2006, Esposito2005} and the references therein) and on Riemann surfaces without boundaries (refer to \cite{Bartolucci2020, Esposito2014singular,figueroa2022bubbling}). 
	Recently, \cite{HBA2024} obtained blow-up solutions with Neumann boundary conditions on Riemann surfaces with boundaries under the condition of nonvanishing of a quantity related to $V$, Gaussian curvature of $\Sigma$ and geodesic curvature of $\partial\Sigma$. 
	
	As in these papers, our approach to finding blow-up solutions of \eqref{eq:main_eq} is based on variational methods combined with the Lyapunov-Schmidt reduction.
	In comparison to~\cite{HBA2024}, we relax the condition on the nonvanishing quantities  and extend our analysis to the case where $\beta\neq 0.$
	
	It is noteworthy that we allow $V$ to be $0$ at $q_i$ for any $i=1,\cdots,\ol$ where $\iota\in\N$. So, it is also possible to  establish  blow-up solutions for the following singular problem:
	\begin{equation}
		\label{eq:singular}
		\left\{\begin{aligned}
			-\Delta_g \tilde{u}+\beta \tilde{u} &=\lambda\left(\frac{\tilde{V} e^{\tilde{u}} }{\int_{\Sigma}  \tilde{V}e^{\tilde{u}} \, dv_g} -\frac{1}{|\Sigma|_g}\right)-\sum_{i=1}^{\ol}\frac{\varrho(q_i)}{2}n_i \left(\delta_{q_i}-\frac 1 {|\Sigma|_g}\right)& & \text { in } \intsigma\\
			\partial_{ \nu_g} \tilde{u}&=0 & & \text { on } \partial \Sigma
		\end{aligned}\right..
	\end{equation}
	Here, $\tilde{V}$ is a positive smooth function, 
	$\varrho(\xi)$ equals 
	$8\pi$ if $\xi\in \intsigma$ and equals 	$4\pi$  if $ \xi\in \partial \Sigma$ and $n_i\in \N_+$ for $i=1,\cdots,\ol.$ Notably, the problem~\eqref{eq:singular} emerges as a specific instance of~\eqref{eq:main_eq}. To elucidate, we define the Green's function through the following equations for any $\xi\in \Sigma$:
	\begin{equation}~\label{eq:green}
		\begin{cases}
			-\Delta_g G^g(x,\xi)+\beta  G^g(x,\xi)   =\delta_{\xi} -\frac{1}{|\Sigma|_g} &  x\in \intsigma\\
			\partial_{ \nu_g } G^g(x,\xi) =0 & x\in \partial \Sigma\\
			\int_{\Sigma} G^g(x,\xi) \, dv_g(x) =0
		\end{cases}. 
	\end{equation}
	We take $u(x)=\tilde{u}(x)+ \sum_{i=1}^{\iota} \frac{\varrho(q_i)}{2}n_iG^g(x,q_i)$ and  $V(x)=\tilde{V}(x) e^{- \sum_{i=1}^{\iota} \frac{\varrho(q_i)}{2} n_iG^g(x,q_i)}.$ $u$ satisfies the equations~\eqref{eq:main_eq} in which $V$ is a  non-negative smooth  function with the zero set $\{q_1,\cdots,q_{\ol}\}.$

	We present the main results, starting with defining the ``stable" critical points set as  in~\cite{del_pino_singular_2005,Esposito2005,Li1997OnAS}.
	\begin{df} 
		Let $F: D \rightarrow \mathbb{R}$ be a $C^{1}$-function and  $K$ be  a compact subset of critical points of $F$, i.e. 
		\[ K\subset \subset \{ x\in D: \nabla F(x)=0\}. \] 
		A critical set $K$ is $C^1$-stable if for any closed neighborhood $U$ of $K$ in $ D$, there exists $\varepsilon>0$ such that if $G: D \rightarrow \mathbb{R}$ is a $C^{1}$-function with $\|F-G\|_{C^1(U)}<\varepsilon$, then $G$ has at least one critical point in $U$.
	\end{df}
	The main theorem asserts the existence of a sequence of blow-up solutions for~\eqref{eq:main_eq}, with these solutions exhibiting blow-up behavior at the stable critical points of a reduced function $\cF^V_{k,l}$.  We define  the configuration set as follows:
	\begin{eqnarray*}
		\Xi_{k,l}=\intsigma^k\times(\partial\Sigma)^{l}\setminus \F_{k,l}(\Sigma),
	\end{eqnarray*}
	where  $	\F_{k,l}(\Sigma):=\left\{\xi=(\xi_1,\cdots,\xi_{k+l}): \xi_i=\xi_j
	\text{ for some } i=j \right\}$ is called the thick diagonal. 
	Let $\Sigma^{\prime}:=\{ x\in \Sigma: V(x)>0\}$ and then we define that
	\begin{equation}
		\label{eq:def_xi_km} \Xi_{k,l}^{\prime}:= \Xi_{k,l}\cap (\Sigma^{\prime})^{k+l}.
	\end{equation} 
	The function is well-defined on $\Xi^{\prime}_{k,l}.$ Specifically,  
	$\cF^V_{k,l}: \Xi^{\prime}_{k,l} \subset \intsigma^k\times(\partial\Sigma)^{l}\rightarrow \mathbb{R}$, 
	\begin{eqnarray}\label{eq:def_red_f}
		\cF^V_{k,l}(\xi_1,\cdots,\xi_{k+l})&=&\sum_{i=1}^{k+l}\varrho^2(\xi_i) R^g(\xi_i)+\sum^{k+l} _{ 
			\begin{array}{l}
				i,j=1\\
				i\neq j
		\end{array} }  \varrho(\xi_i)\varrho(\xi_j) G^g(\xi_i,\xi_j)\\
		&&+  \sum_{i=1}^{k+l} 2\varrho(\xi_i)\log V(\xi_i),\nonumber
	\end{eqnarray}
	where $R^g$ is the Robin's function and $G^g(\cdot,\xi)$ is the Green's function (for details, refer to Section~\ref{prelim}). 
	\begin{theorem}~\label{main_thm}
		Given  $m\in \mathbb{N}_+, k,l \in \N$ with $m=2k+l$, if $ K\subset\subset \Xi_{k,l}^{\prime} $  is a $C^1$-stable critical point set of $ \cF^V_{k,l}$, then there exists $\varepsilon_0>0$ such that for any $\varepsilon\in (0,\varepsilon_0)$
		a family of blow-up solutions $u_{\varepsilon}$ of~\eqref{eq:main_eq} with $\lambda_{\varepsilon}\rightarrow  4\pi m$  can be constructed.  Furthermore, solutions $u_\varepsilon$ blow up precisely at points $\xi_{1}, \cdots, \xi_{k+l}$ with $ \xi=\left(\xi_{1}, \cdots, \xi_{k+l}\right)$ in $ K,$ (up to a subsequence) as $\varepsilon\rightarrow 0$
		\begin{equation*}
			\frac{\lambda_\varepsilon V e^{u_\varepsilon}} { \int_{\Sigma} V e^{u_\varepsilon}  \, dv_g} \rightarrow \sum_{ i=1}^k 8\pi \delta_{\xi_i}+ \sum_{i=k+1}^{k+l} 4\pi \delta_{\xi_i}, 
		\end{equation*}
		which is convergent as measures on $\Sigma$. 
	\end{theorem}
	Theorem~\ref{main_thm} indicates that for any given $k, l \in \mathbb{N}$ satisfying $2k + l = m$, we can construct a family of blow-up solutions that blow up at a stable critical point of $\cF^V_{k,l}$. Clearly, for different  $(k,l)$, the blow-up solutions are distinct, as they blow up at different points. Based on this observation, we immediately obtain the following corollary regarding the multiplicity of blow-up solutions: 
	\begin{corollary}
		Under the same assumptions as Theorem~\ref{main_thm}, for $m \in \mathbb{N}_+$, there exist at least $1 + \lfloor m/2 \rfloor$ distinct families of blow-up solutions to \eqref{eq:main_eq} as $\lambda \to 4\pi m$, where $\lfloor m/2 \rfloor$ denotes the largest integer less than or equal to $m/2$.
	\end{corollary}
	
	Define the set of global minimum points  of $\cF^V_{k,l}$ as follows:
	\begin{equation}
		\label{eq:k_min} \cK_{k,l}:=\left\{x\in\Xi_{k,l}^{\prime}: \cF^V_{k,l}(\xi)=\inf_{\Xi_{k,l}^{\prime}}\cF^V_{k,l}\right\}.
	\end{equation}
	\begin{corollary}\label{cor:mf}
		Given  $m\in \mathbb{N}_+, k,l \in \N$ with $m=2k+l$, suppose that $\cK_{k,l}\neq \emptyset$. Then, the conclusions in Theorem~\ref{main_thm} hold. Furthermore,
		$u_{\varepsilon}$ has  $k$ local maximum points $\xi^\varepsilon_{i}$ in $\intsigma$ for $i=1, \cdots, k$ and $l$ local maximum points $\xi^\varepsilon_{i}$  restricted to the boundary $\partial\Sigma $ for $i=k+1,\cdots, k+l $ such that up to a subsequence $(\xi^\varepsilon_{1},\cdots,\xi^\varepsilon_{k+l})$ converges to $\xi:=(\xi_1,\cdots,\xi_{k+l})\in \cK_{k,l} $ with\[ \lim_{\varepsilon \rightarrow 0} \cF^V_{k,l}(\xi^\varepsilon_{1},\cdots,\xi^\varepsilon_{k+l})= \min_{\Xi_{k,l}^{\prime}}\cF^V_{k,l}= \cF^V_{k,l}(\xi).\]
	\end{corollary}
		If the zero set of $V$ is empty, the behavior of $\cF^V_{k,l}$ as $\xi$ approaches $\partial \Xi^{\prime}_{k,l}$ results in its divergence towards $+\infty$ (as in Lemma~\ref{H_INFTY}). This divergence suggests the presence of at least one global minimum point in the interior of $\Xi^{\prime}_{k,l}$. Additionally, a local minimum point is inherently ``stable''. Consequently, we have the following corollary:
		\begin{corollary}
			Given  $m\in \mathbb{N}_+, k,l \in \N$ with $m=2k+l$, if $V$ is a positive function, then there exists $\varepsilon_{0}>0$ such that   for $\varepsilon\in (0,\varepsilon_0)$ a family of blow-up solutions $u_{\varepsilon}$ of~\eqref{eq:main_eq} with $\lambda_{\varepsilon}\rightarrow 4\pi m$ can be constructed. Moreover, $u_\varepsilon$ satisfied the all conclusions in Corollary~\ref{cor:mf}.  
		\end{corollary} 
		\begin{remark}
			\begin{itemize}
				\item 	When $V(q)=0$ for some $q\in \Sigma$, a complication arises. As $\xi$ approaches $\partial \Xi^{\prime}_{k,l}$, there are cases where the sum of the first terms tends to $+\infty$ while the last term approaches $-\infty$, leading to an indeterminate behavior of $\cF^V_{k,l}$. 
				\item  It is observed that the constructed blow-up points in Theorem~\ref{main_thm} do not coincide with the zero set of $V$. Due to the high singularity of this problem,  constructing blow-up solutions that blow up at a singular point of mean field equation~\eqref{eq:main_eq}, i.e. $q\in \{x\in \Sigma: V(x)=0\}$,  remains a challenging open problem.
			\end{itemize}
		\end{remark}

	\section{Preliminaries}\label{prelim} 
	Throughout this paper, we use the terms ``sequence" and ``subsequence" interchangeably, as the distinction is not crucial for the context of our analysis. The constant denoted by $C$ in our deduction may assume different values across various equations or even within different lines of equations. 
	We also denote $B_r(y)=\{y\in\R^2: |y|<r\}$ and $A_r(y):=B_{2r}(y)\setminus B_r(y).$ For any $\xi\in \Sigma$ we also denote that 
	$\varrho(\xi)$ is  
	$8\pi$ if $\xi\in \intsigma$ and equals 	$4\pi$  if $ \xi\in \partial \Sigma$. 
	
	To construct the ansatz for solutions of problem~\eqref{eq:main_eq}, we firstly introduce a family of isothermal coordinates (see \cite{chern1955,Esposito2014singular,yang2021125440}, for instance). For any $\xi\in\intsigma$,  there exists an isothermal coordinate system $\left(U(\xi), y_{\xi}\right)$ such that $y_{\xi}$  maps an open  neighborhood $U(\xi)$ around $\xi$  onto an open disk $B^{\xi}$ with radius $2r_{\xi}$ and $y_\xi(\xi)=(0,0)$,  in which the Riemann metric has the form as follows: 
	$$g=\sum_{i=1}^2 e^{\hat{\varphi}_\xi(y_{\xi}(x))}  \mathrm{d} x^i \otimes \mathrm{d} x^i.$$ 
	Similarly, 
	for $\xi\in \partial\Sigma$ there exists an isothermal coordinate system $\left(U(\xi), y_{\xi}\right)$ around $\xi$ such that the image of $y_{\xi}$ is a half disk $B^{\xi}:=\{ y=(y_1,y_2)\in \R^2: |y|<2r_{\xi}, y_2\geq 0\}$ with a radius $2r_{\xi}$, $y_\xi(\xi)=(0,0)$ and   $y_{\xi}\left(U(\xi)\cap \partial \Sigma\right)= \{ y=(y_1,y_2)\in \R^2: |y|<2r_{\xi}, y_2=0 \}$,  in which  the Riemann metric has the form as follows:  $$g=\sum_{i=1}^2 e^{\hat{\varphi}_\xi(y_{\xi}(x))}  \mathrm{d} x^i \otimes \mathrm{d} x^i.$$ 
	
	Let 
	$K_g$ be the Gaussian curvature of $\Sigma$ and $k_g$ be the geodesic curvature of the boundary $ \partial\Sigma$. Then,  for $\xi\in \Sigma$
	\begin{equation}
		\label{eq:Gauss}
		-\Delta \hat{\varphi}_\xi(y) = 2K_g\big(y^{-1}_\xi(y)\big) e^{\hat{\varphi}_\xi(y)} \quad\text{for all } y\in B^\xi. 
	\end{equation}
	and for $\xi\in \partial\Sigma$,
	\begin{equation}\label{eq:b_restrict}
		\frac{\partial}{\partial y_2}  \hat{\varphi}_{\xi}(y) =- 2k_g(y_{\xi}^{-1}(y)) e^{ \frac{\hat{\varphi}_{\xi}(y)}{2}} \quad \text{ for all } y\in B^{\xi}\cap \{ y_2=0\}.
	\end{equation} 
	For   $\xi\in \Sigma$ and $0<r\le 2r_\xi$  we set 
	\[
	B_r^\xi := B^\xi \cap \{ y\in\R^2: |y|< r\}\quad \text{and}\quad U_{r}(\xi):=y_\xi^{-1}(B_{r}^{\xi}).
	\]
	Both $y_\xi$ and $\hat{\varphi}_\xi$ are assumed to depend smoothly on $\xi$ as in~\cite{Esposito2014singular,figueroa2022bubbling} for closed surfaces. With a slight modification, we can assume the smooth dependence of $\xi$ for Riemann surfaces with boundary.  Moreover, we can assume   $\hat{\varphi}_{\xi}(0,0)=0$ and $\nabla \hat{\varphi}_{\xi}(0,0)=\begin{cases}
		(0,0) & \text{ for } \xi\in\intsigma\\
		(0, -2k_g(\xi)) & \text{ for }\xi\in \partial\Sigma
	\end{cases}.$ 
	As in~\cite{yang2021125440},  the Neumann boundary conditions preserved by the isothermal coordinates  in following sense:  for any $\xi\in \partial\Sigma$ and  $x\in y_{\xi}^{-1}\left( B^{\xi} \cap \partial \mathbb{R}^{2}_+\right)$, we have
	\begin{equation}\label{eq:out_normal_derivatives}
		\left(y_{\xi}\right)_*(\nu_g(x))=\left. -e^{ -\frac{\hat{\varphi}_{\xi}(y)}2} \frac {\partial} { \partial y_2 }\right|_{	y=y_{\xi}(x)}.
	\end{equation}
	We define the cut-off function $\chi_\xi\in C^\infty(\Sigma,[0,1])$ by 
	\begin{equation}
		\label{eq:cut_off_i} \chi_\xi(x) =	\left\{ \begin{array}{ll}
			\chi\big(\frac{|y_\xi(x)|}{r_0}\big) &\quad \text{if $x\in U(\xi)$}\\
			0&\quad \text{if $x\in \Sigma\setminus U(\xi)$}
		\end{array}  \right.,
	\end{equation}
	where $r_0\in (0, \frac  1 2  r_{\xi}] $ which will be selected later. 
	The Robin's function is defined as follows: 
	\[ R^g(\zeta):=\lim_{x\to\zeta}\left(G^g(x,\zeta)+\frac4{\varrho(\zeta)}\log d_g(x,\zeta)\right). \]
	Observe that for $\zeta\in U(\xi)$, 
	$
	\lim_{x\to\zeta}\frac{d_g(x,\zeta)}{\big|y_\xi(x)-y_\xi(\zeta)\big|} = e^{\frac12
		\hat{\varphi}_\xi\circ y_\xi(\zeta)}.
	$
	It follows  
	\begin{equation}\label{eq:robin}
		R^g(\zeta) = \lim_{x\to\zeta}\left(G^g(x,\zeta)+\frac4{\varrho(\zeta)}\log\big|y_\xi(x)-y_\xi(\zeta)\big|\right)
		+ \frac 2 {\varrho(\zeta)}\hat{\varphi}_\xi\big(y_\xi(\zeta)\big).
	\end{equation}
	In particular, using the assumption  $\hat{\varphi}_\xi\big(y_\xi(\xi)\big)=\hat{\varphi}_\xi(0,0)=0$, we obtain that 
	$$
	R^g(\xi) = \lim_{x\to\xi}\left(G^g(x,\xi)+\frac4{\varrho(\xi)}\log\big|y_\xi(x)\big|\right) . 
	$$
	Let the function
	\begin{equation*}
		\Gamma^g_{\xi}(x)=\Gamma^g(x,\xi)= \left\{ \begin{array}{ll}	\frac{1}{2\pi} \chi_{\xi}(x)\log\frac{1}{|y_{\xi}(x)|}& \text{ if }\xi\in \intsigma \\	{\frac{1}{\pi}}\chi_{\xi}(x)\log\frac{1}{|y_{\xi}(x)|}	& \text{ if } \xi\in\partial\Sigma \end{array}\right..
	\end{equation*}  
	Decomposing the Green's function  
	$ G^g(x,\xi)= \Gamma^g_{\xi}(x)+ H^g_{\xi}(x),$ we have the function $H^g_{\xi}(x):= H^g(x,\xi)$ that solves the following equations: 
	\begin{equation}\label{eqR}
		\left\{\begin{array}{ccll}
			-\Delta_g H^g_\xi+\beta H^g_\xi
			&=&-\beta\frac{4}{\varrho(\xi)}\chi_{\xi} \log \frac 1 {|y_{\xi}|} 
			+ \frac{4}{\varrho(\xi)} (\Delta_g \chi_{\xi}) \log \frac 1 {|y_{\xi}|}& \\
			&&+ \frac{8}{\varrho(\xi)} \la\nabla\chi_{\xi}, \nabla\log \frac 1 {|y_{\xi}|}\ra_g 	-\frac 1 {|\Sigma|_g}, & \text{ in }\intsigma\\
			\partial_{ \nu_g} H^g_\xi&=&- \frac{4}{\varrho(\xi)}(\partial_{ \nu_g} \chi_{\xi}) \log\frac{1}{|y_{\xi}|}-\frac{4}{\varrho(\xi)}\chi_{\xi}\partial_{ \nu_g}  \log\frac{1}{|y_{\xi}|}, &\text{ on }\partial \Sigma\\
			\int_{\Sigma} H^g_\xi \, dv_g&=& - \frac{4}{\varrho(\xi)}\int_{\Sigma}  \chi_{\xi}\log\frac{1}{|y_{\xi}|} \, dv_g &
		\end{array}\right..\end{equation}
	By the regularity of elliptic equations (see Lemma~\ref{lem:re_green}), there is a unique  solution  $H^g(x,\xi)$ that solves \eqref{eqR} in $C^{1,\alpha}(\Sigma)$ for $\alpha\in (0,1)$. 
	$H^g(x,\xi)$ is the regular part of  $G^g(x,\xi)$. 
	It is clear that $R^g(\xi)=H^g(\xi,\xi)$ and $H^g(\xi,\xi)$ is independent of the choice of the cut-off function $\chi$ and the local chart. 
	For $\delta>0$, we consider 
	\begin{eqnarray}\label{eq:def_M_delta}
		{M}_{\delta}:=\left\{
		\begin{array}{ll}
			\xi\in \Xi^\prime_{k,l}:&  d_g( \xi_i, \partial\Sigma) \geq \delta \text{ for } i=1,\cdots,k; \\
			&
			d_g(\xi_i,\xi_j)\geq \delta \text{ for   } i\neq j; V(\xi_i)\geq \delta \text{ for } i=1,\cdots, k+l 
		\end{array} \right\},
	\end{eqnarray} a compact subset $ \Xi^\prime_{k,l}$, where 
	$d_g (\cdot,\cdot):\Sigma \times\Sigma \rightarrow \mathbb{R} $ is the geodesic distance with respect to metric $g$ and $d_g(p, \partial \Sigma):=\inf_{q\in \partial \Sigma} d_g(p,q)$ for any $p\in \intsigma$. 
	We observe that for any $\alpha\in (0,1)$,  $G^g(x, \xi) \in C^{\infty}(\Sigma\setminus \{ \xi \})$ and $H^g(x,\xi)$  is $C^{1,\alpha}(\Sigma)$, too. Thus, $ \cF^V_{k,l}$ is $C^{1,\alpha}(M_{\delta})$ for any fixed $\delta>0$.

	To study the blow-up solutions of $\eqref{eq:main_eq}$, we consider the weak solution of the following problem in the space $\oH:=\left\{ u\in H^1(\Sigma): \int_{\Sigma} u \, dv_g=0\right\},$
	\begin{equation}~\label{rho_2}
		\begin{cases}
			(-\Delta_g+\beta) u= \varepsilon^2V  e^u  -\overline{\varepsilon^2 V e^u}& \quad \text{ in }\intsigma\\
			\partial_{\nu_g} u=0& \quad\text{ on } \partial\Sigma
		\end{cases}, 
	\end{equation}
	such that 
	$\varepsilon^2 Ve^u\rightarrow  \sum_{i=1}^{k+l} \varrho(\xi_i) \delta_{\xi_i},$
	convergent in a sense of measures on  $\Sigma $ as $\varepsilon\rightarrow 0$, for some 
	$ \xi=(\xi_1,\cdots,\xi_{k+l})\in \Xi^\prime_{k,l} .$
	If we take $\lambda=\varepsilon^2 \int_{\Sigma} Ve^{u} \, dv_g $, the weak solutions of~\eqref{rho_2} must be  the weak solutions of~\eqref{eq:main_eq}. So we try to construct a sequence of blow-up solutions of~\eqref{rho_2} as $\varepsilon\rightarrow 0$ and then pass back to the original  problem~\eqref{eq:main_eq} as $\lambda\rightarrow 4\pi m$. 
	
	It is well known that 
	$ u_{\tau,\eta}(y)=\log \frac{ 8\tau^2 }{(\tau^2\varepsilon^2 +|y-\eta|^2)^2}$ for $(\tau, \eta)\in (0,\infty)\times  \mathbb{R}^2 $ are all the solutions of the Liouville-type equations, 
	\begin{equation*}
		\begin{cases}
			-\Delta u= \varepsilon^2 e^u \quad \text{ in } \mathbb{R}^2,\\
			\int_{\mathbb{R}^2}  e^u <\infty. 
		\end{cases}
	\end{equation*}
	Our goal is to construct approximate solutions of~\eqref{rho_2} applying the pull-back of $u_{\tau,\eta}$ to $\Sigma$ by isothermal coordinates and selecting appropriate values for $\tau$ and $\xi$.
	Define 
	\[U_{\tau,\xi}(x) = u_{\tau,0}(y_{\xi}(x))= \log \frac{ 8\tau^2 }{( \tau^2\varepsilon^2+ |y_{\xi}(x)|^2 )}, \text{ for all } x\in U(\xi)\]
	and   $U_{\tau,\xi}(x)=0$ for all $x\in \Sigma \setminus U(\xi)$. 
	For any function $f\in L^1(\Sigma)$, we denote its average over $\Sigma$ as $\overline{f}=\frac{1}{|\Sigma|g}\int_{\Sigma} f \, dv_g$.
	Then, we introduce a projection operator $P$, which is used to project $U_{\tau,\xi}$ into the space $\oH$. The projected function $P U_{\tau,\xi}$ is defined as the solution to the problem:
	\begin{equation}~\label{eq:projection}
		\left\{\begin{array}{ll}
			(-\Delta_g+\beta) P U_{\tau,\xi}(x)= \varepsilon^2 \chi_{\xi}e^{-\varphi_{\xi}}e^{U_{\tau, \xi}} - \overline{\varepsilon^2 \chi_{\xi} e^{-\varphi_{\xi}}e^{U_{\tau, \xi}} }  , &  x\in \intsigma,\\
			\partial_{\nu_g} P U_{\tau,\xi}=0, &x\in \partial \Sigma, \\
			\int_{\Sigma} P U_{\tau, \xi} \, dv_g=0, &
		\end{array}\right. 
	\end{equation}
	For $\beta \neq 0$, the last condition of zero integral of $P U_{\tau, \xi}$ over $\Sigma$ can be inferred from the preceding equations via the divergence theorem. However, it is explicitly included to address the case when $\beta = 0$, ensuring the solution criteria for all $\beta\geq 0$. The solution of \eqref{eq:projection} is unique in $\oH$ and $PU_{\tau,\xi}$ in $C^{\infty}({\Sigma})$ as per regularity theory in Lemma~\ref{lem:schauder}, ensuring that $PU_{\tau,\xi}$ is well-defined. 
	
		Let $$\psi^0_{\tau,\eta}(y)=\frac{\partial}{\partial_{\tau} } u_{\tau,\eta}(x)= \frac{2}{\tau}  \frac{ |y-\eta|^2-\tau^2\varepsilon^2}{ |y-\eta|^2+\tau^2\varepsilon^2},$$ and 
	$$\psi^j_{\tau,\eta} (y)= \frac{\partial }{\partial \eta_j } u_{\tau,\eta}(x)=4\frac{y_j-\eta_j}{ \tau^2\varepsilon^2 +|y-\eta|^2}, $$ 
	for $ j=1,2.$
	It is observed that the derivatives above satisfy the equation:
	$-\Delta \psi= \varepsilon^2 e^{u_{\tau,\eta}}\psi  \text{  in }\mathbb{R}^2,$
	where $\psi=\psi^{j}_{\tau,\eta}, \text{ for }j=0,1,2$. 
	The function $\Psi^j_{\tau,\xi}$ is then defined as the pull-back of $\psi^j_{\tau,0}$ under the isothermal coordinate $y_{\xi}$, i.e.
	$\Psi^j_{\tau, \xi}(x)= \psi^j_{\tau,0}( y_{\xi}(x)),$
	for any $x\in y_{\xi}^{-1}(B^{\xi}_{2r_0}). $
	Let $P\Psi^j_{\tau,\xi}$ be a projection into $\overline{\mathrm{H}}^1$ of $\Psi^j_{\tau,\xi}$,  for $\xi\in \Sigma$ and $j=0,1,\cdots,\ii(\xi_i)$, where $\ii(\xi_i)\text{ equals }
	2 \text{ if } 1\leq i\leq k\text{ and } \text{ equals }
	1 \text{ if } k+1\leq i\leq m$. $P\Psi^j_{\tau,\xi}$ is defined as the solution of 
	\begin{equation}~\label{ppsi}
		\left\{\begin{array}{ll}
			(-\Delta_g+\beta) P \Psi^j_{\tau,\xi}= \varepsilon^2\chi_{\xi} e^{-\varphi_{\xi}} e^{U_{\tau, \xi}} \Psi^j_{\tau,\xi}- \overline{ \varepsilon^2 \chi_{\xi} e^{-\varphi_{\xi}} e^{U_{\tau, \xi}} \Psi^j_{\tau,\xi} }   , & x\in \intsigma,\\
			\partial_{\nu_g} P \Psi^j_{\tau,\xi}=0, &x\in \partial \Sigma, \\
			\int_{\Sigma} P \Psi^j_{\tau, \xi}=0.&
		\end{array}\right.
	\end{equation}
	By the regularity theory in Lemma~\ref{lem:schauder} the solution to problem~\eqref{ppsi} is unique and smooth on $\Sigma$. Hence, $P\Psi^{j}_{\tau,\xi}$ is well-defined and lies in the space $C^{\infty}(\Sigma)$.

	For any  $\xi=(\xi_1,\cdots,\xi_{k+l})\in {M}_{\delta}$, we can establish an isothermal chart around $y_{\xi_i}$ for each point $\xi_i$ for $i=1,\cdots, k+l $. Given the compactness of $\Sigma$, it is possible to select a uniform radius $r_{\xi_i}>0$ for any $\xi\in M_{\delta}$, denoted as $2r_0$. This radius is sufficiently small and depends only on $\delta$ and $\partial \Sigma$. Moreover, we ensure that $U_{4r_0}(\xi_i)\cap  U_{4r_0}(\xi_j)=\emptyset $ for any $i,j=1,\cdots, k+l $ with $i\neq j$ and $U_{4r_0}(\xi_i)\cap \partial\Sigma=\emptyset$ for $i=1,\cdots, k$. For any $i=1, \cdots, k+l$, we define the scaling parameter $\tau_i$ as:
	\begin{equation}~\label{tau}
		\tau_{i}(x)=\sqrt{\frac{1}{8} V(x){e}^{\varrho(\xi_i)H^g\left(x, \xi_{i}\right)+\sum_{j\neq i} \varrho(\xi_j)G^g\left(x, \xi_{j}\right)}}.
	\end{equation}
	For simplicity, we denote that $U_{i}=U_{\tau_{i}(\xi), \xi_{i}}$, $\chi_i=\chi_{\xi_i}$, $\varphi_i:=\varphi_{\xi_i}$, $\hat{\varphi}_i=\hat{\varphi}_{\xi_i}$ and $\tau_i=\tau_i(\xi_i)$. The formulation of the scaling parameter $\tau_i$ is chosen for technical considerations.

	We consider the manifold  for given  $k,l\in \N$ and a positive constant $\varepsilon>0$,
	\[ \mathcal{M}^{k,l}_{\varepsilon}:=\left\{  \sum_{i=1}^{k+l} PU_i: 
	\xi_i\in \intsigma \text{ for } i=1,\cdots,k  \text{ and } \xi_i \in \partial\Sigma \text{ for } i=k+1,\cdots, k+l\right\}. \] 
	The functions in manifold $\mathcal{M}^{k,l}_{\varepsilon}$ serve as approximate solutions of the problem~\eqref{rho_2}. 
	We then denote the projected function for any $i=1,\cdots, k+l $ and $j=0,\cdots,\ii(\xi_i)$ as
	\[P\Psi^j_i:= P\Psi^j_{\tau_i(\xi),\xi_i}.\]
	These projected functions generate a subspace of $\oH$, $\{ P\Psi^j_i: i=1,\cdots, k+l , j=1,\cdots,\ii(\xi_i)\}$ denoted as $K_{\xi}$.
	Furthermore, we introduce an inner product for the space $\oH$ as follows:
	\[ \la \psi,\phi\ra:= \int_{\Sigma} \la \nabla \psi,\nabla \phi\ra_g \, dv_g +\beta\int_{\Sigma} \psi\phi 
	\, dv_g \text{ for any }\psi,\phi\in \oH, \]
	where $\la \cdot,\cdot\ra_g$ denotes the inner product on the tangent bundle of  $\Sigma$ induced by the Riemann metric 
	$g$. 
	The orthogonal complement of $K_{\xi}$ , denoted as $K_{\xi}^{\perp}$, is  as follows: 
	\[ K_{\xi}^{\perp}=\left\{ \phi\in \oH: \langle \phi, f\rangle =0 \text{ for all } f\in K_{\xi}\right\}. \]
	We also introduce $\Pi_{\xi}: \oH\rightarrow K_{\xi}$ and $\Pi_{\xi}^{\perp}: \oH\rightarrow K_{\xi}^{\perp}$ as the orthogonal projections onto $K_{\xi}$ and $K_{\xi}^{\perp}$, respectively.
	The solution $u$ can decompose into two parts: one part lies on the manifold  $\mathcal{M}^{k,l}_{\varepsilon}$; the other part is on $K_{\xi}^\perp$ near the orthogonal space of the tangent space of the manifold $\mathcal{M}^{k,l}_{\varepsilon}$, i.e. 
	$u= \sum_{i=1}^{k+l} PU_i+\phi^{\varepsilon}_{\xi},$
	where  $ \phi^{\varepsilon}_{\xi}$ is  the remainder term.
	
	\section{The Lyapunov-Schmidt reduction}
	Utilizing the Moser-Trudinger type inequality on compact Riemann surfaces, as in~\cite{yang2006extremal}, we have 
	\[\sup _{\int_{\Sigma}|\nabla_g u|^2 \, dv_g=1, \int_{\Sigma} u \, dv_g=0} \int e^{2\pi u^2} \, dv_g <+\infty.
	\]
	Since $\left(\int_{\Sigma}|\nabla u|^2_g \, dv_g+\beta\int_{\Sigma} |u|^2 \, dv_g\right)^{\frac 1 2}$ and $\left(\int_{\Sigma}|\nabla u|^2_g \, dv_g\right)^{\frac 1 2}$ are equivalent norms in the Hilbert space $\oH$,
	it follows that  for any $u\in \oH$
	\begin{eqnarray*}
		\log \int_{\Sigma} e^u \, dv_g &\leq& \log \int_{\Sigma}  e^{ 2\pi\frac{ u^2}{\|u\|^2} + \frac 1 {8\pi}\|u\|^2  } \, dv_g \quad \text{(by Young's Inequality)}\\
		&=& 
		\frac 1 {8\pi} \int_{\Sigma} |\nabla_g u|^2 \, dv_g  +C\leq \frac{1}{8\pi C} \la u,u\ra+C, 
	\end{eqnarray*}
	where $C>0$ is a constant. 
	Consequently, $\oH\rightarrow L^p(\Sigma), u\mapsto e^u$ is continuous.  
	For any $p>1$, let $i^*_p: L^p(\Sigma)\rightarrow \oH$  be the adjoint operator corresponding to the immersion $i: \oH \rightarrow L^{\frac{p}{p-1}}$ and $\tilde{i}^* : \cup_{p>1} L^p(\Sigma)\rightarrow \oH$. For any $f\in L^p(\Sigma)$, we define that 
	$i^*(f):=\tilde{i}^*( f-\bar{f})$, i.e. for any $h\in \oH $, 
	$\langle i^*(f), h\rangle =\int_{\Sigma}(f-\bar{f}) h \, dv_g.$
	
	The problem~\eqref{rho_2} has the following equivalent form,
	\begin{equation}
		\begin{cases}
			u=i^*(\varepsilon^2 V e^u) \\
			u\in \oH 
		\end{cases}. 
	\end{equation}
	\subsection{The linearized operator}
	We consider the linearized operator 
	\[ L_{\xi}^{\varepsilon}(\phi):= \Pi^{\perp}_{\xi}(\phi-i^*(\varepsilon^2 Ve^{\sum_{i=1}^{k+l} PU_i}\phi))\]
	for any fixed $\xi\in {M}_{\delta}$. 
	The following lemma shows that for fixed $\varepsilon$ the linearized operator is invertible in the space $K_{\xi}^{\perp}$, and the norm of the inverse operator is controlled by $|\log \varepsilon|$ as $\varepsilon\rightarrow 0$, which is a key lemma to solve the problem~\eqref{rho_2}. 
	\begin{lemma}~\label{t1} For any $\delta>0$, let $\xi=(\xi_1,\cdots,\xi_{k+l})\in M_{\delta}$. There exists $\varepsilon_{0}>0$ and a constant $c>0$ such that for any $\varepsilon \in\left(0, \varepsilon_{0}\right)$ we have 
		\[\left\|L_{\xi}^{\varepsilon}(\phi)\right\| \geqslant \frac{c}{|\log \varepsilon|}\|\phi\|, \quad \forall \phi \in K_{\xi}^{\perp} .\]
		In particular, the operator $L_{\xi}^{\varepsilon}$ is invertible and $\left\|\left(L_{\xi}^{\varepsilon}\right)^{-1}\right\| \leqslant|\log \varepsilon| / c$.
	\end{lemma}
	By~\cite{Esposito2005}, the proof of Lemma~\ref{t1} is relatively standard, which is given in Appendix~\ref{linAp}.

	For fixed $\varepsilon$ and $ \xi\in \Xi_{k,l}$,  we try to obtain the solution of  
	\[  \Pi_{\xi}^{\perp}\left( \sum_{i=1}^{k+l} PU_i + \phi^{\varepsilon}_{\xi}- i^*( \varepsilon^2 e^{\sum_{i=1}^{k+l} PU_i +\phi^{\varepsilon}_{\xi}} ) \right)=0, \]
	for $\phi^{\varepsilon}_{\xi}\in K_{\xi}^{\perp}$ applying the fixed-point theorem. Then, it is reduced to a finite-dimensional problem. 
	\begin{proposition}~\label{thm2} For any $\delta>0$,
		and $\xi=(\xi_1,\cdots,\xi_{k+l})\in M_{\delta}$. For any $p\in(1,\frac 6 5)$ there exist $\varepsilon_0>0$ and $R>0$ (uniformly in $\xi$) such that for any $\varepsilon\in (0,\varepsilon_0)$ there is a unique $\phi^{\varepsilon}_{\xi}\in K^{\perp}_{\xi}$ such that 
		\begin{equation}~\label{pi}
			\Pi_{\xi}^{\perp} \left[ \sum_{i=1}^{k+l} PU_i +\phi^{\varepsilon}_{\xi} -i^*\left( \varepsilon^2 V e^{\sum_{i=1}^{k+l} P U_i +\phi^{\varepsilon}_{\xi} }  \right)\right] =0. 
		\end{equation}
		and 
		$ \|\phi^{\varepsilon}_{\xi}\|\leq R \varepsilon^{ \frac{ 2-p}{p}}|\log \varepsilon|. $
	\end{proposition} 
	\begin{proof}
		Define operators  $T^{\varepsilon}_{\xi}$ and $M^{\varepsilon}_\xi$ on $K_{\xi}^{\perp}$ as follows: 
		\begin{equation*}
			\begin{aligned}
				&T_{\xi}^{\varepsilon}(\phi)=\left[\left(L_{\xi}^{\varepsilon}\right)^{-1} \circ \Pi_{\xi}^{\perp} \circ i^{*}\right] M_{\xi}^{\varepsilon}(\phi), \\
				&M_{\xi}^{\varepsilon}(\phi)=\varepsilon^{2} V  e^{\sum_{i=1}^{k+l} P U_{i}}\left[e^{\phi}-1-\phi\right]+\varepsilon^{2} \left[ V  e^{\sum_{i=1}^{k+l} P U_{i}}- \sum_{i=1}^{k+l}\chi_{i} e^{-{\varphi}_{\xi_i}} e^{U_{i}}\right].
			\end{aligned}
		\end{equation*}
		Since $i^*(\varepsilon^2\sum_{i=1}^{k+l}\chi_{i} e^{-{\varphi}_{\xi_i}} e^{U_i})= \sum_{i=1}^{k+l} PU_i$, it follows that 
		$\phi$ is a fixed point of $T^{\varepsilon}_{\xi}$ if and only if 
		$\phi$ solves~\eqref{pi} on $K^{\perp}_{\xi}$. \\
		{\it Claim. There exist $\varepsilon_0>0$ and $R > 0$ such that $T^{\varepsilon}_{\xi}$ is a contraction map for any $\varepsilon \in (0, \varepsilon_0)$ and $|\phi| \leq R \varepsilon^{\frac{2-p}{p}}|\log \varepsilon|$.} \\
		Applying Lemma~\ref{t1}, Lemma~\ref{lem5}, Lemma~\ref{lem6}, and the Moser-Trudinger inequality, we obtain 
		\begin{eqnarray*} 
			\|T^{\varepsilon}_{\xi}(\phi)\|&\leq&  C|\log \varepsilon| \| i^*\circ M^{\varepsilon}_{\xi}(\phi)\|
			\leq  C|\log \varepsilon| \left|  M^{\varepsilon}_{\xi}(\phi)\right|_{L^p(\Sigma)}\\
			&\leq & C |\log \varepsilon|\left(  
			\left| \varepsilon^2 Ve^{\sum_{i=1}^{k+l} PU_i} ( e^{\phi}-1-\phi)\right|_{L^p(\Sigma)}\right.\\
			&&+\left.  \left| \varepsilon^2 Ve^{\sum_{i=1}^{k+l} PU_i} -\varepsilon^2\sum_{i=1}^{k+l}\chi_{i} e^{U_i}\right|_{L^p(\Sigma)}\right)\\
			&\leq &C |\log \varepsilon|\left(\|\phi\|^2 e^{c_2\|\phi\|^2 } \varepsilon^{\frac{2-2pr}{pr}}  + \varepsilon^{\frac{2-p}{p}}\right) ,
		\end{eqnarray*}
		where $c_2 > 0$ is a constant, $r > 1$ is sufficiently close to 1, and $p \in (1, \frac 65)$. We then fix arbitrary $p \in (1, \frac 65 )$ and choose $R > 0$ large enough such that $C(1 + e^{c_2}) \leq R$. Next, we select $\varepsilon_1 > 0$ such that $\max\{R\varepsilon^{\frac{2-2pr}{pr} + \frac{2-p}{p}}| \log \varepsilon|, R \varepsilon^{\frac{2-p}{p}} |\log \varepsilon|\} \leq 1$ for all $\varepsilon \in (0, \varepsilon_1)$. Consequently, for any $|\phi| \leq R \varepsilon^{\frac{2-p}{p}} |\log \varepsilon|$, we have $|T^{\varepsilon}_{\xi}| \leq R \varepsilon^{\frac{2-p}{p}}|\log \varepsilon|$ for all $\varepsilon\in (0,\varepsilon_{1})$.
		And similarly,  by Lemma~\ref{lem6} we deduce that 
		\begin{eqnarray*}
			\|T^{\varepsilon}_{\xi}(\phi_1)-T^{\varepsilon}_{\xi}(\phi_2) \|&\leq & C'|\log \varepsilon| \left| \varepsilon^2 V  e^{\sum_{i=1}^{k+l}  PU_i }( e^{\phi_1} -e^{\phi_2}-(\phi_1-\phi_2)) \right|_{L^p(\Sigma)}\\
			&\leq & C' |\log \varepsilon|e^{ c_2( \sum_{j=1}^2 \|\phi_j\|^2 )} \varepsilon ^{\frac{2-2pr}{pr}}( \sum_{j=1}^2 \|\phi_j\| )) \|\phi_1-\phi_2\|\\
			&\leq&  2RC^\prime e^{2c_2}  \varepsilon^{\frac{2-2pr}{pr}+ \frac{1+\alpha_0-p}{p}} \log ^2\varepsilon \|\phi_1-\phi_2\|\leq  \frac{1}{2}\|\phi_1-\phi_2\|,  
		\end{eqnarray*}
		uniformly for all $\varepsilon\in (0,\varepsilon_2)$ and $\xi\in M_{\delta}$, where $\varepsilon_2>0$ is chosen such that \[ \max\{  R \varepsilon^{\frac{2-p}{p}} |\log \varepsilon|,  2RC^{\prime} e^{c_2} \varepsilon^{\frac{2-2pr}{pr}+\frac{2-p}{p}} \log^2 \varepsilon\}<\frac{1}{2},\] for any $  \varepsilon \in (0,\varepsilon_2).$  Then define $\varepsilon_0=\min\{ \varepsilon_1,\varepsilon_2\}$.
		Thus $T^{\varepsilon}_{\xi}(\phi)$ is a contraction map on $\{\phi\in K_{\xi}^{\perp}: \|\phi\|\leq R \varepsilon^{\frac{2-p}{p}} |\log \varepsilon| \}. $
		By the contracting-mapping principle, there exists a unique fixed point of $T^{\varepsilon}_{\xi}$ on $ \{\phi\in K_{\xi}^{\perp}: \|\phi\|\leq R \varepsilon^{\frac{2-p}{p}} |\log \varepsilon| \}$. 
	\end{proof}

	
	\section{ The reduced functional and its expansion}
	The associated functional $E_{\varepsilon}(u)$ of the problem~\eqref{rho_2} is defined as following:
	\begin{equation}
		E_{\varepsilon}(u)= \frac{1}{2} \int_{\Sigma} (|\nabla u|_g^2 +\beta |u|^2)\, dv_g -\varepsilon^2 \int_{\Sigma}V e^u \, dv_g.
	\end{equation}
	Assume $u$ has the form $\sum_{i=1}^{k+l} PU_i+\phi^{\varepsilon}_{\xi}$, where $\phi^{\varepsilon}_{\xi}$ is obtained by Proposition~\ref{thm2}. Then, the reduced functional is defined by  $\tilde{E}_{\varepsilon}(\xi):= E_{\varepsilon}( \sum_{i=1}^{k+l} PU_i +\phi^{\varepsilon}_{\xi})$ with $\|\phi^{\varepsilon}_{\xi}\|\leq R \varepsilon^{\frac{2-p } {p}}|\log \varepsilon|$, i.e. 
	\begin{eqnarray}\label{eq:def_reduced_tilde_E}
		\widetilde{E}_{\varepsilon}(\xi):&=& \frac{1}{2} \int_{\Sigma} \left(\left|\nabla \Big( \sum_{i=1}^{k+l} PU_i +\phi^{\varepsilon}_{\xi}\Big)\right|_g^2+\beta\left |\sum_{i=1}^{k+l} \Big(PU_i +\phi^{\varepsilon}_{\xi}\Big)\right|^2\right)  \, dv_g \\
		&&-\varepsilon^2 \int_{\Sigma} Ve^{  \sum_{i=1}^{k+l} PU_i +\phi^{\varepsilon}_{\xi} }\, dv_g.
	\end{eqnarray}
	The reduced functional  $\tilde{E}_{\varepsilon}$ has a $C^1$-expansion with respect to $\xi$ as stated in the following  proposition: \par 
	\begin{proposition}~\label{thm3} As $\varepsilon\rightarrow 0,$
		$$
		\widetilde{E}_{\varepsilon}(\xi) = 4\pi m (3\log 2 - 2) -8\pi m \log \varepsilon -\frac{1}{2}\cF^V_{k,l}(\xi) +o(1),
		$$
		$C^{1}$-uniformly convergent in any compact sets of $\Xi^\prime_{k,l}$, where $m=2k+l$.
	\end{proposition}
	
	\begin{proof}
		Denote $\phi=\phi^{\varepsilon}_{\xi}$ to simplify the notation. 
		Then 
		\begin{eqnarray*}
			\widetilde{E}_{\varepsilon}(\xi)
			&=& \frac{1}{2}\left(\sum_{i=1}^{k+l} \la PU_i,PU_i\ra + \sum_{i\neq j} \la PU_i,PU_j\ra \right) +  \frac{1}{2}\left(\|\phi\|^2+2 \sum_{i=1}^{k+l}  \la PU_i,\phi\ra  \right)\\
			&&-  \int_{\Sigma} \varepsilon^2 Ve^{\sum_{i=1}^{k+l} PU_i  } \, dv_g 
			-\int_{\Sigma} \varepsilon^2 (Ve^{\sum_{i=1}^{k+l} PU_i +\phi }-Ve^{\sum_{i=1}^{k+l} PU_i }) \, dv_g . 
		\end{eqnarray*}
		We notice that $|e^s -1|\leq e^{|s|}|s| (\forall s\in \mathbb{R}). $
		By Lemma \ref{lemb2}, we obtain that 
		\begin{eqnarray*}
			&&\left|\int_{\Sigma} \varepsilon^2 (Ve^{\sum_{i=1}^{k+l} PU_i +\phi }-Ve^{\sum_{i=1}^{k+l} PU_i }) \, dv_g \right|\leq \left| \int_{\Sigma} \varepsilon^2 Ve^{\sum_{i=1}^{k+l} PU_i } e^{|\phi|}|\phi| \, dv_g \right|  \\
			&\leq& \mathcal{O}\left( \varepsilon^2 \left(\int_{\Sigma} e^{r \sum_{i=1}^{k+l} PU_i} \, dv_g \right)^{1/r}\left|e^{|\phi|} \right|_{L^s(\Sigma)} \left|\phi\right|_{L^t(\Sigma)}\right)\\
			&\leq & \mathcal{O}\left( \left|\varepsilon^2  Ve^{\sum_{i=1}^{k+l} PU_i} \right|_{L^r(\Sigma)} \|\phi\|\right) =o(1) ,
		\end{eqnarray*}
		where $r\in (1,2)$ with $\frac{1}{s}+\frac{1}{r}+\frac{1}{t}=1$ and $ \frac{2(1-r)}{r}+\frac{2-p}{p}>0$. 
		By Lemma \ref{lem7} and Lemma \ref{lem8}, as $\varepsilon\rightarrow 0$, $
		\widetilde{E}_{\varepsilon}(\xi) = \sum_{i=1}^{k+l} \varrho(\xi_i) (3\log 2 - 2\log \varepsilon) -2\sum_{ i=1}^{k+l} \varrho(\xi_i) -\frac{1}{2}\cF^V_{k,l}(\xi) +o(1).$ 
		By \eqref{pi},  it holds 
		\begin{equation}~\label{finidim}
			\sum_{i=1}^{k+l} PU_i +\phi -i^*\left( \varepsilon^2 Ve^{\sum_{i=1}^{k+l} PU_i +\phi} \right)
			= \sum_{s=1}^{k+l}\sum_{t=1}^{\ii(\xi_s)}c^{\varepsilon}_{st} P\Psi^t_{s},
		\end{equation}
		where $c^{\varepsilon}_{st}$ are coefficients.  
		Combining~\eqref{finidim} with Lemma~\ref{lem11}, we deduce that 
		\begin{equation}~\label{cc}
			\sum_{i=1}^{k+l}\sum_{j=1}^{\ii(\xi_i)} \left|c^{\varepsilon}_{ij} \right|=\mathcal{ O}(\varepsilon^2),
		\end{equation}
		via Lemma \ref{lemb1} and Remark \ref{rklem4}.  
		For the $C^{1}$-expansion,    Lemma \ref{lemb3} and Lemma \ref{lem11} imply that  
		\begin{eqnarray*}	 & &\partial_{\left(\xi_{h}\right)_{j}}E_{\varepsilon}\left(\sum_{i=1}^{k+l} P U_{i}+\phi\right) \\
			&=&\left \la\sum_{i=1}^{k+l} P U_{i}+\phi-i^{*}\left(\varepsilon^{2} V e^{\sum_{i=1}^{k+l} P U_{i}+\phi}\right), 	\partial_{\left(\xi_{h}\right)_{j}}PU_h+\sum_{i=1}^{k+l}  P \Psi_{i}^{0}\partial_{\left(\xi_{h}\right)_{j}} \tau_{i}(\xi)+\partial_{\left(\xi_{h}\right)_{j}} \phi\right\ra  \\
			& =&-\frac 1 2  \frac{\partial \mathcal{F}}{\partial\left(\xi_{h}\right)_{j}}
			\left( \xi_{1}, \cdots, \xi_{k+l}\right)+	\left\la \sum_{s=1}^{k+l} \sum_{t=1}^{\ii(\xi_s)} c_{st}^{\varepsilon}{ P\Psi^t_s}
			,  \sum_{i=1}^{k+l}  P \Psi_{i}^{0}\partial_{\left(\xi_{h}\right)_{j}} \tau_{i}(\xi)+\partial_{\left(\xi_{h}\right)_{j}} \phi\right\ra \nonumber +o(1)\\
			&=&-\frac 1 2  \frac{\partial \mathcal{F}}{\partial\left(\xi_{h}\right)_{j}}\left(\xi_{1}, \cdots, \xi_{k+l}\right)+\sum_{s=1}^{k+l} \sum_{t=1}^{\ii(\xi_s)} c_{s t}^{\varepsilon}\left\la {P \Psi_{s}^{t}}, \sum_{i=1}^{k+l}  P \Psi_{i}^{0}\partial_{\left(\xi_{h}\right)_{j}} \tau_{i}(\xi)+\partial_{\left(\xi_{h}\right)_{j}} \phi\right\ra +o(1),
		\end{eqnarray*}
		for any $h=1, \cdots,  k+l$ and  $j=1,\cdots,\ii(\xi_h)$. 
		Utilizing Lemma~\ref{lem4}, we have 
		$$
		|	\la P \Psi_{s}^{t}, P \Psi_{i}^{0}\ra |\leq \|P \Psi_{s}^{t}\|  \| P \Psi_{i}^{0}\|=O\left(\frac{1}{\varepsilon}\right).
		$$
		Taking into account that 
		$\la P\Psi^t_s,\phi\ra=0$ and $|\partial_{(\xi_h)_j}P\Psi^t_s|\leq |\partial_{(\xi_h)_j}\Psi^t_s|=\mathcal{ O}(\frac 1 {\varepsilon^2})$,   we obtain 
		\begin{equation*}
			\begin{array}{lll}
				\la P \Psi_{s}^{t}, \partial_{\left(\xi_{h}\right)_{j}} \phi \ra&=&\partial_{(\xi_h)_j}\la P \Psi_{s}^{t},\phi \ra- \la \partial_{(\xi_h)_j}  P \Psi_{s}^{t} ,\phi\ra\\
				&\leq&O\left(\|\phi\|\left\|\partial_{\left(\xi_{h}\right)_{j}} P \Psi_{s}^{t}\right\|\right)=O\left(\frac{\|\phi\|}{\varepsilon^{2}}\right)=o\left(\frac{1}{\varepsilon^{2}}\right).
			\end{array}
		\end{equation*}
		Consequently, we have
		\begin{equation}~~\label{occ}
			\left\la \sum_{s=1}^{k+l} \sum_{t=1}^{\ii(\xi_s)} c_{st}^{\varepsilon} P\Psi^t_s
			,  \sum_{i=1}^{k+l} \partial_{\left(\xi_{h}\right)_{j}} \tau_{i}(\xi) P \Psi_{i}^{0}+\partial_{\left(\xi_{h}\right)_{j}} \phi\right\ra
			=o\left( \frac{1}{\varepsilon^2 } \sum_{s=1}^{k+l} \sum_{t=1}^{\ii(\xi_s)} |c^{\varepsilon}_{st}|\right). 
		\end{equation}
		It follows that 
		\begin{equation}~\label{eq:partial_1_c}
			\partial_{(\xi_h)_j} \tilde{E}_{\varepsilon}(\xi) = -\frac 1 2 \frac{\partial \cF^V_{k,l}}{\partial\left(\xi_{h}\right)_{j}}\left(\xi_{1}, \cdots, \xi_{k+l}\right)
			+ o\left(\frac{1}{\varepsilon^2} \sum_{s=1}^{k+l} \sum_{t=1}^{\ii(\xi_s)} |c^{\varepsilon}_{st
			}|\right). 
		\end{equation}
		Then, \eqref{cc} and  \eqref{eq:partial_1_c}  imply that for any $h=1,\cdots, k+l $ and $j=1,\cdots, \ii(\xi_h)$
		\[	\partial_{\left(\xi_{h}\right)_{j}}E_{\varepsilon}\left( \sum_{i=1}^{k+l} PU_i +\phi \right) = -\frac 1 2 \frac{\partial \mathcal{F}}{\partial\left(\xi_{h}\right)_{j}}\left(\xi_{1}, \cdots, \xi_{k+l}\right)+o(1), \]
		as $\varepsilon\rightarrow 0.$
	\end{proof}
	On the other hand,  $\sum_{i=1}^{k+l} PU_i+ \phi^{\varepsilon}_{\xi}$ is a critical point of $E_{\varepsilon}(u)$ in $\oH$, which is equivalent to $\xi$ being a critical point of $\tilde{E}_{\varepsilon}(\xi)$ in $\Xi^\prime_{k,l}$.  
	
	\begin{proposition}~\label{thm4} There exists $\varepsilon_0>0$ such that for any fixed $\varepsilon\in (0,\varepsilon_{0})$,
		the function $\sum_{i=1}^{k+l} P U_{\tau_{i}(\xi), \xi_{i}}+\phi_{\xi}^{\varepsilon}$ is a solution of~\eqref{rho_2} for some $\xi\in \Xi^\prime_{k,l}$ if and only if $\xi$ is a critical point of the reduced map $$\tilde{E}_{\varepsilon}: M_{\delta}\rightarrow \R^2, \xi \mapsto \tilde{E}_{\varepsilon}(\xi)= E_{\varepsilon}\left(\sum_{i=1}^{k+l} P U_{\tau_{i}(\xi), \xi_{i}}+\phi_{\xi}^{\varepsilon}\right),$$
		for some $\tau>0.$
	\end{proposition}
	\begin{proof}
		Denote  $\phi:=\phi^{\varepsilon}_{\xi}$ to simplify the notations. 
		Assume that $\xi$ is a critical point of the reduced map $\tilde{E}_{\varepsilon}(\xi)$. Then $\xi$ satisfies 
		\begin{eqnarray}~\label{r1}
			\partial_{(\xi_i)_j}\tilde{E}_{\varepsilon}(\xi)=0,
		\end{eqnarray}
		for any $i=1, \cdots, k+l$ and $j=1,\cdots,\ii(\xi_i)$. 
		
		By \eqref{pi} of Proposition~\ref{thm2}, 
		$\sum_{h=1}^{k+l} PU_h +\phi -i^*\left( \varepsilon^2 Ve^{\sum_{h=1}^{k+l} PU_h +\phi} \right)
		= \sum_{s=1}^{k+l}\sum_{t=1}^{\ii(\xi_s)} c^{\varepsilon}_{st} P\Psi_s^{t}, $
		where $c^{\varepsilon}_{st}$ are coefficients. 
		Then ,
		\begin{equation}~\label{oc}
			\left\la   \sum_{s=1}^{k+l}\sum_{t=1}^{\ii(\xi_s)} c^{\varepsilon}_{st} P\Psi_s^{t}   ,\partial_{(\xi_i)_j}PU_i+\sum_{h=1}^{k+l}  P \Psi_{i}^{0}\partial_{\left(\xi_{i}\right)_{j}} \tau_{h}(\xi)+\partial_{\left(\xi_{i}\right)_{j}} \phi\right\ra =0.
		\end{equation}
		Applying~\eqref{occ} and  and~\eqref{oc}, we derive that 
		$$
		\sum_{s=1}^{k+l} \sum_{t=1}^{\ii(\xi_s)} c_{st}^{\varepsilon}\left\la P \Psi_{s}^{t},\partial_{(\xi_i)_j}PU_i\right\ra=o\left(\frac{1}{\varepsilon^{2}} \sum_{s=1}^{k+l} \sum_{t=1}^{\ii(\xi_s)}\left|c_{st}^{\varepsilon}\right|\right).
		$$
		By Remark \ref{rklem4}, we conclude that $c_{ij}^{\varepsilon}=0$ for any $i=1, \cdots, k+l$ and $j=1,\cdots,\ii(\xi_i)$.
		Thus 
		\begin{eqnarray}~\label{r2}
			\sum_{h=1}^{k+l} P U_{h}+\phi-i^{*}\left(\varepsilon^{2} e^{\sum_{h=1}^{k+l} P U_{h}+\phi}\right)=0. 
		\end{eqnarray}

		Conversely, suppose  $\sum_{h=1}^{k+l} PU_h +\phi^{\varepsilon}_{\xi}$ is a weak solution to~\eqref{rho_2} in $\oH$  for $\xi\in \Xi^\prime_{k,l}$. Then, there exists $\delta>0$ sufficiently small such that  $\xi\in \mathrm{\Xi_{k,l}^{\prime}}$ and  \eqref{r2} is verified. Hence, \eqref{r1} holds true, leading to the conclusion that $\xi$ is a critical point of the reduced function $\tilde{E}_{\varepsilon}(\xi)$. 
	\end{proof}
	\section{Proof of the Main Result}
	Now, we are ready to prove the main results.\\ 
	\begin{altproof}{Theorem~\ref{main_thm}}	Let $K$ be a stable critical point set of $\cF^{V}_{k,l}$.  As  $\varepsilon\rightarrow 0$  there exists  a sequence of points $\xi^{\varepsilon}=\left(\xi^\varepsilon_{1}, \cdots, \xi^\varepsilon_{k+l}\right) \in \Xi_{k,l}$ such that $d_g(\xi^{\varepsilon},K)\rightarrow 0$ and  $\xi_{\varepsilon}$ is a critical point of $\tilde{E}_{\varepsilon}: \Xi^\prime_{k,l}\rightarrow \mathbb{R}$. Assume that up to a subsequence 
		\[ \xi^{\varepsilon}=\left(\xi^\varepsilon_{1}, \cdots, \xi^\varepsilon_{m }\right)\rightarrow \xi=(\xi_1,\cdots,\xi_{k+l})\in K, \]
		as $\varepsilon\rightarrow 0.$
		Define 
		$u_{\varepsilon}=$ $\sum_{i=1}^{k+l} P U_{\tau_{i}\left(\xi^{\varepsilon}\right), \xi^{\varepsilon}_{i}}+\phi_{\xi^{\varepsilon}}^{\varepsilon}$. According to Proposition~\ref{thm4}, 
		$u_{\varepsilon}$ solves \eqref{rho_2} as $\varepsilon\rightarrow 0$, which means that  $u_{\varepsilon}$ solves problem~\eqref{eq:main_eq} in the weak sense for some $\lambda:=\lambda_\varepsilon= \varepsilon^2\int_{\Sigma} V e^{u_\varepsilon} \, dv_g$. 
		Applying Lemma~\ref{lemb2}, Lemma~\ref{lem6} and Lemma~\ref{lem8}, 
		$\lambda= 4\pi m +o(1), \text{ as } \varepsilon\rightarrow 0. $\\
		{\it Claim. For any $\Psi \in C(\Sigma),$
			$
			\varepsilon^{2} \int_{\Sigma}  V e^{u_{\varepsilon}} \Psi \, dv_g  \rightarrow  \sum_{i=1}^{k+l} \varrho(\xi_i) \Psi\left(\xi_{i}\right), \text { as } \varepsilon\rightarrow 0.
			$}
		In fact, by the inequality $\left|e^{s}-1\right| \leqslant e^{|s|}|s|$ for any $s \in \mathbb{R}$ and   Lemma~\ref{lem5}, we have 
		\begin{eqnarray*}
			\varepsilon^{2} \int_{\Sigma} V e^{u_{\varepsilon}} \Psi \, dv_g& =&\varepsilon^{2} \int_{\Sigma} V e^{\sum_{i=1}^{k+l} P U_{i}} \Psi \, dv_g+o(1)=
			\sum_{ i=1}^{k+l} \int_{\Sigma}  \varepsilon^2\chi_{\xi_{i}} e^{U_i}\Psi \, dv_g +o(1) \\
			&=& \sum_{i=1}^{k+l} \varrho(\xi_i) \Psi\left(\xi_{i}\right)+o(1),
		\end{eqnarray*} 
		as $\varepsilon\rightarrow 0.$
		Therefore, $u_\varepsilon$ is a family of blow-up solutions of \eqref{eq:main_eq} as $\varepsilon\rightarrow 0$. The proof is concluded. 
	\end{altproof}
	
	\begin{altproof}{Corollary~\ref{cor:mf}}
		The set of global minimum points $\cK_{k,l}$ is a $C^1$-stable critical point set of $\cF^V_{k,l}$.  There exists $\delta>0$ sufficiently smalll such that  $\cK_{k,l}\subset\subset M_{\delta}$ given by~\eqref{eq:def_M_delta}. 
		As demonstrated in the proof of Theorem~\ref{main_thm}, for any $\varepsilon>0$ sufficiently small we can construct $\xi^{\varepsilon} \in M_{\delta}$ and $\lambda_\varepsilon$ such that up to a subsequence
		$ \xi^{\varepsilon} \rightarrow \xi\in \cK_{k,l},$  $ \lambda_\varepsilon\rightarrow 4\pi m=4\pi m,$ 
		and $u_{\varepsilon}=\sum_{i=1}^{k+l} PU_{\tau_i(\xi^\varepsilon),\xi^\varepsilon_{i}}+\phi^{\varepsilon}_{\xi^{\varepsilon}}$ solving~\eqref{eq:main_eq} for the parameter $\lambda_\varepsilon$.
		It follows that 
		$$
		\cF^{V}_{k,l}(\xi^{\varepsilon}) \rightarrow \cF^{V}_{k,l}(\xi)=\min_{\xi \in \Xi_{k,l}} \cF^{V}_{k,l}(\xi), \quad \text{as }\varepsilon\rightarrow 0.
		$$
		We recall the following expansion from Proposition~\ref{thm3},
		\begin{equation*}
			\widetilde{E}_{\varepsilon}(\xi) = \sum_{i=1}^{k+l} \varrho(\xi_i) (3\log 2 - 2\log \varepsilon) -2\sum_{i=1}^{k+l} \varrho(\xi_i) -\frac{1}{2}\cF^{V}_{k,l}(\xi) +o(1)
		\end{equation*} 
		in $C^1$-sense. 
		As $\varepsilon\rightarrow 0$, $u_{\varepsilon}$  is uniformly bounded on $\Sigma\setminus \cup_{i=1}^{k+l} U_{\epsilon}(\xi_{i}^\varepsilon)$ 
		for any $\epsilon>0$ and  $\sup_{U_{\epsilon}(\xi_{i})} u_{\varepsilon}\rightarrow +\infty,$  as $\varepsilon\rightarrow 0.$
		Lemma~\ref{lemb1} implies that 
		\[ u_{\varepsilon}= -2\sum_{i=1}^{k+l} \chi\left(|y_{\xi^\varepsilon_{i}}(x)|/r_0\right) \log (\varepsilon^2\tau^2_i(\xi^\varepsilon)+ |y_{\xi^\varepsilon_{i}}(x)|^2)+ \mathcal{O}(1), \text{ as } \varepsilon\rightarrow 0.\]
		There exists a constant $C>0$ independent with $\varepsilon$ such that  around $\xi_i$
		\[ u_{\varepsilon}\leq C+2\log \frac{1}{\varepsilon \tau_i(\xi^\varepsilon) }, \text{ for any }|y_{\xi^\varepsilon_i}(x)|\geq \sqrt{\varepsilon \tau_i(\xi^\varepsilon)} \]
		While for any $ |y_{\xi^\varepsilon_i}(x)|< \varepsilon ^2\tau_i(\xi^\varepsilon)$, 
		we have 
		$ u_{\varepsilon}\geq -C + 4\log\frac 1 {\varepsilon \tau_i(\xi^\varepsilon)}. $
		It follows that for all sufficiently small $\varepsilon>0$, we have   for $i=1,\cdots, k+l$
		$$ \max_{U_{r_0}(\xi^\varepsilon_{i})}u_{\varepsilon}=\max\left\{u_{\varepsilon}(x): |y_{\xi^\varepsilon_{i}}(x)|< (\varepsilon \tau_i(\xi^\varepsilon))^{\frac 1 2} \right\}.  
		$$
		Then, there exists $\tilde{\xi}_{\varepsilon
			,i}$ satisfying that $ |y_{\xi_{\varepsilon
				,i}}(\tilde{\xi}_{\varepsilon 
			,i})|< \sqrt{\varepsilon \tau_i(\xi^\varepsilon)}$ attaining the local maximum of $u_{\varepsilon}$ for any $i=1,\cdots, k+l $. 
		Moreover, $\tilde{\xi}_{\varepsilon
		}:=(\tilde{\xi}_{\varepsilon
			,1},\cdots,\tilde{\xi}_{\varepsilon,m})\rightarrow \xi$ and 
		\[ \cF^{V}_{k,l}(\tilde{\xi}_{\varepsilon})\rightarrow \min_{\xi\in \Xi_{k,l}^\prime} \cF^{V}_{k,l}(\xi), \text{ as }\varepsilon\rightarrow 0. \] 
		Applying Theorem~\ref{main_thm}, we can conclude the proof.
	\end{altproof}

	\begin{appendices}
		\section{Regularity theory for Neumann boundary conditions}
		
		\begin{lemma}\label{lem:Lp_est}
			Let $(\Sigma)$ be a compact Riemann surface with smooth boundary $\partial \Sigma$. For any $\beta\geq 0$, if $f \in L^2(\Sigma, g)$ satisfies 	$$
			\int_{\Sigma} f =0,
			$$ then there exists a unique weak solution of
			\begin{equation}
				\label{eq:poission}	\left\{\begin{array}{lll}
					-\Delta_g u+\beta u=f & \text { in } &\Sigma \\
					\partial_{\nu_g} u =0& \text { on } &\partial \Sigma \\
					\int_{\Sigma} u \, dv_g=0 & & 
				\end{array}\right.,
			\end{equation}
			i.e. there exists  a unique $ u \in \overline{H}^1(\Sigma)$ satisfying  
			$$
			\int_{\Sigma} \lan\nabla u, \nabla\varphi \ran_g \, dv_g+\beta\int_{\Sigma} u\varphi \, dv_g=\int_{\Sigma} f \varphi \, dv_g+ \int_{\partial\Sigma} h \varphi ds_g, \quad \forall \varphi \in  H^1(\Sigma).
			$$
			Moreover, for any $p>1$ if $f\in L^p(\Sigma)$, there exists a $u\in W_0^{2,p}(\Sigma):= W^{2,p}(\Sigma)\cap \{ u: \int_{\Sigma} u \, dv_g =0\}$ solving~\eqref{eq:poission}	with the following $W^{2,p}$-estimate: 
			$$
			\|u\|_{W^{2,p}(\Sigma)} \leq C|f|_{L^p(\Sigma)} .
			$$
		\end{lemma}
		For the Poisson equation with homogeneous  Neumann boundary condition, the $L^p$-estimate was proven in \cite[Lemma 5]{yang2021125440}. And we can deduce \eqref{eq:poission} by the same approach. 
		\begin{proof}
			For the uniqueness, we assume that $u_1, u_2$ are two weak solutions of  \eqref{eq:poission}  in $\oH$. It follows that 
			\[	\int_{\Sigma} \lan\nabla (u_1-u_2), \nabla\varphi \ran_g \, dv_g+\beta\int_{\Sigma} (u_1-u_2)\varphi \, dv_g=0,  \]
			for any $\varphi\in H^1(\Sigma)$. Then, 
			$u_1=u_2$ up to the addition of a constant. Observing that $\int_{\Sigma} u_1 \, dv_g =\int_{\Sigma} u_2 \, dv_g =0$, we deduce that $u_1\equiv u_2.$
			
			We will prove the existence of solutions using  variational methods. 
			Consider the energy functional 
			\[ J(u)= \frac 1 2 \int_{\Sigma} ( |\nabla u|^2_g +\beta u^2) \, dv_g -\int_{\Sigma} f u \, dv_g. \]
			Applying the H\"{o}lder inequality and the Poincar\'{e} inequality, we deduce that 
			\begin{eqnarray*}
				\left| \int_{\Sigma} f u \, dv_g\right|\leq  |f|_{L^2(\Sigma)}|u|_{L^2(\Sigma)}\leq 
				\|f|_{L^2(\Sigma)}|\nabla u|_{L^2(\Sigma)},
			\end{eqnarray*}
			which yields that $J$ has a lower bound in $\oH$. Let $u_n$ be a sequence in $\oH$ such that $J$ attains the minimum value, i.e.
			\[ \lim _{n\rightarrow +\infty} J(u_n)= \inf _{u\in \oH} J(u). \]
			For any $n\in \N_+$, 
			$
			J(u_n) \geq \frac 1 2 \|u_n\|^2 - C |f|_{L^2(\Sigma)}\|u_n\|. 
			$
			Given that $ \inf _{u\in \oH} J(u)\leq J(0)=0$, $u_n$ is uniformly bounded in $\oH$. 
			Up to a subsequence, we assume that  $u_n$ converges to some $u_0\in \oH$ weakly. By the Rellich–Kondrachov theorem, 
			$u_n\rightarrow u_0$ strongly in $L^q(\Sigma)$ for any $q>1$ and almost everywhere. Fatou's lemma implies that 
			\[ J(u_0) \leq \liminf _{n\rightarrow +\infty} J(u_n)= \inf_{u\in \oH} J(u). \] 
			Thus, $u_0$ is a minimizer of $J(u)$ on $\oH$. 
			
			Next, we consider the $W^{2,p}$-estimates of the solutions. 
			Employing the isothermal coordinates introduced in Section~\ref{prelim} it is sufficient to prove the $L^p$-regularity locally in an open disk or half-disk in $\R^2$. Specifically, in the case of a half-disk, we can extend the problem by the reflection of the $x$-axis to a full open disk, considering that $\partial_{\nu_g} u=0$ on the boundary. This extension allows for the application of the standard local $L^p$-theory, thereby we can establish the $L^p$-regularity for the Neumann boundary problem~\eqref{eq:poission} on a compact Riemann surface $\Sigma$.
		\end{proof}
		Let $W^{s,p}_{\partial}(\Sigma):= \{ h|_{\partial\Sigma}: h\in W^{s,p}(\Sigma)\}$ equipped with the norm $$\|h\|_{W^{s,p}_{\partial}(\Sigma)}:= \inf \left\{ \|\psi\|_{W^{s,p}(\Sigma)}: \psi\in W^{s,p}(\Sigma) \text{ with } \psi|_{\partial\Sigma}=h\right\},$$
		for any $s\in \N$ and $p\in (1, +\infty)$. 
		For the inhomogeneous boudnary condition, we have the following $L^p$-theory:
		\begin{lemma}[Theorem 3.2 of~\cite{Wehrheim2004}]\label{lem:inh_LP}
			Suppose that $f\in L^p(\Sigma)$ and $h\in W^{1,p}_{\partial}(\Sigma)$. Let $u$ be a weak solution with $\int_{\Sigma} u \, dv_g =0$ of 
			\begin{equation*}
				\left\{\begin{array}{lll}
					-\Delta_g u+\beta u=f & \text { in } &\Sigma \\
					\partial_{\nu_g} u =h& \text { on } &\partial \Sigma \\
					\int_{\Sigma} u \, dv_g=0 & & 
				\end{array}\right..
			\end{equation*}
			Then, $u\in W^{2,p}(\Sigma)$ with the estimate 
			\[ \|u\|_{W^{2,p}(\Sigma)}\leq C\left( |f|_{L^p(\Sigma)}+ \|h\|_{W^{1,p}_{\partial}(\Sigma)}\right). \]
		\end{lemma} 
		For the case $\beta = 0$, we refer to~\cite{Agmon1959} and~\cite{Wehrheim2004}. By the same approach, Lemma~\ref{lem:inh_LP} can be proven for $\beta > 0$; hence, we omit the details.
		
		Next, we consider the Schauder estimates for the Neumann boundary condition on compact Riemann surfaces. 
		\begin{lemma}
			\label{lem:schauder}
			For any given $\alpha\in(0,1),\beta\geq 0$,  let $(\Sigma,g)$ be a compact Riemann surface with boundary in $C^{2, \alpha}$-class and let $f \in C^{ \alpha}(\Sigma), h \in C^{1, \alpha}(\Sigma)$ such that:
			\begin{eqnarray}
				\label{eq:f_h} 	\int_{\Sigma} f  \, dv_g=\int_{\partial \Sigma} h  ds_g.
			\end{eqnarray}
			Then, there exists a unique solution to the problem
			\begin{equation}\label{eq:schauder_c2_alpha}
				\left\{\begin{array}{ll}
					-\Delta_g u+\beta u=f & \text {in } \Sigma \\ 
					\partial_{ \nu_g} u=h& \text { on } \partial \Sigma\\
				\end{array}\right.
			\end{equation}
			in the space 
			$ C_0^{2, \alpha}(\Sigma):= C^{2,\alpha}(\Sigma)\cap \{ u : \int_{\Sigma} u \, dv_g =0\}.
			$
			Moreover, it has the following Schauder estimate: 
			$
			\left\|u\right\|_{C^{2, \alpha}(\Sigma)} \leq C\left(\|f\|_{C^{ \alpha}(\Sigma)}+\|h\|_{C^{1, \alpha}(\Sigma)}\right),
			$
			where $C>0$ is a constant. 
		\end{lemma}
		
		We refer to the Schauder interior estimates for domains as in \cite{GilbargTrudinger2001}.
		
		\begin{theorem}[Corollary~6.3 of \cite{GilbargTrudinger2001}]\label{thm:inter_schauder}
			Let $\Omega$ be an open subset of $\mathbb{R}^n$ and let $u \in C^{2,\alpha}(\Omega)$ be a bounded solution in $\Omega$ of the equation
			$
			Lu = a^{ij}D_{ij}u + b^iD_iu + cu = f,
			$
			where $f \in C^{\alpha}(\Omega)$ and there are positive constants $\lambda, \Lambda$ such that the coefficients satisfy
			$
			a^{ij}\xi_i\xi_j \geq \lambda|\xi|^2, \text{ for any } x \in \Omega, \xi \in \mathbb{R}^n
			$
			and
			$
			\|a^{ij}\|_{C^0(\Omega)} + \|b^i\|_{C^0(\Omega)} + \|c\|_{C^0(\Omega)} \leq \Lambda.
			$
			Then we have the interior estimate: for any $\Omega^{\prime} \subset\subset \Omega$, 
			\begin{equation}
				\|u\|_{C^{2,\alpha}(\Omega^{\prime})} \leq C(\|u\|_{C^0(\Omega)} + \|f\|_{C^{\alpha}(\Omega)})
				\label{6.14}
			\end{equation}
			where $C = C(n, \Omega^{\prime}, \alpha, \lambda, \Lambda)$ is a constant.
		\end{theorem}
		The Schauder estimate with oblique derivative boundary conditions is as follows:
		\begin{theorem}[Lemma 6.29 of \cite{GilbargTrudinger2001}]\label{thm:b_schauder}
			Let $\Omega$ be a bounded open set in $\mathbb{R}^n_+$ with a boundary portion $T$ on $x_n = 0$. Suppose that $u \in C^{2,\alpha}(\Omega \cup T)$ is a solution in $\Omega$ of $Lu = f$ (as in~Theorem~\ref{thm:inter_schauder}) satisfying the boundary condition
			\begin{equation}
				N(x')u = \gamma(x')u + \sum_{i=1}^{n} \beta_i(x')D_iu = h(x'), \quad x' \in T,
			\end{equation}
			where $|\beta_n| \geq \kappa > 0$ for some constant $\kappa$. Assume that $f \in C^{\alpha}(\Omega)$, $h \in C^{1,\alpha}(T)$, $a^{ij}, b^{i}, c \in C^{\alpha}(\Omega)$ and $\gamma, \beta_i \in C^{1,\alpha}(T)$ with
			\[
			\|a^{ij}, b^i, c\|_{C^{0,\alpha}(\Omega)}, \|\gamma, \beta_i\|_{C^{1,\alpha}(T)} \leq \Lambda, \quad i, j = 1, \cdots, n.
			\]
			Then for any $\Omega^{\prime}\subset\subset \Omega\cup T$, 
			\begin{equation}
				\|u\|_{C^{2,\alpha}(\Omega^{\prime})} \leq C(\|u\|_{C^0(\Omega)} + \|h\|_{C^{1,\alpha}(T)} + \|f\|_{C^{\alpha}(\Omega)}),
			\end{equation}
			where $C = C(n, \Omega^{\prime}, \alpha, \lambda, \kappa, \Lambda, \text{diam}\ \Omega)$ is a constant.
		\end{theorem}
		
		\begin{altproof}{Lemma~\ref{lem:schauder}}
			By combining the isothermal coordinates  with the results from Theorem~\ref{thm:inter_schauder} and Theorem~\ref{thm:b_schauder}, we can infer the lemma.
			
			We consider  $u\in C^{2,\alpha}(\Sigma)$ solving~\eqref{eq:schauder_c2_alpha}. For each point $\zeta\in\Sigma$, there exists an isothermal chart  $(U(\zeta), y_{\zeta})$  defined in Section~\ref{prelim}.
			Given the compactness of $\Sigma$, it can be expressed as a finite union of local charts:
			$$\Sigma= \bigcup_{i=1}^{l_1+l_2} U_{r_{\zeta_i
			}}(\zeta_i),$$
			where $\zeta_i\in \intsigma,$ for $i=1,\cdots, l_1$ and $\zeta_i\in\partial\Sigma$ for $i=l_1+1,\cdots, l_1+l_2$ and $U_{r_{\zeta_i}}\subset U(\zeta_i)$.  
			Applying Theorem~\ref{thm:inter_schauder}, for each $i=1,\cdots, l_1$, 
			\[ 	\|u\|_{C^{2, \alpha}(U_{r_{\zeta_i}}(\zeta_i))} \leq C( \|u\|_{C^0(U(\zeta_i))}+\|f\|_{C^{\alpha}(U(\zeta_i))}).\]
			Then, utilizing the method in \cite[Theorem 6.31]{GilbargTrudinger2001}, we estimate $\|u\|_{C^0(U(\zeta_i))}$ in term of $\|f\|_{C^0(\Sigma)}.$
			Consequently, 
			\[ 	\|u\|_{C^{2, \alpha}(U_{r_{\zeta_i}}(\zeta_i))} \leq C( \|f\|_{C^{\alpha}(\Sigma)}). \]
			Similarly, 
			Theorem~\ref{thm:b_schauder} implies that  for $i=l_1+1,\cdots, l_1+l_2$, 
			\[ 	\|u\|_{C^{2,\alpha}(U_{r_{\zeta_i}}(\zeta_i))} \leq C(\|u\|_{C^0(U(\zeta_i))} +\|h\|_{_{C^{1,\alpha}(U(\zeta_i)\cap \partial\Sigma)}}+ \|f\|_{C^{\alpha}(\Sigma)}).\]
			\cite[Theorem~6.31]{GilbargTrudinger2001} yields that 
			$\|u\|_{C^0(U(\zeta_i))}\leq C \|f\|_{C^0(\Sigma)}. $ It follows that 
			\[ 	\|u\|_{C^{2, \alpha}(U_{r_{\zeta_i}}(\zeta_i))} \leq C( \|f\|_{C^{\alpha}(\Sigma)}). \]
			Summing up the local Schauder estimates for $i=1,\cdots, l_1+l_2$, we deduce that
			\begin{equation}\label{eq:schauder_ex}
				\|u\|_{C^{2, \alpha}(\Sigma)} \leq C( \|f\|_{C^{\alpha}(\Sigma)}+\|h\|_{C^{1,\alpha}(\Sigma)}) .
			\end{equation}
			
			Applying Lemma~\ref{lem:Lp_est}, when $h\equiv 0$ we have a unique solution $u\in W^{2,2}(\Sigma)$ solving~\eqref{eq:schauder_c2_alpha}. Then the estimate~\eqref{eq:schauder_ex} implies $u\in C^{2,\alpha}(\Sigma).$  Due to the Fredholm alternative mentioned in \cite[P. 130]{GilbargTrudinger2001},  for any inhomogeneous $h\in C^{2,\alpha}(\Sigma)$ satisfying~\eqref{eq:f_h}, there exists a unique solution $u\in C^{2,\alpha}_0(\Sigma)$ of~\eqref{eq:schauder_c2_alpha}. 
		\end{altproof}
		\begin{lemma}~\label{lem:re_green}
			For any fixed $\xi\in \Sigma$ and $\alpha\in (0,1)$, 
			$H^g_{\xi}$ is $C^{1,\alpha}$-smooth. Moreover, $ H^g_{\xi}$ is uniformly bounded in $C^{1,\alpha}(\Sigma)$ for any $\xi$ in any compact subset of $\intsigma$ or on $\partial\Sigma$.\\
		\end{lemma} 
		\begin{proof}
			We apply the isothermal coordinate $(y_{\xi}, U(\xi))$ introduced in Section~\ref{prelim}.
			By the transformation law for $\Delta_g$ under a conformal map, $\Delta_{\tilde{g}}=e^{-\varphi}\Delta_g$ for any $\tilde{g}=e^{\varphi} g$. It follows that  
			$ \Delta_g\left( \log\frac 1{ |y_{\xi}(x)|}\right) =e^{-{\varphi}_{\xi}(y)} \left. \Delta \log\frac 1 {|y|}\right|_{y= y_{\xi}(x)}= -\frac{\varrho(\xi)} 4 \delta_{\xi},$
			where $\delta_{\xi}$ is the Dirac mass concentrated at $\xi\in \Sigma$.  
			For any $x\in U(\xi)\cap \partial \Sigma$, 
			\begin{equation*}
				\begin{aligned}
					\partial_{ \nu_g } \log|y_{\xi}(x)| &
					\stackrel{\eqref{eq:out_normal_derivatives}}{=} \left. - e^{-\frac 1 2 {\varphi}_{\xi}(y)}\frac {\partial}{\partial y_2} \log |y| \right|_{y=y_{\xi}(x)}= \left. - e^{-\frac 1 2 {\varphi}_{\xi}(y)}\frac{y_2}{|y|^2}\right|_{y=y_{\xi}(x)}\equiv 0. 
				\end{aligned}
			\end{equation*}
			Clearly, $\partial_{\nu_g} \chi(|y_{\xi}(x)|)=0$ for $x \in \partial\Sigma \cap  U_{r_0}(\xi)$.
			It follows that that $\partial_{ \nu_g } H^g(\cdot,\xi)$ is smooth on  $\partial \Sigma$. 
			$\Delta_gH^g(\cdot,\xi)$ is bounded in $L^p(\Sigma)$, for any $p\geq 1$. 
			Using the $L^p$-estimate in Lemma~\ref{lem:inh_LP}, we derive that 
			\begin{eqnarray*}
				\|H^g_{\xi} -\overline{H^g_{\xi}}\|_{C^{2,\alpha}(\Sigma)}&\leq& C (\|\partial_{\nu_g} H^g_{\xi}\|_{W_{\partial}^{1,p}(\Sigma)}+
				|-\Delta_g H^g_{\xi}|_{L^p(\Sigma)})
			\end{eqnarray*}
			for same constant $C>0$ which is independent with $\xi.$ Given $p=\frac{2}{1-\alpha}$  for any $\alpha\in (0,1)$, the Sobolev embedding theorem yields that  $H^g_{\xi}(x)$  in $C^{1,\alpha}(\Sigma)$. 
			Considering  that  $|-\Delta_g H^g(\cdot,\xi)|_{L^p(\Sigma) }$,  $\|\partial_{\nu_g}H^g_{\xi}\|_{C^{1,\alpha}(\partial\Sigma)}$ and $\left|\int_{\Sigma} H^g_{\xi} \, dv_g \right| $ are uniformly bounded  for any $\xi$ in any compact subset of $\intsigma$ or on  $\partial\Sigma$, we have $H^g_{\xi}(x)$ is uniformly bounded for any $\xi$ in any compact subset of $\intsigma$ or on $\partial\Sigma$. 
		\end{proof}

		\begin{lemma}~\label{H_INFTY} Suppose that $V>0$ on $\Sigma$. Then, for any $\xi \in \intsigma$, we have:
			\begin{equation*}
				R^g(\xi,\xi) = H^g(\xi,\xi) \rightarrow +\infty \text{ as } \xi \text{ approaches } \partial\Sigma.
			\end{equation*}
			Furthermore, for any $\xi = (\xi_1, \cdots, \xi_{k,l}) \in \Xi_{k,l}$, it holds that
			\[ \cF^V_{k,l}(\xi)\rightarrow +\infty,\]
			as $\xi$ approaches $\partial\Xi_{k,l}$.
		\end{lemma}
		\begin{proof}
			Since $V(x) > 0$, for any $x \in \Sigma$ the function $\cF^V_{k,l}$ is well-defined on $$\Xi_{k,l} = \intsigma^k \times (\partial \Sigma)^l \setminus \F_{k,l}(\Sigma).$$
			For any $\zeta \in \partial\Sigma$, consider an isothermal chart $(y_{\zeta}, U(\zeta))$. Set $r_0 = r_{\zeta}/2$. Then, for any $\xi \in U_{r_{\zeta}}(\zeta)$, we decompose the Green's function as follows:
			$$G^g(x,\xi)= \tilde{H}^g(x,\xi)-\frac 4 {\varrho(\xi)}\chi\left(\frac{|y_{\zeta}(x)-y_{\zeta}(\xi)|}{r_0}\right)\log  {|y_{\zeta}(x)-y_{\zeta}(\xi) |},$$
			where $\chi$ is a cut-off function defined by~\eqref{eq:cut_off_i}.	Applying the representation formula and divergence theorem, for any $\xi\in U_{r_{\zeta}}(\zeta)$, we obtain
			\begin{eqnarray*}
				&&\tilde{H}^g(\xi,\xi)= \int_{\Sigma} G^g(x,\xi)(-\Delta_g+\beta) \tilde{H}^g(x,\xi) \, dv_g(x)+ \int_{\partial\Sigma}  G^g(x,\xi) \partial_{\nu_g}\tilde{H}^g(x,\xi) ds_g(x)+\mathcal{O}(1) \\
				&=& \int_{\Sigma} (|\nabla \tilde{H}^g(x,\xi)|_g^2 + \beta| \tilde{H}^g(x,\xi)|^2
				) \, dv_g(x)	- \frac 1 {4\pi^2} \int_{\partial\Sigma} \partial_{\nu_g}\left(\chi(|y_{\zeta}(x)-y_{\zeta}(\xi)|)\log {|y_{\zeta}(x)-y_{\zeta}(\xi) |}\right) \\
				&&\cdot \chi(|y_{\zeta}(x)-y_{\zeta}(\xi)|)\log  {|y_{\zeta}(x)-y_{\zeta}(\xi) |}  ds_g(x) +\mathcal{O}(1)\\
				&\geq & - \frac 1 {4\pi^2} \int_{\{x: |y_{\zeta}(x)-y_{\zeta}(\xi)|< r_0\}\cap \partial \Sigma} \log {|y_{\zeta}(x)-y_{\zeta}(\xi) |} \partial_{\nu_g} \log  {|y_{\zeta}(x)-y_{\zeta}(\xi) |} ds_g(x)+  \mathcal{O}(1) \\
				&\geq &\frac 1 {4\pi^2} \int_{ \{y: |y-y_{\zeta}(\xi)|< r_0\}\cap \partial\R^2_+} \frac {-y_{\zeta}(\xi)_2} {|y-y_{\zeta}(\xi) |^2} \log {|y-y_{\zeta}(\xi) |}  \,d y_1+  \mathcal{O}(1) \\
				&\geq &- \frac 1 {4\pi^2} \log ({|y_{\zeta}(\xi)_2|})\int_{\R} \frac {1}{1+s^2} ds+  \mathcal{O}(1) =- \frac{1}{4\pi } \log |y_{\zeta}(\xi)_2|+\mathcal{O}(1)\rightarrow +\infty ,
			\end{eqnarray*}
			as $ d_g(\xi,\partial\Sigma)\rightarrow 0$, where $ds_g$ is the line element of $\partial\Sigma$. 
			It is straightforward to see that $H^g(\xi,\xi)=\tilde{H}^g(\xi,\xi)$. The first statement is concluded. 
			
			Next, we assume that $\xi \in \Xi_{k,l}$.  
			\begin{claim}\label{claim:G^g}
				There exists a constant $c_0$ satisfying  $ G^g(\xi_i,\xi_j)\geq c_0$, for any $\xi_i\neq \xi_j$.
			\end{claim}
			Before proving Claim~\ref{claim:G^g} we first show how Lemma~\ref{H_INFTY} follows.
			We denote that 
			$  \mathcal{I}_0=\left\{i:  1\leq i\leq k, d_g(\xi_i, \partial\Sigma)\rightarrow 0 \text{ as } \xi \text{ going to  } \partial\Xi_{k,l}\right\}. $
			For any $i\in \mathcal{I}_0$, 
			$H^g(\xi_i,\xi_i)\rightarrow +\infty$. 
			There exists a compact subset set $F$ of $\intsigma$ such that  $\xi_i \in F$ for any $i\in \{1,\cdots,k\} \setminus \mathcal{I}_0$. 
			It follows that any $i\notin \mathcal{I}_0,$ $H^g(\xi_i,\xi_i)\geq - \sup_{x \in F} \|H^g_ x\|_{C(\Sigma)}>-\infty. $\\
			{\it Case I. $ \mathcal{I}_0\neq \emptyset$.} 
			As $\xi \text{ approaches } \partial\Xi_{k,l}$,
			\begin{eqnarray*}
				\cF^V_{k,l}(\xi)&\geq &\sum_{i\in  \mathcal{I}_0 } \varrho(\xi_i)^2 H^g(\xi_i,\xi_i) - \sum_{i\notin \mathcal{I}_0}  \sup_{x \in \partial\Sigma\cup F} \varrho(\xi_i)^2\|H^g(\cdot, x)\|_{C(\Sigma)}\\
				& &-\sum_{i\neq h} \varrho(\xi_i)\varrho(\xi_h)|c_0|+ \sum_{i=1}^{k+l} 2 \varrho(\xi_i) \inf_{x\in \Sigma} \log V(x) \rightarrow+\infty. 
			\end{eqnarray*}
			{\it Case II. $ \mathcal{I}_0= \emptyset$. }
			Then there exists a compact subset $F$ such that $\xi_i\in F$ for any $1\leq i\leq k$ and 
			\[\mathcal{I}_1:=\{(i,j): i,j=1,\cdots, k+l; i\neq j \text{ such that } d_g(\xi_i,\xi_j)\rightarrow 0 \text{ as }  \xi \rightarrow  \partial\Xi_{k,l} \}\] 
			is non-empty. 
			For any  $(i,j)\in \mathcal{I}_1$, 
			\begin{eqnarray*}
				G^g(\xi_i,\xi_j) &=& H^g(\xi_i,\xi_j) + \frac 4{\varrho(\xi_j)} \chi(|y_{\xi_j}(\xi_i)|/r_0) \log \frac 1 { |y_{\xi_j}(\xi_i)| } \\
				&\geq&  -  \sup_{x \in  F\cup \partial \Sigma} \|H^g(\cdot, x)\|_{C(\Sigma)}+c_1\frac 4 {\varrho(\xi_j)} \log \frac 1 {|d_g(\xi_i,\xi_j)|},
			\end{eqnarray*}
			in which $c_1>0$ is a constant. 
			Consequently, as $\xi \text{ approaches to } \partial\Xi_{k,l}$, 
			\begin{eqnarray*}
				\cF^V_{k,l}{(\xi)}&\geq&  -64\pi^2 (k+l)^2 c_0+ 64\pi^2 (k+l)  \sup_{x \in \partial\Sigma\cup F} \|H^g(\cdot, x)\|_{C(\Sigma)} \\
				&&+ c_1\sum_{(i,j)\in \mathcal{I}_1} \frac 4 {\varrho(\xi_i)}\log \frac 1 {d_g(\xi_i,\xi_j)} + \sum_{i=1}^{k+l} 2 \varrho(\xi_i) \inf_{x\in \Sigma} \log V(x)
				\rightarrow +\infty.
			\end{eqnarray*}
			{ \it It remains to establish Claim~\ref{claim:G^g}.}
			We begin by decomposing the Green's function as follows:
			\begin{eqnarray*}
				G^g(x,\xi_j)&=& H^g(x,\xi_j) +\frac 4 {\varrho(\xi_j)} \chi(|y_{\xi_j}(x)|/r_0) \log \frac 1 {|y_{\xi_j}(x)|}\\
				&\geq& - \|H^g(\cdot, \xi_j)\|_{C(\Sigma)} +  \frac 4 {\varrho(\xi_j)} \chi(|y_{\xi_j}(x)|/ r_0) \log \frac 1 {|y_{\xi_j}(x)|}.
			\end{eqnarray*}
			If $\xi_j\in \partial\Sigma$,  $\|H^g(\cdot, \xi_j)\|_{C(\Sigma)}$ is uniformly bounded. It is clear that $G^g(x,\xi_j)\geq c_0$, for some $c_0>0$. Thus, it suffices to focus on the cases where  $j=1,\cdots, k.$
			We observe that $G^g(x,\xi_j)\in C^{1,\alpha}_{loc}( \Sigma\setminus\{\xi_j\} )$ for any $\alpha\in (0,1)$ and $\lim_{x\rightarrow \xi_j}G^g(x,\xi_j)= +\infty.$
			Let $h(x)$ be the unique solution of the Dirichlet problem: 
			\[ \left\{\begin{array}{ll}
				(-\Delta_g+\beta) h(x)= -\frac 1 {|\Sigma|_g}, & x\in \intsigma\\
				h(x)=0 & x\in \partial\Sigma
			\end{array}\right.. \]
			Define that $\tilde{G}^g(x,\xi_j)= G^g(x,\xi_j) - h(x)$. Then, $-\Delta_g \tilde{G}^g(x,\xi_j)=0 $ on $\Sigma\setminus \{\xi_j\}$. Considering that  $\lim_{x\rightarrow \xi_j} \tilde{G}^g(x,\xi_j)=+\infty$,  it follows that 
			\[ \inf_{\Sigma\setminus \{\xi_j\} } \tilde{G}^g(x,\xi_j)= \min_{x\in \partial\Sigma} \tilde{G}^g(x,\xi_j), \]
			by the maximum principle.  
			Thus we have  for some constants $c_2, c_0>0$
			\begin{eqnarray*}
				\inf_{\Sigma\setminus \{\xi_j\} } {G}(x,\xi_j)
				&\geq &\inf_{\Sigma\setminus \{\xi_j\} } \tilde{G}^g(x,\xi_j)- \|h\|_{C(\Sigma)} \geq  \min_{x\in \partial\Sigma} \tilde{G}^g(x,\xi_j) -  \|h\|_{C(\Sigma)}
				\\
				&\geq&   \min_{x\in \partial\Sigma} {G}^g(\xi_j, x) - 2 \|h\|_{C(\Sigma)}\\
				&\geq &-\sup_{x\in \partial \Sigma}\|H^g_{x}\|_{C(\Sigma)} - 2 \|h\|_{C(\Sigma)}+c_2 \min_{x\in \partial\Sigma} \frac 1 {\pi}\log\frac 1 {d_g(x,\xi_j)}:=c_0.
			\end{eqnarray*}
		\end{proof}
		
		\section{Technique estimates}
		Firstly, this section will provide detailed proofs of crucial estimates for the projected bubbles $PU_{\tau,\xi}$ for $\tau\in (0,\infty)$ and $\xi\in \Sigma$.
		For any $\xi$  in a compact subset of  $\intsigma$ or $\partial \Sigma$, we set $r_{\xi}$ to be $2r_0$, where $r_0$ is a positive constant.

		The following lemma is the asymptotic expansion of $PU_{\tau,\xi}$  as $\varepsilon\rightarrow 0$. 
		\begin{lemma}~\label{lemb1}
			The function $P U_{\delta, \xi}$ satisfies
			$$
			P U_{\tau, \xi}=\chi_{\xi}\left(U_{\tau, \xi}-\log \left(8 \tau^2\right)\right)+\varrho(\xi) H^g(x, \xi)+{\mathcal{O}}(\varepsilon^{1+\alpha_0}) ,
			$$
			for any $\alpha_0\in (0,1)$ and the convergent is  locally uniform for $\xi$ in  $\intsigma$ and $\partial \Sigma$ and also locally uniform for $\tau$ in $(0, +\infty)$. 
			In particular, 
			$$
			P U_{\tau, \xi}=\varrho(\xi)G^g(x, \xi)+{\mathcal{O}}(\varepsilon^{1+\alpha_0}) , 
			$$
			locally uniformly in $\Sigma\backslash\{\xi\}$.
		\end{lemma}
		\begin{proof}
			Let $\eta_{\tau,\xi}(x)= PU_{\tau,\xi}-\chi_{\xi}(U_{\tau,\xi}-\log 8\tau^2)  -\varrho(\xi) H^g(x,\xi) $. 
			If $\xi\in \intsigma$,  $$\partial _{\nu_g}\eta_{\tau,\xi}= 2\partial_{\nu_g} \chi_{\xi} \log \left(1+\frac{\tau^2\varepsilon^2}{|y_{\xi}(x)|^2}\right) -2 \chi_{\xi} \partial_{ \nu_g}\log \left(1+\frac{\tau^2\varepsilon^2}{|y_{\xi}(x)|^2}\right) \equiv 0$$ 
			on $\partial \Sigma$.  
			We observe that for any $x\in \partial\Sigma\cap U(\xi)$
			\[ \partial_{\nu_g} |y_{\xi}(x)|^2=- e^{-\frac 1 2 \hat{\varphi}_{\xi}(y)} \left. \frac{\partial}{\partial y_2} |y|^2\right|_{y=y_{\xi}(x)}= 0.\]
			If $\xi\in \partial \Sigma$,  
			for any $x\in \partial \Sigma$, as $\varepsilon\rightarrow 0.$
			\begin{eqnarray*}
				\partial_{ \nu_g}\eta_{\tau,\xi}(x)
				&=& 2(\partial_{ \nu_g}\chi_{\xi}) \frac{\tau^2\varepsilon^2}{|y_{\xi}(x)|^2} -2\chi_{\xi} \partial_{ \nu_g}\log \left(1+\frac{\tau^2\varepsilon^2}{|y_{\xi}(x)|^2}\right)+\mathcal{ O}(\varepsilon^4) =\mathcal{O}(\varepsilon^2). 
			\end{eqnarray*}
			Then for any $\xi\in \Sigma$ we have 	$
			\partial_{ \nu_g} \eta_{\tau,\xi} =
			\mathcal{ O}(\varepsilon^2). $
			For any  $A\subset \mathbb{R}^2$, denote $aA:=\{ ay: y\in A\}$. 
			\begin{eqnarray*}
				&&\int_{\Sigma} \eta_{\tau,\xi} \, dv_g =  -\int_{\Sigma} \chi_{\xi}( U_{\tau,\xi} -\log(8\tau^2)) +\varrho(x) \Gamma_{\xi}(x) \, dv_g(x) \\
				&=& -\int_{\Sigma} \chi_{\xi} \log \frac{ |y_{\xi}(x)|^4}{( \tau^2\varepsilon^2+|y_{\xi}(x)|^2)^2} \, dv_g(x)\\
				&=& 2\int_{B^{\xi}_{r_0}} \log \frac{\tau^2\varepsilon^2 +|y|^2}{ |y|^2} e^{-\hat{\varphi}_{\xi}(y)} \,d y +2 \int_{B^{\xi}_{2r_0}\setminus B_{r_0}(0)} \chi(|y|)
				\left( \frac{\tau^2\varepsilon^2 }{|y|^2} +\mathcal{ O}(\varepsilon^4)\right) e^{\hat{\varphi}_{\xi}(y)} \,d y\\
				&=&  2\tau^2\varepsilon^2 \int_{\frac 1 {\tau\varepsilon}(B^{\xi}_{r_0}\cap B_{r_0}(0))} \log \left( 1+ \frac 1 {|y|^2}\right) e^{-\hat{\varphi}(\tau\varepsilon y)} \, d y +\mathcal{ O}(\varepsilon^2)\\
				&=& 2\tau^2\varepsilon^2 (1+\mathcal{O}(\varepsilon)) \int_{B_{r_0/(\tau\varepsilon)}(0) }\log \left( 1+\frac 1 {|y|^2}\right) \, d y+\mathcal{O}(\varepsilon^2)={\mathcal{ O}}(\varepsilon^2|\log \varepsilon|), 
			\end{eqnarray*}
			where we applied 
			\begin{eqnarray*}
				&&\int_{|y|<\frac {r_0}{\tau\varepsilon}}     \log \left( 1+\frac 1 {|y|^2}\right) \, d y 
				=2\pi \int_0^{r_0/(\tau\varepsilon)} \log   \left( 1+\frac 1 {r^2}\right)r  dr\\
				&=& \pi \int_0^{r^2_0/(\tau\varepsilon)^2} \log   \left( 1+\frac 1 {t}\right) dt \\
				&=& \pi \frac{r_0^2}{\tau^2\varepsilon^2} \log\left(1+\frac{\tau^2\varepsilon^2}{r_0^2}\right)-\pi \int_0^{r_0^2/(\tau\varepsilon)^2} (1-\frac 1 {1+t}) dt \\
				&=& \pi \frac{r_0^2}{\tau^2\varepsilon^2}\left(1+\frac{\tau^2\varepsilon^2}{r_0^2}+ \mathcal{ O}(\varepsilon^4)\right) - \pi\frac{r_0^2}{\tau^2\varepsilon^2} + \pi \log \left( 1+ \frac{r_0^2}{\tau^2\varepsilon^2} \right)\\
				&=& \mathcal{ O}(|\log \varepsilon|).
			\end{eqnarray*}
			For any $x\in U_{2r_0}(\xi)$, $-\Delta_g U_{\tau,\xi}= e^{-\hat{\varphi}_{\xi}(y)} \Delta u_{\tau,0}|_{ y=y_{\xi}(x)} =e^{-\varphi_{\xi}} e^{U_{\tau,\xi}}$.  It follows that 
			\begin{eqnarray*}
				(-\Delta_{g}+\beta) \eta_{\tau,\xi}&=& (-\Delta _g+\beta)\left( PU_{\tau,\xi}-\chi_{\xi}(U_{\tau,\xi}-\log  8\tau^2)  -\varrho(\xi) H^g_{\xi}   \right)\\
				&=&
				(\Delta_g\chi_{\xi})\log\frac{|y_{\xi}|^4}{(\tau^2\varepsilon^2+|y_{\xi}|^2)^2}+2\lan \nabla\chi_{\xi},\nabla\log\frac{|y_{\xi}|^4}{(\tau^2\varepsilon^2+|y_{\xi}|^2)^2}\ran_g\\
				&&+\frac 1 {|\Sigma|_g} \left( \varrho(\xi)- \int_{\Sigma} \varepsilon^2\chi_{\xi} e^{-\varphi_{\xi}} e^{U_{\tau,\xi}} \, dv_g \right)+2\beta\log\left(1+\frac{\tau^2\varepsilon^2}{|y_{\xi}|^2}\right). 
			\end{eqnarray*}
			We observe that $\Delta_{g} \chi_{\xi}\equiv 0$ and $\nabla \chi_{\xi} \equiv 0$ in $U_{2r_0}(\xi)\setminus U_{r_0}(\xi)$. 
			For any $x\in A_{r_0}(\xi)$, we have
			\[ U_{\tau,\xi}-\log(8\tau^2) + 4\log |y_{\xi}(x)|
			= -2\log\left( 1+ \frac{\tau^2\varepsilon^2}{|y_{\xi}(x)|^2}\right)
			= -2\tau^2\varepsilon^2  |y_{\xi}(x)|^{-2} +\mathcal{ O}(\varepsilon^4) \] 
			and 
			\[	\nabla \left(  U_{\tau,\xi}-\log(8\tau^2) + 4\log |y_{\xi}(x)|\right) 
			= -2\tau^2\varepsilon^2 \nabla |y_{\xi}(x)|^{-2} +\mathcal{ O}(\varepsilon^4).  \] 
			\begin{eqnarray*}
				&&	\int_{\Sigma} \varepsilon^2 \chi_{\xi} e^{-\varphi_{\xi}} e^{U_{\tau,\xi}} \, dv_g 
				=\int_{B_{2r_0}^{\xi}} \varepsilon^2 \chi(|y|) \frac{8\tau^2}{(\tau^2\varepsilon^2 +|y|^2)^2} \, d y \\
				&=&  \int_{B_{r_0}^{\xi}} \varepsilon^2 \chi(|y|) \frac{8\tau^2}{(\tau^2\varepsilon^2 +|y|^2)^2} \, d y +\mathcal{ O}(\varepsilon^2 )= \varrho(\xi)+{\mathcal{O}}(\varepsilon^2),
			\end{eqnarray*}
			where we applied the fact that $ 	\int _{|y|<r} \frac{\tau^2\varepsilon^2}{     ( \tau^2\varepsilon^2+ |y|^2)^2} \, d y = \pi- \frac{\pi\tau^2\varepsilon^2}{r^2}+ \frac{\pi\tau^4\varepsilon^4}{(r^2+\tau^2\varepsilon^2)r^2}$ for any $r\geq 0.$
			Hence, for any $p>1$
			$ |(-\Delta_g+\beta)\eta_{\tau,\xi}|_{L^p(\Sigma)}= \mathcal{O}(\varepsilon^2+\beta \varepsilon^{\frac 2 p}). $
			By the regularity theory in  Lemma~\ref{lem:inh_LP}, we have  
			$
			\| \eta_{\tau,\xi}-\overline{ \eta_{\tau,\xi}}\|_{W^{2,p}(\Sigma)}\leq C\left( \left\| {\partial_{\nu_g} \eta_{\tau,\xi} }\right\|_{W^{1,p}_{\partial}(\Sigma)} + \left|(-\Delta_g+\beta) \eta_{\tau,\xi} \right|_{L^p(\Sigma)}\right)\leq C(\varepsilon^2+\beta \varepsilon^{\frac 2 p} ),
			$
			for $p>1$. 
			We take $p\in(1,2)$	such that $\alpha_0=\frac 2 p- 1>0 $. Then, the Sobolev inequality  implies that as $\varepsilon\rightarrow 0$,
			$\eta_{\tau,\xi}= \mathcal{O}(\varepsilon^{1+\alpha_0}),$
			uniformly in $C(\Sigma).$
		\end{proof} 
		
		\begin{lemma}~\label{lemb2}
			If $p\geq 1$ then $| \varepsilon^2 \chi_{\xi} e^{U_{\tau,\xi}}|_{L^p(\Sigma)}=\mathcal{O}(\varepsilon^{\frac{2(1-p)}{p}})$ which is   uniform for $\xi$ in $ \Sigma$ and  locally uniform for $\tau$ in $(0, +\infty)$. 
		\end{lemma}
		
		\begin{proof} By direct calculation, we have 
			\begin{eqnarray*}
				\int_{\Sigma}(	\varepsilon^2 \chi_{\xi}  e^{U_{\tau,\xi}})^p \, dv_g &= &   \int_{B^{\xi}_{2r_0} } e^{\hat{\varphi}_{\xi} (y)
				} \frac{ (8\tau^2 \varepsilon^2)^p}{ (\tau^2\varepsilon^2+|y|^2)^p} \, d y \\
				&=& \int_{B^{\xi}_{2r_0} }
				\frac{ (8\tau^2 \varepsilon^2)^p}{ (\tau^2\varepsilon^2+|y|^2)^p} \, d y +
				\int_{B^{\xi}_{2r_0} }( e^{\hat{\varphi}_{\xi} (y)
				}-1) \frac{ (8\tau^2 \varepsilon^2)^p}{ (\tau^2\varepsilon^2+|y|^2)^p} \, d y\\
				&= &(\tau^2\varepsilon^2)^{1-p} \int_{\frac{1}{\tau\varepsilon} B^{\xi}_{2r_0}  } {(1+ {\mathcal{O}} ( \tau\varepsilon |y|)) } \frac{8}{(1+|y|^2)^2} \, d y  =\mathcal{O}( \varepsilon^{2(1-p)}). 
			\end{eqnarray*}
			Thus 
			$| \varepsilon^2 \chi_{\xi}e^{-\varphi_{\xi}} e^{U_{\tau,\xi}}|_{L^p(\Sigma)}=\mathcal{O}(\varepsilon^{\frac{2(1-p)}{p}})$
			uniformly in $\xi \in \Sigma$ and $\tau$ is bounded away from zero. 
		\end{proof}
		Next, we discuss the asymptotic expansions of $P\Psi^j_i$ as $\varepsilon\rightarrow 0$ analogue to $PU_{\tau,\xi}$.  
		\begin{lemma}~\label{lemb3}
			{\it For any $\alpha_0\in (0,1)$, 
				\[P\Psi^0_{\tau,\xi}(x)=\chi_{\xi} \left(\Psi^0_{\tau,\xi}(x) -\frac{2}{\tau}\right) +\mathcal{O}(\varepsilon^{1+\alpha_0}) = -4\chi_{\xi}(x)\frac{ \tau\varepsilon^2}{ \tau^2\varepsilon^2 +|y_{\xi}(x)|^2}+\mathcal{O} (\varepsilon^{1+\alpha_0})\] 
				in $C(\Sigma)$  as $\varepsilon\rightarrow 0$. 
				And 
				$ P\Psi^0_{\tau,\xi}(x)=\mathcal{O}(\varepsilon^{1+\alpha_0}) , $
				in $C_{loc}(\Sigma \setminus\{\xi\})$  uniformly for $\xi$ in any compact subset of $\intsigma$ or $\xi\in \partial \Sigma$ and $\tau$ is bounded away from zero.  
				For $\xi\in \intsigma$ with $j=1,2$, or for $\xi\in \partial\Sigma$ with $j=1$,
				\[P\Psi^j_{\tau,\xi}(x)
				= \chi_{\xi}(x) \Psi^j_{\tau,\xi}(x)+\varrho(\xi) H^j(x,\xi) +{ \mathcal{O}}
				(\varepsilon^{\alpha_0})
				\] 
				in $C(\Sigma)$ as $\varepsilon\rightarrow 0$, where  $H^j(x,\xi)$ is the unique solution of the following problem
				\begin{equation}~\label{H_j}
					\begin{cases}
						(-\Delta_{g}+\beta) H^j(x,\xi)&=-\beta\frac{4}{\varrho(\xi)}\chi_{\xi} \frac{y_{\xi}(x)_j}{|y_{\xi}(x)|^2} + \frac{4}{\varrho(\xi)} (\Delta_g \chi_{\xi}) \frac{ y_{\xi}(x)_j}{|y_{\xi}(x)|^2}\\
						&+\frac{8}{\varrho(\xi)} \left\la\nabla \chi_{\xi}, \nabla\left( \frac{ y_{\xi}(x)_j}{|y_{\xi}(x)|^2}\right)\right\ra_g,\quad x\in \intsigma\\
						\partial_{ \nu_g} H^j(x,\xi) &=-\frac{4}{\varrho(\xi)} \partial_{ \nu_g} \left( \frac{ y_{\xi}(x)_j }{|y_{\xi}(x)|^2}\right) \chi_{\xi} -\frac{4}{\varrho(\xi)} \frac{ y_{\xi}(x)_j}{|y_{\xi}|^2} \partial_{\nu_g} \chi_{\xi},\quad x\in \partial \Sigma\\
						\int_{\Sigma} H^{j}(x,\xi) \, dv_g&= -\frac{ 4}{\varrho(\xi)} \int_{\Sigma}  \frac{ y_{\xi}(x)_j}{|y_{\xi}(x)|^2} \chi_{\xi}(x) \, dv_g 
					\end{cases}.
				\end{equation} 
				In addition, the convergences above are 
				uniform for $\xi$ in any compact subset of $\intsigma$ or  $\xi\in \partial \Sigma$ and $\tau$ bounded away from zero.}
		\end{lemma}
		\begin{proof}
			Let $\eta_{\tau,\xi} = P\Psi^0_{\tau,\xi} - \chi_{\xi}\left( \Psi^0_{\tau,\xi}-\frac 2 {\tau}\right). $
			For $x\in\partial \Sigma$, 
			\begin{eqnarray*}
				\partial_{\nu_g} \left(\chi_{\xi} \left(\Psi^0_{\tau,\xi} (x) -\frac{ 2}{\tau}\right)\right)&=& -\partial_{ \nu_g} \chi_{\xi} \frac{ 4\tau\varepsilon^2}{ |y_{\xi}(x)|^2 +\tau^2\varepsilon^2}+  \chi_{\xi} \frac{ 8\tau\varepsilon^2 |y_{\xi}(x)|^2}{( |y_{\xi}(x)|^2+\tau^2\varepsilon^2)^2 } \partial_{\nu_g} \log |y_{\xi}(x)|.
			\end{eqnarray*}
			If $\xi \in \intsigma$, $\partial_{ \nu_g} \eta_{\tau,\xi}\equiv 0$ in $\partial \Sigma$; if $\xi\in \partial \Sigma$, 
			$\partial_{\nu_g} \eta_{\tau,\xi}=\mathcal{O}(\varepsilon^2)$ on  $\partial\Sigma$. By direct calculation, we have 
			\begin{equation*}
				\begin{array}{ll}
					\int_{\Sigma} \chi_{\xi} \left( \Psi^0_{\tau,\xi} -\frac{2}{\tau}\right) \, dv_g &=2
					\int_{\Sigma}\chi_{\xi} \frac{ \tau\varepsilon^2}{ |y_{\xi}(x)|^2 +\tau^2\varepsilon^2 } \, dv_g(x)= 2\tau \varepsilon^2 \int_{ B_{2r_0}^{\xi}} \frac{1}{|y|^2+\tau^2\varepsilon^2} e^{\hat{\varphi} _{\xi}(y)} \, d y\\
					&=2\tau \varepsilon^2 \int_{ B_{2r_0}^{\xi}} \frac{1}{|y|^2+\tau^2\varepsilon^2}  \, d y + 2\tau \varepsilon^2 \int_{ B_{2r_0}^{\xi}} \frac{1}{|y|^2+\tau^2\varepsilon^2} (e^{\hat{\varphi} _{\xi}(y)}-1) \, d y\\
					&= \mathcal{O}(\varepsilon^2|\log \varepsilon|)
				\end{array}
			\end{equation*}
			and 
			\begin{equation*}
				\begin{array}{lcl}
					&&	(-\Delta_g+\beta) \eta_{\tau,\xi}(x)= (-\Delta_g+\beta)\left(  P\Psi^0_{\tau,\xi} - \chi_{\xi}\left( \Psi^0_{\tau,\xi}-\frac 2 {\tau}\right)   \right)\\
					&=&
					(\Delta_{g} \chi_{\xi}) \left(\Psi^0_{\tau,\xi}-\frac 2 {\tau}\right) +2 \lan \nabla \chi_{\xi}, \nabla \Psi^0_{\tau,\xi} \ran_g - \overline{ \varepsilon^2 \chi_{\xi} e^{-\varphi_{\xi}} e^{U_{\tau, \xi}} \Psi^0_{\tau,\xi} }+\beta\chi_{\xi}\frac{4\tau\varepsilon^2}{|y_{\xi}|^2+\tau^2\varepsilon^2}\\
					&=&\beta\chi_{\xi} \frac{4\tau\varepsilon^2}{|y_{\xi}|^2+\tau^2\varepsilon^2}+\mathcal{O}(\varepsilon^2),
				\end{array}
			\end{equation*}
			where we applied the fact for any fixed $r>0$, $ \int_{|y|<r} \frac{ \tau^2\varepsilon^2 -|y|^2}{(\tau^2\varepsilon^2+|y|^2)^3}= \mathcal{O}(\varepsilon^2)$ as $\varepsilon\rightarrow 0.$
			Via  the regularity theory in  Lemma~\ref{lem:inh_LP} and Sobolev inequality, there exists a constant $C>0$ such that 
			$
			\| \eta_{\tau,\xi}-\overline{ \eta_{\tau,\xi}}\|_{C(\Sigma)}
			\leq  C(\varepsilon^2+\beta \varepsilon^{\frac 2 p} |\log \varepsilon|^{\frac 1 p}).
			$
			We choose $p\in (1,2)$ such that $\alpha_0<\frac 2 p-1 $, then 
			$ \eta_{\tau,\xi}= \mathcal{O}(\varepsilon^{1+\alpha_0}), $
			uniformly in $C(\Sigma)$.

			If $\xi\in \intsigma$, $\partial_{ \nu_g} H^{j}(x,\xi)=0$ for any $x\in \partial\Sigma$. If $\xi\in \partial \Sigma$, for any $x\in\partial\Sigma$
			by direct calculation,  
			$$
			\chi_{\xi}(x) \partial_{ \nu_g}\left( \frac{ y_{\xi}(x)_1}{|y_{\xi}(x)|^2} \right)=0.$$
			Denote $\partial_{ \nu_g} H^j(\xi,\xi):= 0 $, then $\partial_{ \nu_g} H^j(\cdot,\xi)\in C^{\infty}(\partial \Sigma)$.  By Lemma~\ref{lem:schauder}, there is a unique solution to the problem~\eqref{H_j} in  $C^{1,\alpha}(\partial \Sigma)$ for any $\alpha\in (0,1).$ 
			Let $$\zeta_{\tau,\xi}(x)= P\Psi^j_{\tau,\xi}(x)-\chi_{\xi}(x)\Psi^j_{\tau,\xi}(x) - \varrho(\xi) H^j(x,\xi).$$
			Since $\int_{B} \frac{ \varepsilon^3y_j}{ (\varepsilon^2+|y|^2)^3} \, d y=	0 \text{ for  } j=1, 2 \text{ and }B= B_r \text{ or } j=1\text{ and } B= B_r\cap\{y_2\geq 0\},$ we have the following estimates:
			\begin{eqnarray*}
				\overline{ \varepsilon^2 \chi_{\xi} e^{-\varphi_{\xi}} e^{U_{\tau,\xi} } \Psi^j_{\tau,\xi}}&=&\int_{B^{\xi}_{2r_0}}  \frac{ 8\tau^2\varepsilon^2 \chi(|y|) y_j }{(\tau^2\varepsilon^2+|y|^2)^3} \, d y\\
				&=&  \int_{B}  \frac{ 8\tau^2\varepsilon^2 y_j }{(\tau^2\varepsilon^2+|y|^2)^3} \, d y +\mathcal{ O}(\varepsilon^{2})= \mathcal{O}(\varepsilon^2),
			\end{eqnarray*}
			\begin{eqnarray*}
				(-\Delta_g+\beta) \zeta_{\tau,\xi}
				&=& -\frac{ 4\tau^2\varepsilon^2 (y_{\xi})_j}{(\tau^2\varepsilon^2+ |y_{\xi}|^2) |y_{\xi}|^2}  \Delta_{g}\chi_{\xi}- 8\tau^2\varepsilon^2 \left\lan \nabla \chi_{\xi}, \nabla\left( \frac{ (y_{\xi})_j}{  (\tau^2\varepsilon^2+ |y_{\xi}|^2) |y_{\xi}|^2   }\right)\right\ran_g \\
				&&-\overline{ \varepsilon^2 \chi_{\xi} e^{-\varphi_{\xi}} e^{U_{\tau,\xi} } \Psi^j_{\tau,\xi}}+4\beta\chi_{\xi} \frac{\tau^2\varepsilon^2 (y_{\xi})_j}{(\tau^2\varepsilon^2+|y_{\xi}|^2)|y_{\xi}|^2}\\
				&=&  4\beta\chi_{\xi} \frac{\tau^2\varepsilon^2 (y_{\xi})_j}{(\tau^2\varepsilon^2+|y_{\xi}|^2)|y_{\xi}|^2} +\mathcal{O}(\varepsilon^2)
			\end{eqnarray*}
			and 
			\begin{eqnarray*}
				\int_{\Sigma} \zeta_{\tau,\xi} \, dv_g & =& 4 \int_{\Sigma} \chi_{\xi}(x)\frac{ \tau^2\varepsilon^2 y_{\xi}(x)_j}{|y_{\xi}(x)|^2( \tau^2\varepsilon^2 + |y_{\xi}(x)|^2) } \, dv_g(x) =4 \int_{B_{2r_0}^{\xi} } \chi\left( \frac {|y|} { r_0}\right) e^{\hat{\varphi}_{\xi}(y)} \frac{ \tau^2\varepsilon^2 y_j}{|y|^2 (\tau^2\varepsilon^2 +|y|^2 ) } .
			\end{eqnarray*}
			For any  $\xi\in \intsigma$ and $j=1,\cdots, \ii(\xi)$
			\begin{eqnarray*}
				&&	\left| \int_{B_{2r_0} } \chi(|y|) \frac{ \tau^2\varepsilon^2 (e^{\hat{\varphi}_{\xi}(y)}-1) y_j}{|y|^2 (\tau^2\varepsilon^2 +|y|^2 ) } \, d y \right| \leq C
				\int_{|y|<2r_0}   \frac{ \tau^2\varepsilon^2 }{ (\tau^2\varepsilon^2 +|y|^2 ) } \, d y
				= \mathcal{ O}(\varepsilon^2|\log\varepsilon|). 
			\end{eqnarray*}
		Then, we have 
			\begin{eqnarray*}
				&&\int_{\Sigma} \zeta_{\tau,\xi} \, dv_g=4 \int_{B_{2r_0}^{\xi} } \chi\left(\frac{|y|}{r_0}\right) e^{\hat{\varphi}_{\xi}(y)} \frac{ \tau^2\varepsilon^2 y_j}{|y|^2 (\tau^2\varepsilon^2 +|y|^2 ) } \, d y\\
				&=& - 4 \int_{B_{2r_0}^{\xi} } \chi\left(\frac{|y|}{r_0}\right) \frac{ \tau^2\varepsilon^2 y_j}{|y|^2 (\tau^2\varepsilon^2 +|y|^2 ) } \, d y +\mathcal{O}(\varepsilon^2|\log\varepsilon|)=\mathcal{O}(\varepsilon^2|\log\varepsilon|),
			\end{eqnarray*}
			where we applied the symmetric property of the integral 
			$\int_{B_{2r_0}} \chi\left(\frac{|y|}{r_0}\right) \frac{ \tau^2\varepsilon^2 y_j}{|y|^2 (\tau^2\varepsilon^2 +|y|^2 ) } \, d y=0 $ for $j=1,\cdots,\ii(\xi).$
		
			If $\xi \in \intsigma$, $\partial_{ \nu_g} \zeta_{\tau,\xi}(x)\equiv 0$ for any $x\in \partial \Sigma$. 
			If $x\in \partial \Sigma$, 
			by calculation, we deduce that 
			\begin{equation*}
				\begin{array}{lcl}
					\partial_{\nu_g} \zeta_{\tau,\xi} (x)&=&
					-\partial_{ \nu_g}\Big(\chi_{\xi} ( \Psi^j_{\tau,\xi} +\varrho(\xi)H^j(x,\xi))  -\chi_{\xi} \partial_{ \nu_g}\left(  \Psi^j_{\tau,\xi} +\varrho(\xi)H^j(x,\xi) \right)\Big) \\
					&=& (\partial_{ \nu_g} \chi_{\xi})\frac{ 4\tau^2\varepsilon^2 y_{\xi}(x)_j}{ 
						(\tau^2\varepsilon^2+|y_{\xi}(x)|^2) |y_{\xi}(x)|^2 } + \chi_{\xi} \partial_{ \nu_g}  \frac{ 4\tau^2\varepsilon^2 y_{\xi}(x)_j}{ 
						(\tau^2\varepsilon^2+|y_{\xi}(x)|^2) |y_{\xi}(x)|^2 }= \mathcal{O}(\varepsilon^2).
				\end{array}
			\end{equation*}
			Applying  the regularity theory in  Lemma~\ref{lem:inh_LP} and the Sobolev inequality, for any $p\in(1,2)$, we deduce that 
			$
			\| \zeta_{\tau,\xi}-\overline{ \zeta_{\tau,\xi}}\|_{C(\Sigma)}\leq C(\varepsilon^2+\beta \varepsilon^{\frac 1 p} ).
			$
			We take $p\in (0,1)$ such that $\alpha_0=\frac 1p$.
			Then as $\varepsilon\rightarrow 0$, we have 
			$\eta_{\tau,\xi}= { \mathcal{O}}(\varepsilon^{\alpha_0}) $
			uniformly in $C(\Sigma).$
		\end{proof}
		{
			\begin{remark}~\label{rklemb3}
				$\partial_{\tau} PU_{\tau,\xi}= P\Psi^{0}_{\tau,\xi}$ by the uniqueness of the solution to the problem~\eqref{ppsi}. However, $\partial_{\xi_j}PU_{\tau,\xi}\neq P\Psi^j_{\tau,\xi}.$
				Analogous to the proof of Lemma~\ref{lemb3}, we obtain the following expansion for any $\alpha_0\in (0,1)$, 
				\begin{equation}\label{eq:xi_j_oringinal}
					\partial_{\xi_j}PU_{\tau,\xi}=\chi_{\xi}  \partial_{\xi_j}(\chi_j U_{\tau,\xi}) +\varrho(\xi) \partial_{\xi_j} H^g_{\xi} +\mathcal{O}(\varepsilon^{\alpha_0}),
				\end{equation}
				as $\varepsilon\rightarrow 0$
				in $C(\Sigma)$,	which is uniformly convergent for $\xi$ in any compact subset of $\intsigma$ or $\xi\in \partial\Sigma$ and $\tau$ in any compact subset of  $(0,\infty)$.  \\
				{\it Indeed, }  we notice that for any $y\in U_{2r_0}(\xi)$ as $y\rightarrow 0$
				\[ \left. \partial_{\xi_j}|y_{\xi}(x)|^2\right|_{x=y^{-1}_{\xi}(y)}= -2y_j+ \mathcal{ O}(|y|^3).\] 
				Let $ \zeta_{\tau,\xi}^*=	\partial_{\xi_j}PU_{\tau,\xi}-\partial_{\xi_j}(\chi_{\xi}  U_{\tau,\xi})-\varrho(\xi) \partial_{\xi_j} H^g(x,\xi)$. 
				It is easy to obtain 
				\[ (-\Delta_g+\beta)\zeta_{\tau,\xi}^*= -\beta\partial_{\xi_j}\left(\chi_{\xi}U_{\tau,\xi}+4\chi_{\xi}\log |y_{\xi}| \right)+\mathcal{O}(\varepsilon^2|\log\varepsilon|), \quad\text{ in }\intsigma\]
				\[ \int_{\Sigma}\zeta_{\tau,\xi}^* \, dv_g= \mathcal{O}(\varepsilon^2|\log\varepsilon|),\]
				and $\partial_{\xi}\zeta_{\tau,\xi}^*= \mathcal{O}(\varepsilon^2)\text{ on } \partial\Sigma.$
				Applying the regularity theory in  Lemma~\ref{lem:inh_LP} and Sobolev inequality, we have 
				$ \zeta_{\tau,\xi}^*=\mathcal{O}(\varepsilon^2|\log\varepsilon|+ \beta \varepsilon^{\frac 1p}), $
				convergent in $C(\Sigma)$ for any $p\in(1,2)$. 
				We take $p\in(1,2)$ such that $\alpha_0=\frac 1 p$, then we deduce~\eqref{eq:xi_j_oringinal}.
			\end{remark}
		}
		The following lemma shows asymptotic ``orthogonality" properties of $P\Psi^j_i$. 
		\begin{lemma}~\label{lem4} For any $\alpha_0\in(0,1)$, we have as $\varepsilon\rightarrow 0$ for $j,i= 0,\cdots, \ii(\xi)$,		
			\[ \langle P\Psi^i_{\tau,\xi}, P\Psi^j_{\tau,\xi}\rangle=\begin{cases}
				\frac{ 8\varrho(\xi)D_i}{\pi\tau^2}\delta_{ij} + {\mathcal{O}}(\varepsilon^{\alpha_0}) & \text{ when }  i \text{ or } j=0\\
				\frac{ 8\varrho(\xi)D_i}{\pi\tau^2\varepsilon^2}\delta_{ij} + {\mathcal{O}}(\varepsilon^{\alpha_0-1})& \text{ otherwise } 
			\end{cases}
			,
			\]
			and 
			\[ \langle P\Psi^i_{\tau^0,\xi_0},  P\Psi^j_{\tau^1,\xi_1}\rangle=\begin{cases}
				\mathcal{O}(\varepsilon^{\alpha_0}) &\text{ when } i \text{ or } j=0\\
				{\mathcal{O}}(\varepsilon^{\alpha_0-1})& \text{ otherwise }
			\end{cases}, \]
			where three different points  $\xi,\xi_0,\xi_{1}\in\Sigma$ and uniformly in  $\tau,\tau^0,\tau^1$ are bounded away from zero and the $\delta_{ij}$ is the Kronecker symbol, and  $D_0=\int_{\mathbb{R}^2}  \frac{1-|y|^2}{ (1+|y|^2)^4 }\, d y$, $D_1= D_2=\int_{\mathbb{R}^2} \frac{ |y|^2 }{ (1+|y|^2)^4}\, d y . $
		\end{lemma}
		\begin{proof}
			We estimate the inner product  by computing the integral separately in following two areas:
			\begin{eqnarray*}
				\langle P\Psi^i_{\tau,\xi}, P\Psi^j_{\tau,\xi}\rangle&=& \int_{\Sigma}\varepsilon^2 \chi_{\xi}(x) e^{-\varphi_{\xi}} e^{U_{\tau,\xi}}\Psi^i_{\tau,\xi} P\Psi^j_{\tau,\xi} \, dv_g(x)\\
				&=& \int_{\Sigma\cap U_{2r_0}(\xi)} +\int_{\Sigma\setminus U_{2r_0}(\xi)}  \varepsilon^2 \chi_{\xi}(x) e^{-\varphi_{\xi}} e^{U_{\tau,\xi}}\Psi^i_{\tau,\xi} P\Psi^j_{\tau,\xi}  \, dv_g(x). 
			\end{eqnarray*}
			For $i=j=0$, by Lemma~\ref{lemb3}, we have  
			\begin{eqnarray*}
				&&\int_{\Sigma\cap U_{2r_0}(\xi) } \varepsilon^2 \chi_{\xi}(x) e^{-\varphi_{\xi}} e^{U_{\tau,\xi}}\Psi^0_{\tau,\xi} P\Psi^0_{\tau,\xi}  \, dv_g(x)\\
				&= &16\tau\varepsilon^2 
				\int_{B^{\xi}_{2r_0}} \chi\left(\frac{|y|}{r_0}\right)  \frac{ |y|^2-\tau^2\varepsilon^2}{(\tau^2\varepsilon^2+|y|^2)^3} \left( - \frac{ 4\tau\varepsilon^2 \chi\left(\frac{|y|}{r_0}\right)}{ \tau^2\varepsilon^2 +|y|^2} +\mathcal{O}(\varepsilon^{1+\alpha_0})\right) \, d y \\
				&= & \frac{64}{\tau^2} \int_{\frac 1 {\tau\varepsilon} B_{r_0}^{\xi}}  \frac{ 1-|y|^2}{(1+|y|^2)^4}
				+ \mathcal{O}(\varepsilon^{1+\alpha_0}).
			\end{eqnarray*} 
			Considering that 	$\frac{64}{\tau^2} \int_{\frac 1 {\tau\varepsilon} B_{r_0}^{\xi}}  \frac{ 1-|y|^2}{(1+|y|^2)^4}= \frac{8\varrho(\xi)}{\tau^2\pi} \int_{\mathbb{R}^2}  \frac{ 1-|y|^2}{(1+|y|^2)^4} \, d y+ \mathcal{ O}(\varepsilon^2)$, 
			\[ 	\la P\Psi^0_{\tau,\xi}, P\Psi^0_{\tau,\xi}\ra=\int_{\Sigma } \varepsilon^2 \chi_{\xi}(x) e^{-\varphi_{\xi}} e^{U_{\tau,\xi}}\Psi^0_{\tau,\xi} P\Psi^0_{\tau,\xi}  \, dv_g(x) =\frac{8\varrho(\xi)D_0}{\pi \tau^2} +{\mathcal{O}}(\varepsilon^{1+\alpha_0}), \]
			where $D_0=\int_{\mathbb{R}^2}  \frac{1-|y|^2}{ (1+|y|^2)^4 }\, d y.$ \par 
			Similarly, for $j=0$ and $i=1,\cdots,\ii(\xi)$  we have 
			\begin{eqnarray*}
				&&	\la P\Psi^i_{\tau,\xi}, P\Psi^0_{\tau,\xi}\ra=	\varepsilon^2\int_{\Sigma\cap U_{2r_0}(\xi)} \chi_{\xi}  e^{-\varphi_{\xi}} e^{U_{\tau,\xi}} \Psi^i_{\tau,\xi} P\Psi^0_{\tau,\xi} \, dv_g\\
				&=& 32\tau^2 \varepsilon^2\int_{ B_{2r_0}^{\xi}} \chi\left(\frac{|y|}{r_0}\right) \frac{y_i}{( \tau^2\varepsilon^2 +|y|^2)^3} \left(-\frac{ 4\tau\varepsilon^2\chi\left(\frac{|y|}{r_0}\right)  }{ \tau^2\varepsilon^2 +|y|^2}  +\mathcal{O}(\varepsilon^{1+\alpha_0}) \right) \, d y=\mathcal{O}(\varepsilon^{\alpha_0}). 
			\end{eqnarray*}
			Applying Lemma~\ref{lemb3}, for $\xi\in \intsigma$ we have 
			\begin{eqnarray*}
				&&	\varepsilon^2\int_{\Sigma\cap U_{2r_0}(\xi)} \chi_{\xi} e^{-\varphi_{\xi}} e^{U_{\tau,\xi}} \Psi^i_{\tau,\xi} P\Psi^j_{\tau,\xi} \, dv_g\\
				&
				=& 32\tau^2 \varepsilon^2\int_{ B_{2r_0}^{\xi}} \chi\left(\frac{|y|}{r_0}\right) \frac{y_i}{( \tau^2\varepsilon^2 +|y|^2)^3} \left(\chi\left(\frac{|y|}{r_0}\right) 
				\frac{ 4y_j}{ \tau^2\varepsilon^2 +|y|^2}  \right.
				\\
				&&\left.+ \varrho(\xi) H^{j}( y_{\xi}^{-1}(y),\xi)+{\mathcal{O}}(\varepsilon^{\alpha_0}) \right) \, d y\\
				&=&  \frac{128}{\tau^2\varepsilon^2}\int_{ \frac 1 {\tau\varepsilon}B_{r_0} ^{\xi}}  \frac{y_iy_j }{(1+|y|^2)^4} \, d y +   \int_{ B_{r_0}^{\xi}}  \frac{32\tau^2 \varepsilon^2\varrho(\xi)y_i}{( \tau^2\varepsilon^2 +|y|^2)^3}(H^{j}( y_{\xi}^{-1}(y),\xi)  - H^{j}(\xi,\xi)) \, d y  \\
				&&+  32\tau^2 \varepsilon^2  \varrho(\xi)  H^j(\xi,\xi) \int_{ B_{r_0}^{\xi}}  \frac{y_i}{( \tau^2\varepsilon^2 +|y|^2)^3} \, d y  + { \mathcal{O}}(\varepsilon^{\alpha_0-1})\\
				&=& \frac{128}{\tau^2\varepsilon^2}  \int_{ \frac 1 {\tau\varepsilon} B_{r_0}^{\xi}}  \frac{y_iy_j }{(1+|y|^2)^4} \, d y  +\mathcal{ O}\left(  \int_{ B_{r_0}^{\xi}}  \frac{32\tau^2 \varepsilon^2|y|^2}{( \tau^2\varepsilon^2 +|y|^2)^3}  \, d y\right) +{ \mathcal{O}}(\varepsilon^{\alpha_0-1})\\
				&=& \frac{ 8\varrho(\xi)D_i}{\pi\tau^2\varepsilon^2} \delta_{ij}+
				{\mathcal{O}}(\varepsilon^{\alpha_0-1}) ( \varepsilon\rightarrow 0),
			\end{eqnarray*}
			where  $D_i= \int_{\mathbb{R}^2} \frac{ |y|^2 }{ (1+|y|^2)^4}  \, d y . $
			For $\xi\in \partial \Sigma$, applying Lemma~\ref{lemb3} again, 
			\begin{eqnarray*}
				&&\varepsilon^2\int_{\Sigma\cap U_{2r_0}(\xi)}  e^{U_{\tau,\xi}} \Psi^1_{\tau,\xi} P\Psi^1_{\tau,\xi} \, dv_g(x)\\
				&=& \int_{ B_{2r_0}^\xi}\chi\left(\frac{|y|}{r_0}\right)  \frac{ 32\tau^2 \varepsilon^2y_1}{( \tau^2\varepsilon^2 +|y|^2)^3}
				\left(  \frac{ 4\chi\left(\frac{|y|}{r_0}\right) y_1}{ \tau^2\varepsilon^2+ |y|^2}  + \varrho(\xi) H^1( y_{\xi}^{-1}(y),\xi)  + {\mathcal{O}}(\varepsilon^{\alpha_0})  \right)\\
				&=&\frac{128}{\tau^2\varepsilon^2}  \int_{\frac 1 {\tau\varepsilon} B^{\xi}_{r_0} }  \frac{ y_1^2 }{(1+|y|^2)^4}  +{\mathcal{O}}(\varepsilon^{\alpha_0-1}).
			\end{eqnarray*}
			We observe that as $\varepsilon\rightarrow 0$
			\begin{equation*}
				\begin{array}{lcl}
					\left| \frac{128}{\tau^2\varepsilon^2}  \int_{\frac 1 {\tau\varepsilon} B^{\xi}_{r_0} }  \frac{ y_1^2 }{(1+|y|^2)^4} - \frac{128}{\tau^2\varepsilon^2}  \int_{\mathbb{R}^2_+}  \frac{ y_1^2 }{(1+|y|^2)^4}\right| &\leq & \frac{128}{\tau^2\varepsilon^2}  \int_{ \mathbb{R}^2_+\setminus \frac 1 {\tau\varepsilon} B^{\xi}_{r_0}}  \frac{ 1 }{(1+|y|^2)^3} \, d y\leq  \mathcal{O}(\varepsilon^2),
				\end{array}
			\end{equation*}
			and 
			\[  \varepsilon^2\int_{\Sigma\setminus U_{2r_0}(\xi) }\chi_{\xi}(x)e^{-\varphi_{\xi}(x)} e^{U_{\tau,\xi}} \Psi^i_{\tau,\xi} P\Psi^j_{\tau,\xi} \, dv_g={\mathcal{O}}(\varepsilon^2\|P\Psi^j_{\tau,\xi}\|)=\mathcal{O}(\varepsilon),  \]
			for $i,j=1,\cdots,\ii(\xi)$. 
			Thus, we have 
			$ \la P\Psi^i_{\tau,\xi},P\Psi^j_{\tau,\xi} \ra=	\frac{ 8\varrho(\xi)D_i}{\pi\tau^2\varepsilon^2}\delta_{ij} + \mathcal{O}(\varepsilon^{\alpha_0-1}).$
			By assumption, 
			$r_0>0$ sufficiently small such that 
			$ U_{2r_0}(\xi_0)\cap U_{2r_0}(\xi_1)=\emptyset,$
			and  for $l=0,1$, if $\xi_{l}\in \Sigma$, 
			$ U_{2r_0}(\xi_l) \subset \subset \Sigma.$
			\[ \langle P\Psi^i_{\tau^0,\xi_0}, P\Psi^j_{\tau^1,\xi_1}\rangle= \int_{\Sigma \setminus U_{2r_0}(\xi_0)} + \int_{\Sigma\cap U_{2r_0}(\xi_0)} \varepsilon^2 \chi_{\xi_0}e^{-\varphi_{\xi} } e^{U_{\tau^0,\xi_0}} \Psi^i_{\tau^0,\xi_0} P \Psi^j_{\tau^1,\xi_1} \, dv_g . \]
			As $\varepsilon\rightarrow 0$, we have
			\begin{eqnarray*}
				\int_{\Sigma \setminus U_{2r_0}(\xi_0)}  \varepsilon^2\chi_{\xi_0}(x) e^{-\varphi_{\xi}(x)} e^{U_{\tau^0,\xi_0}} \Psi^i_{\tau^0,\xi_0} P \Psi^j_{\tau^1,\xi_1} \, dv_g 
				= {\mathcal{O}}(\varepsilon^2 \|P\Psi^j_{\tau^1,\xi_1}\|)={\mathcal{O}}(\varepsilon). 
			\end{eqnarray*}
			By  Lemma~\ref{lemb3}, for $j\neq 0$ 
			\begin{equation*}
				\begin{array}{lcl}
					&&	\int_{ U_{2r_0}(\xi_0)\cap \Sigma}  \varepsilon^2 \chi_{\xi_0} e^{-\varphi_{\xi_0}}e^{U_{\tau^0,\xi_0}} \Psi^i_{\tau^0,\xi_0} P \Psi^j_{\tau^1,\xi_1}  \, dv_g \\
					&	= &   \int_{ U_{2r_0}(\xi_0) }   \varepsilon^2 \chi_{\xi_0} e^{-\varphi_{\xi_0}}e^{U_{\tau^0,\xi_0}} \Psi^i_{\tau^0,\xi_0} \left(  \chi_{\xi_1} \frac{ 4y_{\xi_1}(x)_j}{ \tau^2\varepsilon^2 +|y_{\xi_1}(x)|^2} + \varrho(\xi_1) H^j(x, \xi_1) +{ \mathcal{O}}(\varepsilon^{\alpha_0})\right)\\
					&=& \varrho(\xi_i)H^j(\xi_0,\xi_1)\int_{U_{2r_0}(\xi_0)}   \varepsilon^2 \chi_{\xi_0} e^{-\varphi_{\xi_0}}e^{U_{\tau^0,\xi_0}} \Psi^i_{\tau^0,\xi_0} \, dv_g\\
					&& +\mathcal{O}\left( \int_{ U_{2r_0}(\xi_0) }   \varepsilon^2 \chi_{\xi_0} e^{-\varphi_{\xi_0}}e^{U_{\tau^0,\xi_0}} \Psi^i_{\tau^0,\xi_0}(|y_{\xi_0}|+\varepsilon^{\alpha_0}) \, dv_g\right)= \mathcal{O}(\varepsilon^{\alpha_0-1});
				\end{array}
			\end{equation*}
			for $j=0$, 
			\begin{equation*}
				\begin{array}{lcl}
					&&	\int_{ U_{2r_0}(\xi_0)\cap \Sigma}  \varepsilon^2 \chi_{\xi_0} e^{-\varphi_{\xi_0}}e^{U_{\tau^0,\xi_0}} \Psi^i_{\tau^0,\xi_0} P \Psi^j_{\tau^1,\xi_1}  \, dv_g \\
					&	= &   \int_{ U_{2r_0}(\xi_0) }   \varepsilon^2 \chi_{\xi_0} e^{-\varphi_{\xi_0}}e^{U_{\tau^0,\xi_0}} \Psi^i_{\tau^0,\xi_0} \left(  -\chi_{\xi_1}\frac{ 4\tau\varepsilon^2 }{ \tau^2\varepsilon^2 +|y_{\xi_1}|^2}+\mathcal{O}(\varepsilon^{\alpha_0+1})\right)= \mathcal{O}(\varepsilon^{\alpha_0}).
				\end{array}
			\end{equation*}
			Therefore  for any $\xi_1\neq \xi_0$, 
			\[ \langle P\Psi^i_{\tau^0,\xi_0},  P\Psi^j_{\tau^1,\xi_1}\rangle=\begin{cases}
				\mathcal{O}(\varepsilon^{\alpha_0}) &\text{ when } i \text{ or } j=0\\
				\mathcal{O}(\varepsilon^{\alpha_0-1})& \text{ otherwise }
			\end{cases}.  \]
		\end{proof}
		{
			\begin{remark}~\label{rklem4}
				Analogue to the proof in Lemma~\ref{lem5}, for any $\alpha_0\in(0,1)$, we have as $\varepsilon\rightarrow 0$ for $j,i= 1,2$ for $\xi\in\intsigma$ and $i,j=0,1$ for $\xi\in\partial\Sigma$,
				\[ \langle P\Psi^i_{\tau,\xi}, \partial_{\xi_j}PU_{\tau,\xi}\rangle=
				\frac{ 8\varrho(\xi)D_i}{\pi\tau^2\varepsilon^2}\delta_{ij} + {\mathcal{O}}(\varepsilon^{\alpha_0-1}),
				\]
				and 
				\[ \langle P\Psi^i_{\tau^0,\xi_0},  \partial_{\xi_j}PU_{\tau^1,\xi_1}\rangle=\begin{cases}
					\mathcal{O}(\varepsilon^{\alpha_0}) &\text{ when } i \text{ or } j=0\\
					{\mathcal{O}}(\varepsilon^{\alpha_0-1})& \text{ otherwise }
				\end{cases}, \]
				where three different points  $\xi,\xi_0,\xi_{1}\in\Sigma$ and uniformly in  $\tau,\tau^0,\tau^1$ are bounded away from zero and the $\delta_{ij}$ is the Kronecker symbol, and  $D_0=\int_{\mathbb{R}^2}  \frac{1-|y|^2}{ (1+|y|^2)^4 }\, d y$, $D_1= D_2=\int_{\mathbb{R}^2} \frac{ |y|^2 }{ (1+|y|^2)^4}\, d y . $
			\end{remark}
		}
		In the remaining part, we  consider $\xi=(\xi_1,\cdots,\xi_{k+l})  \text{ in a compact subset of } \Xi^\prime_{k,l}$.
		Next, we give some technical lemmas to prove Proposition~\ref{thm2} which reduces the problem into a finite-dimensional problem. 
		\begin{lemma}~\label{lem5}
			Let $\xi=(\xi_1,\cdots,\xi_{k+l}) \in M_{\delta}$ (see~\eqref{eq:def_M_delta}).  
			For any $p\in [1,2)$, there is a positive constant $c:=c(p)$ such that for any $\varepsilon>0$,
			\[ 
			\left| \varepsilon^2 Ve^{\sum_{i=1}^{k+l} PU_i}-\varepsilon^2 \sum_{i=1}^{k+l} e^{-\varphi_i}\chi_{i}  e^{U_i}\right|_{L^p(\Sigma)}\leq c  \varepsilon^{\frac{2-p}{p}}.\] 
		\end{lemma}
		\begin{proof} Let $\mathcal{D}\subset \Xi^\prime_{k,l}$ be a compact subset. Then there exists $\delta>0$ such that $ \mathcal{D}\subset  M_{\delta}$. There is a uniform $r_0>0$ for any  $\xi\in M_{\delta}$.  By calculation, we deduce that 
			\begin{eqnarray*}
				\int_{\Sigma} \left|\varepsilon^2V e^{\sum_{i=1}^{k+l} PU_i } -\varepsilon^2\sum_{i=1 }^{k+l} e^{-\varphi_i} \chi_{i}  e^{U_i}\right|^p \, dv_g 
				& =& \sum_{i=1}^{k+l} 
				\int_{\Sigma\cap U_{2r_0}(\xi)}  \left|\varepsilon^2V e^{\sum_{i=1}^{k+l} PU_i } -\varepsilon^2\sum_{h=1 }^{k+l}e^{-\varphi_h}\chi_h e^{U_h} \right|^p \, dv_g\\
				&&+ \int_{\Sigma\setminus  \cup_{i=1}^{k+l} U_{2r_0}(\xi)} \left|\varepsilon^2V e^{\sum_{i=1}^{k+l} PU_i } -\varepsilon^2\sum_{h=1 }^{k+l} e^{-\varphi_h} \chi_he^{U_h}\right|^p \, dv_g,
			\end{eqnarray*}
			and 
			as $\varepsilon\rightarrow 0$, 
			$\int_{\Sigma \setminus \cup_{i=1}^{k+l} U_{2r_0}(\xi)} |\varepsilon^2V e^{\sum_{i=1}^{k+l} PU_i } -\varepsilon^2\sum_{h=1 }^{k+l}e^{-\varphi_h}\chi_h e^{U_h}|^p \, dv_g= {\mathcal{O}}(\varepsilon^{2p}).$
			By Lemma \ref{lemb1}, for any $x\in U_{2r_0}(\xi_h)$
			\begin{eqnarray*} ~\label{p1}
				&&\sum_{i=1}^{k+l} PU_i -\chi_h U_h+\varphi_h\\
				&=&{ \left( \sum_{i\neq h} \varrho(\xi_i) G^g(\xi_h, \xi_i)+ \varrho(\xi_h) H^g(\xi_h,\xi_h)  -\log (8\tau_h^2)\right)} +\mathcal{O}( |y_{\xi_h}|  +\varepsilon^{1+\alpha_0})\\
				&=& -\log V(\xi_h)+ \mathcal{O} (\varepsilon ^{1+\alpha_0}+|y_{\xi_h}|).
			\end{eqnarray*}
			Hence, for 	 $p\in [1, 2)$
			\begin{eqnarray*}
				&&\int_{ U_{2r_0}(\xi_h)\cap \Sigma} | \varepsilon^2V e^{\sum_{i=1}^{k+l} PU_i } -\varepsilon^2 e^{-\varphi_h}\chi_h e^{U_h}|^p  \, dv_g\\
				&=&  \int_{ U_{r_0}(\xi_h)\cap \Sigma} \left|  \varepsilon^2 e^{U_h} ( e^{ \sum_{i=1}^{k+l} PU_i - \chi_h U_h+ \varphi_h +\log V} -1 )\right|^p \, dv_g +\mathcal{ O}
				(\varepsilon^{2p})\\
				&=& \mathcal{O}\left( \int_{ U_{r_0}(\xi_h )\cap \Sigma} \varepsilon^{2p} e^{pU_h} (|y_{\xi_h}|+\varepsilon^{1+\alpha_0})^p \, dv_g\right) +\mathcal{ O}(\varepsilon^{2p}) \\
				&=&\mathcal{O}\left( \int_{ B_{r_0}^{\xi_h}} \left( \frac{8\tau_h^2\varepsilon^2( |y|+\varepsilon^{1+\alpha_0}) }{(\tau^2_h\varepsilon^2 +|y|^2)^2 }\right)^p  \, d y +\varepsilon^{2p}\right)=\mathcal{O}(\varepsilon^{2-p}).
			\end{eqnarray*}
		\end{proof}
		\begin{lemma}~\label{lem6}
			For any $p\geq 1$ and $r>1$, there are positive constants $c_1, c_2$ such that for any $\varepsilon>0$, the following estimates hold for any $\phi_1,\phi_2\in \oH$. 
			\begin{equation}~\label{diff1}
				\| \varepsilon^2 V e^{\sum_{i=1}^{k+l} PU_i }(e^{\phi_1} -1-\phi_1)\|_{p}\leq c_1 e^{c_2 \|\phi_1\|^2} \varepsilon^{\frac{(2-2pr)}{pr}}  \|\phi_1\|^2,
			\end{equation}
			and 
			\begin{eqnarray}
				~\label{diff12}
				&&\| \varepsilon^2 V e^{\sum_{i=1}^{k+l} PU_i }(e^{\phi_1}-e^{\phi_2} -(\phi_1-\phi_2))\|_{p}\leq  c_1 e^{c_2( \|\phi_1\|^2+\|\phi_2\|^2) } \varepsilon^{\frac{(2-2pr)}{pr}} (\|\phi_1\|+\|\phi_2\|) \|\phi_1-\phi_2\|. \nonumber
			\end{eqnarray}
		\end{lemma}
		\begin{proof}
			By the mean value theorem,  for some $s\in (0,1)$
			\[ |(e^{\phi_1}-e^{\phi_2} -(\phi_1-\phi_2)|\leq \left| e^{s\phi_1+(1-s)\phi_2} -1\right||\phi_1-\phi_2|\leq  e^{|\phi_1|+|\phi_2|} |\phi_1-\phi_2|(|\phi_1|+|\phi_2|).\]
			By applying the H\"{o}lder Inequality, Sobolev Inequality, and Moser-Trudinger Inequality, we derive the following estimate:
			\begin{eqnarray*}
				&&	\left( \int_{\Sigma} V^pe^{p \sum_{i=1}^{k+l}PU_{i}} | e^{\phi_1} -e^{\phi_2}-(\phi_1-\phi_2)| ^p \, dv_g\right)^{1/p} \\
				&\leq& C \sum_{h=1}^2
				\left(  \int_{\Sigma}
				V^p	e^{p \sum_{i=1}^{k+l}PU_{i}}(  e^{|\phi_1|+|\phi_2|}  |\phi_1-\phi_2||\phi_h| ) ^p \, dv_g \right)^{1/p}  \\
				&\leq &  C \sum_{h=1}^2 \left(\int_{\Sigma} V^{pr}e^{pr\sum_{i=1}^{k+l} PU_{i}} \, dv_g\right)^{\frac{1}{pr}} \left( \int_{\Sigma}   e^{ps (|\phi_1|+|\phi_2|)} \, dv_g\right)^{\frac{1}{ps}}\left(\int_{\Sigma} |\phi_1-\phi_2|^{pt} |\phi_h|^{pt} \, dv_g\right)^{\frac{1}{pt}} \\
				&\leq & C \sum_{h=1}^2 \left(\int_{\Sigma} V^{pr}e^{pr\sum_{i=1}^{k+l} PU_{i}} \, dv_g(x)\right)^{\frac{1}{pr}} e^{\frac{ps}{8\pi}(\|\phi_1\|^2+\|\phi_2\|^2)} \|\phi_1-\phi_2\| \|\phi_h\|,
			\end{eqnarray*}
			where $ r,s,t \in (1, +\infty), {  \frac{1}{r}+\frac{1}{s}+\frac{1}{t}=1}$. 
			By Lemma~\ref{lemb1}, it follows that 
			\begin{eqnarray*}
				&&	\int_{\cup_{i=1}^{k+l} U_{2r_0}(\xi_i)} V^{pr} e^{pr\sum_{i=1}^{k+l} PU_{i}} \, dv_g\\
				&= & \sum_{i=1}^{k+l} 	\int_{U_{2r_0}(\xi_i)} \exp\left\{
				pr\chi_iU_i + pr\left( \sum_{h\neq i} G^g(\xi_i,\xi_h) +\varrho(\xi_i) H^g(\xi_i,\xi_i)\right. \right. \\
				&& \left.\left.+ \log V(\xi_i) -\log( 8\tau_i^2)\right) + \mathcal{O}( \varepsilon^{1+\alpha_0} +|y_{\xi_i}|)\right \} \, dv_g\\
				&  \leq &      C   \left( \sum_{i=1}^{k+l} \int_{U_{2r_0}(\xi_i)} e^{pr\chi_i U_i} ( 1 + { \mathcal{O}}(\varepsilon^{1+\alpha_0}+ |y_{\xi_i}(x)|)) \, dv_g(x) \right)\\
				&\leq &
				C   \left( \sum_{i=1}^{k+l} \int_{B^{\xi_i}_{2r_0}}e^{\hat{\varphi}_{\xi_i}(y
					)}  \left(  \frac{ 8\tau_i^2}{(\tau_i^2\varepsilon^2+ |y|^2)^2}\right)^{pr}( 1 + \mathcal{O}(\varepsilon^{1+\alpha_0}+ |y|)) \, d y \right)\\
				&\leq& C\varepsilon^{2-4pr}.
			\end{eqnarray*}
			By the definition of $PU_i$, $PU_i=\mathcal{O}(1)$  in $\Sigma\setminus U_{2r_0}(\xi_i)$. It follows that 
			\[ \sum_{\Sigma\setminus \cup_{i=1}^{k+l} U_{2r_0}(\xi_i) } e^{pr \sum_{i=1}^{k+l} PU_i} =\mathcal{ O}(1). \]
			Therefore, the estimate~\eqref{diff12} holds and if we take $\phi_2\equiv 0$, we obtain the estimate ~\eqref{diff1}. 
		\end{proof}

		Next, we will give some technique lemmas to obtain the $C^1$-expansion of the reduced functional $\tilde{E}_{\varepsilon}$ defined by~\eqref{eq:def_reduced_tilde_E}. 
		\begin{lemma}~\label{lem7}
			As $\varepsilon\rightarrow 0$, the following asymptotic expansions hold
			\begin{equation*}
				\begin{array}{lcl}
					\lan PU_i,PU_i\ran&=&\varrho(\xi_i) (6\log 2 -4\log \varepsilon-2\log(8\tau_i^2)+\varrho(\xi_i)H^g(\xi_i,\xi_i)-2) \\
					&&+\mathcal{O}(\varepsilon|\log\varepsilon|),
				\end{array} 
			\end{equation*}
			and for any $i\neq j$,
			$
			\lan PU_i, \nabla PU_j\ran=  \varrho(\xi_i)\varrho(\xi_j) G^g(\xi_i,\xi_j) +\mathcal{O}(\varepsilon). $
		\end{lemma}
		\begin{proof}
			Applying Lemma~\ref{lemb1} with~\eqref{tau}, we drive that  as $\varepsilon\rightarrow 0$
			\begin{eqnarray*}
				&&\lan PU_i,PU_i\ran=\int_{\Sigma} |\nabla PU_i|_g^2 +\beta |PU_i|^2 \, dv_g= \varepsilon^2 \int_{\Sigma}\chi_i e^{-\varphi_i} e^{U_i} PU_i  \, dv_g \\
				&=& \int_{ U_{r_0}(\xi_i)} \frac{8\tau_i^2\varepsilon^2}{ (\tau_i^2\varepsilon^2+|y_{\xi_i}|^2)^2 } e^{-\varphi_i} \left(\log \frac 1{(\tau_i^2\varepsilon^2 +|y_{\xi_i}|^2)^2 } +\varrho(\xi_i) H^g(\xi_i,\xi_i) \right. \\
				&&\left. +{\mathcal{O}}(|y_{\xi_i}|+\varepsilon^{1+\alpha_0}  )\right) \, dv_g
				+{\mathcal{O}}(\varepsilon^2) \\
				&=& \int_{ B_{r_0}^{\xi_i}} \frac{8\tau_i^2\varepsilon^2}{ (\tau_i^2\varepsilon^2+|y|^2)^2 }  \left(\log \frac { \tau_i^4\varepsilon^4}{(\tau_i^2\varepsilon^2 +|y|^2)^2 }-2\log (\tau_i^2\varepsilon^2) +\varrho(\xi_i) H^g(\xi_i,\xi_i) \right. \\
				&&\left.  +{\mathcal{O}}(|y|+\varepsilon^{1+\alpha_0}  )\right) \, d y
				+{\mathcal{O}}(\varepsilon^2) \\
				&=& \varrho(\xi_i) (6\log 2 -4\log \varepsilon-2\log(8\tau_i^2)+\varrho(\xi_i)H^g(\xi_i,\xi_i)-2) +\mathcal{O}(\varepsilon|\log\varepsilon|), 
			\end{eqnarray*}
			where we applied the fact that for any $r>0$, as $\varepsilon\rightarrow 0$,
			$	\int _{|y|<r} \frac{\varepsilon^2}{                                                                                              ( \varepsilon^2+ |y|^2)^2} \, d y =\pi- \frac{\pi\varepsilon^2}{r^2}+ \frac{\pi\varepsilon^4}{(r^2+\varepsilon^2)r^2}$
			and 
			$\int_{|y|<r} \frac{ \varepsilon^2\log( \frac{\varepsilon^2+ |y|^2}{\varepsilon^2} )}{ (\varepsilon^2+|y|^2)^2} \, d y=\pi + \frac{\pi\varepsilon^2\log(\varepsilon^2)}{r^2}+ \mathcal{O}(\varepsilon^2).$
			For any $i\neq j$,  Lemma~\ref{lemb1} yields as $\varepsilon\rightarrow 0$
			\begin{eqnarray*}
				&&	\lan PU_i, PU_j\ran= \varepsilon^2 \int_{\Sigma} \chi_{i} e^{-\varphi_i}e^{U_i} PU_j \, dv_g \\
				&=& \int_{U_{2r_0}(\xi_i)} \frac{8\tau_i^2 \varepsilon^2}{(\tau^2_i \varepsilon^2 +|y_{\xi_i}(x)|^2)^2}e^{-\varphi_i(x)}( \varrho(\xi_j) G^g(\xi_i,\xi_j)+{\mathcal{O}}(|y_{\xi_i}(x)|+\varepsilon^{1+\alpha_0}))+{\mathcal{O}} (\varepsilon ^2 \|PU_j\|)\\
				&= & 8\varrho(\xi_j) G^g(\xi_i,\xi_j) \int_{B^{\xi_i}_{2r_0}} \frac{\tau_i^2\varepsilon^2}{(\tau_i^2\varepsilon^2+|y|^2)^2} \, d y +{\mathcal{O}}(\varepsilon) 
				=\varrho(\xi_i)\varrho(\xi_j) G^g(\xi_i,\xi_j) +\mathcal{O}(\varepsilon).
			\end{eqnarray*}
		\end{proof}
		\begin{lemma}~\label{lem8}
			For any $m\in \N_+$ and $k,l\in \N$ with $m=2k+l$, we have  as  $\varepsilon\rightarrow 0
			$
			$$
			\varepsilon^2\int_{\Sigma} Ve^{\sum_{i=1}^{k+l} PU_i } =\sum_{i=1}^{k+l} \varrho(\xi_i)+o(1)= 4\pi m +o(1).$$ 
		\end{lemma}
		\begin{proof}
			Applying Lemma~\ref{lemb1} and~\eqref{tau}, as  $\varepsilon\rightarrow 0$
			\begin{eqnarray*}
				&&\varepsilon^2\int_{\Sigma} Ve^{\sum_{i=1}^{k+l} PU_i} \, dv_g\\
				&=&\sum_{i=1}^{k+l} \varepsilon^2\int_{ U_{2r_0}(\xi_i)} e^{ \chi_i U_i +\varrho(\xi_i) H^g(\cdot,\xi_i) -\log 8\tau_i^2 +\sum_{j\neq i } \varrho(\xi_j) G^g(\cdot,\xi_j)+ {\mathcal{O}}(\varepsilon^{1+\alpha_0})} \, dv_g+{\mathcal{O}}(\varepsilon^2 )\\
				&=&\sum_{i=1}^{k+l} \int_{U_{r_0}(\xi_{i})} \frac{8\tau^2_i \varepsilon^2  e^{\varrho(\xi_i) H^g(\xi_i,\xi_i)-\log (8\tau_i^2)+\log V(\xi_i) + \sum_{j\neq i}\varrho(\xi_j) G^g(\xi_i,\xi_j)} }{ (\tau_i^2\varepsilon^2+ |y_{\xi}(x)|^2)^2} \\
				&&  (1+{\mathcal{O}}(|y_{\xi}| +\varepsilon^{1+\alpha_0})) \, dv_g  +{\mathcal{O}}(\varepsilon^2)\\
				&=& \sum_{i=1}^{k+l} \int_{ B_{r_0}^{\xi_i} } \frac{8 \tau_i^2\varepsilon^2  e^{\hat{\varphi}_i(y)}}{( \tau_i^2\varepsilon^2+|y|^2)^2} (1+{\mathcal{O}}(|y|+\varepsilon^{1+\alpha_0})  ) \, d y +{\mathcal{O}}(\varepsilon^2)\\
				&=&  \sum_{i=1}^{k+l} \int_{ \frac 1 {\tau_i\varepsilon}B_{r_0}^{\xi_i} } (1+\mathcal{O}(\varepsilon|y|))(1+\mathcal{O}(\varepsilon|y|+\varepsilon^{1+\alpha_0}))\frac{8 }{(1+|y|^2)^2} \, d y +{\mathcal{O}}(\varepsilon^2)\\
				&=&       \sum_{i=1}^{k+l} \varrho(\xi_i)   +{\mathcal{O}}(\varepsilon). 
			\end{eqnarray*}
		\end{proof}
		{
			\begin{lemma}~\label{lem9}
				Let $i,h=1,\cdots, k+l $ and $j=1,\cdots,\ii(\xi_i)$.
				Then, as $\varepsilon\rightarrow 0$,
				\begin{eqnarray*}
					&&\varepsilon^2\int_{\Sigma}  e^{-\varphi_h} \chi_h e^{U_h} \partial_{(\xi_i)_j} PU_i \, dv_g\\
					&=&\frac {\delta_{ih}} 2\varrho(\xi_i)^2\partial_{(\xi_i)_j} H^g(\xi_i,\xi_i) + (1-\delta_{ih}) \varrho(\xi_i)\varrho(\xi_h) \partial_{(\xi_i)_j} G^g(\xi_h,\xi_i) +o(1),
				\end{eqnarray*}
				where $\delta_{ih}=1$ if $i=h$; 0 if $i\neq h.$
			\end{lemma}
		}  
		\begin{proof} We decompose the integral into the following two parts: 
			$$\varepsilon^2 \int_{\Sigma}e^{-\varphi_h}\chi_h e^{U_h} \partial_{(\xi_i)_j}PU_i=\varepsilon^2 \left(\int_{\Sigma\cap U_{2r_0}(\xi_h)}  +\int_{\Sigma\setminus U_{2r_0}(\xi_h )}\right) e^{-\varphi_h} \chi_he^{U_h}\partial_{(\xi_i)_j}PU_i.$$  
			It is clear that 
			$
			\int_{\Sigma\setminus U_{2r_0}(\xi_h)} \varepsilon^2e^{-\varphi_h}\chi_h e^{U_h}\partial_{(\xi_i)_j}PU_i
			=0.
			$
			For    $h\neq i$,
			$U_{2r_0}(\xi_h)\cap U_{2r_0}(\xi_i)=\emptyset$ by the choice of $r_0$. Notice that as $|y|\rightarrow 0$	\begin{eqnarray*}
				\partial_{(\xi_i)_j} |y_{\xi}(x)|^2|_{x=y_{\xi_i}^{-1}(y)}&=& -2\lan (y_{\xi_i})_* \partial_{(\xi_i)_j} y^{-1}_{\xi}(y), y\ran=-2y_j+\mathcal{O}(|y|^3).
			\end{eqnarray*}
			\begin{claim}\label{claim:12}
				As $\varepsilon\rightarrow 0$,\begin{eqnarray*}
					\int_{ U_{2r_0}(\xi_i)}   \varepsilon^2e^{-\varphi_i}\chi_{i} e^{U_i} \frac{2\partial_{(\xi_i)_j} |y_{\xi_i}|^2  }{ \tau_i^2\varepsilon^2 +|y_{\xi_{i}}|^2} \, dv_g&=&{\mathcal{O}}(\varepsilon^2)+ 	\int_{ U_{r_0}(\xi_i)}   \varepsilon^2 e^{-\varphi_i}e^{U_i} \frac{4(-(y_{\xi_i})_j+\mathcal{O}(|y_{\xi_i}|^3)) }{ \tau_i^2\varepsilon^2 +|y_{\xi_{i}}|^2} \, dv_g\\
					&=& o(1). 
				\end{eqnarray*}
			\end{claim}
			{\it Indeed, } as $ |y|\rightarrow 0, $
			\begin{eqnarray*}
				\int_{ U_{r_0}(\xi_i)\cap \Sigma}   \varepsilon^2 e^{-\varphi_i}\chi_{i}e^{U_i} \frac{2\partial_{(\xi_i)_j} |y_{\xi_i}|^2 }{ \tau_i^2\varepsilon^2 +|y_{\xi_{i}}|^2} \, dv_g  &=& \int_{B_{r_0}^{\xi_i}} \varepsilon^2 \frac{ 32\tau_i^2\varepsilon^2 (-y_j +\mathcal{O}(|y|^2)) }{( \tau_i^2\varepsilon^2 +|y|^2)^3} \, d y+\mathcal{O}(\varepsilon^2) \\
				&= &\int_{B_{r_0}^{\xi_i}} \varepsilon^2 \frac{- 32\tau_i^2 y_j+ \mathcal{O}(|y|^3)}{( \tau_i^2\varepsilon^2 +|y|^2)^3} \, d y ={\mathcal{O}}(\varepsilon). 
			\end{eqnarray*} 
			Claim~\ref{claim:12}  is concluded. 
			By Remark \ref{rklemb3}, 
			\begin{eqnarray*}
				&&	\int_{\Sigma} \varepsilon^2 \chi_{i}e^{U_i}\partial_{(\xi_i)_j}PU_i \, dv_g\\
				&=&  \int_{ \Sigma} \frac{8\tau_i^2\varepsilon^2\chi_{i}}{(\tau_i^2\varepsilon^2 +|y_{\xi_{i}}|^2)^2} \left(\chi_{i}  \frac{2\partial_{(\xi_i)_j} |y_{\xi}|^2}{ \tau_i^2\varepsilon^2 +|y_{\xi_i}|^2 } +\varrho(\xi_i) \partial_{(\xi_i)_j}  H^g_{\xi_i} +{\mathcal{O}}(\varepsilon^{\alpha_0})  \right) \, dv_g\\
				&=& \int_{U_{r_0}(\xi_i)} \varepsilon^2 \chi_{i}(x) e^{U_i(x)} \frac{2\partial_{(\xi_i)_j} |y_{\xi}(x)|^2}{ \tau_i^2\varepsilon^2+ |y_{\xi_i}(x)|^2} \, dv_g(x) \\
				&& +\frac 1 2  \varrho(\xi_i)\partial_{(\xi_i)_j }H^g(\xi_i,\xi_{i})\int_{U_{r_0}(\xi_i)} \frac{ 8\tau_i^2\varepsilon^2}{(\tau_i^2\varepsilon^2+ |y_{\xi_i}(x)|^2)^2} \, dv_g(x)+ {\mathcal{O}}(\varepsilon^{\alpha_0})\\
				&=&\frac 1 2 \varrho(\xi_i)^2\partial_{(\xi_i)_j} H^g(\xi_i,\xi_i) +o(1).
			\end{eqnarray*}
			For $i\neq h$, via Lemma \ref{lemb2}, we drive that 
			\begin{eqnarray*}
				&&	\int_{ U_{2r_0}(\xi_h)\cap \Sigma} \varepsilon^2 \chi_he^{U_h}\partial_{(\xi_i)_j}PU_i \, dv_g\\
				&=&  \int_{ U_{2r_0}(\xi_h)\cap \Sigma}\frac{8\tau_h^2\varepsilon^2\chi_h}{(\tau_h^2\varepsilon^2 +|y_{\xi_h}|^2)^2} \left(\chi_i\frac{  2\partial_{(\xi_i)_j}|y_{\xi_i}|^2}{\tau_i^2\varepsilon^2+|y_{\xi_i}|^2} +\varrho(\xi_i) \partial_{(\xi_i)_j}H^g_{\xi_i} +{\mathcal{O}}(\varepsilon^{\alpha_0})
				\right) \, dv_g\\
				&=&  \int_{ U_{2r_0}(\xi_h)\cap \Sigma}\chi_h(x) \frac{8\tau_h^2\varepsilon^2}{(\tau_h^2\varepsilon^2 +|y_{\xi_h}|^2)^2} \left( \varrho(\xi_i) \partial_{(\xi_i)_j} G^g(\cdot,\xi_i) +{\mathcal{O}}(\varepsilon^{\alpha_0})
				\right) \, dv_g\\
				&=& \varrho(\xi_i)\varrho(\xi_h) \partial_{(\xi_i)_j}G^g(\xi_h,\xi_i)+{\mathcal{O}}(\varepsilon^{\alpha_0}). 
			\end{eqnarray*}
			Combining all the estimates above, Lemma~\ref{lem9} is concluded. 
		\end{proof}
		\begin{lemma}~\label{lem10} Let $i=1,\cdots, k+l $ and $j=1,\cdots,\ii(\xi_i)$. As $\varepsilon\rightarrow 0$, 
			\begin{eqnarray*}
				&&	\varepsilon^2 \int_{\Sigma}Ve^{\sum_{h=1}^{k+l} P U_h}\partial_{(\xi_i)_j}PU_i \, dv_g =\frac  1 2 	\partial_{(\xi_i)_j}\cF^V_{k,l}(\xi) +o(1). 
			\end{eqnarray*}
		\end{lemma}
		\begin{proof} First, we divide the integral into three parts to calculate:
			\begin{eqnarray*}
				&&	\varepsilon^2 \int_{\Sigma} Ve^{\sum_{h=1}^{k+l} P U_h} \partial_{(\xi_i)_j}PU_i \, dv_g \\
				&=&\varepsilon^2 \left(\int_{\Sigma\setminus \cup_{h=1}^{k+l} U_{2r_0}(\xi_h )} +\int_{U_{2r_0}(\xi_i)}+\int_{\cup_{l\neq i}  U_{2r_0}(\xi_l)} \right) Ve^{\sum_{h=1}^{k+l} P U_h}\partial_{(\xi_i)_j}PU_i \, dv_g\\
				&:=& I^1+I^2+I^3.
			\end{eqnarray*}
			
			The first term $I^1$ can be easily estimated by Remark \ref{rklemb3}.  As $\varepsilon\rightarrow 0$, we have 
			\begin{eqnarray*}
				I^1
				&=& {\mathcal{O}} \left(\varepsilon^2\int_{\Sigma \setminus \cup _{h=1}^{k+l} U_{2r_0}(\xi_h)} \left|\partial_{\xi_j}(\chi_{\xi}  U_{\tau,\xi}) +\varrho(\xi) \partial_{\xi_j} H^g_{\xi} +\mathcal{O}(\varepsilon^{\alpha_0}) \right|\, dv_g\right)\\
				& =& {\mathcal{O}}(\varepsilon^2). 
			\end{eqnarray*}
			We observe that for any $i=1,\cdots, k+l $ and $j=1,\cdots,\ii(\xi_i)$, as $|y|\rightarrow 0$,
			$\partial_{(\xi_i)_j} H^g(\xi_i,\xi_i)=2 \partial_{x_j} H^g(x,\xi_i) |_{x=\xi_i},$
			$ e^{\hat{\varphi}_i(y)}=\begin{cases}
				1+\mathcal{O}(|y|^2) & \xi_i\in \intsigma\\
				1- 2k_g(\xi_i) y_2+\mathcal{O}(|y|)& \xi_i \in \partial\Sigma
			\end{cases},
			$
			and 
			$
			\partial_{(\xi_i)_j} |y_{\xi}(x)|^2|_{x=y_{\xi_i}^{-1}(y)}=-2y_j+\mathcal{O}(|y|^3). 
			$
			Applying Lemma~\ref{lemb1} and Remark~\ref{rklemb3} with~\eqref{tau}, we derive that 
			\begin{eqnarray*}
			I^2
				&=&\int_{U_{2r_0}(\xi_i)} \left(\frac{ \varepsilon^2 Ve^{\varrho(\xi_i) H^g_{\xi_i}+\sum_{l\neq i} \varrho(\xi_l)G^g(\cdot,\xi_l) +{\mathcal{O}}(\varepsilon^{1+\alpha_0})}  }{(\tau_i^2\varepsilon^2 +|y_{\xi_i}|^2)^2}\right)\\
				&& 
				\left(-\frac{2\chi_i \partial_{(\xi_i)_j} |y_{\xi_i}|^2 } {(\tau_i^2\varepsilon^2+|y_{\xi_i}|^2)} +\varrho(\xi_i)\partial_{(\xi_i)_j} H^g_{\xi_i}+ {\mathcal{O}}(\varepsilon^{\alpha_0})\right) \, dv_g \\
				&=& 	\int_{ B_{r_0}^{\xi_i}} \frac{8\tau_i^2\varepsilon^2e^{\hat{\varphi}_{\xi_i}(y)} }{(\tau_i^2\varepsilon^2 +|y|^2)^2}  \exp\left\{\varrho(\xi_i)H^g(y_{\xi_i}^{-1}(y),\xi_i)+\sum_{h\neq i}\varrho(\xi_h) G^g(y_{\xi_i}^{-1}(y),\xi_h)\right. \\
				&&\left. + \log V(y_{\xi_i}^{-1}(y))-\log(8\tau_i^2) +{\mathcal{O}}(\varepsilon ^{1+\alpha_0} )\right\}\left(\left. \frac{-2\partial_{(\xi_i)_j} |y_{\xi_i}(x)|^2 }{(\tau_i^2\varepsilon^2+|y_{\xi_i
					}(x)|^2)} \right|_{x=y_{\xi_i}^{-1}(y)}\right. \\
				&&\left.+ \frac 1 2\varrho(\xi_i) \partial_{(\xi_i)_j} H^g(\xi_i,\xi_i)+ {\mathcal{O}}(|y|+\varepsilon^{\alpha_0})\right) \, d y+{\mathcal{O}}(\varepsilon^2)\\
				&=&\int_{\frac 1 {\tau_i\varepsilon} B^{\xi_i}_{r_0}} 
				\frac 8 {(1 +|y|^2)^2}(1+ \nabla \hat{\varphi}_i(0)\cdot y+\mathcal{O}(\varepsilon^2|y|)^2)\left( 1+\frac 1 2 \tau_i\varepsilon\sum_{s=1}^2\varrho(\xi_i)\partial_{(\xi_i)_s}H^g(\xi_i,\xi_i)y_s\right. \\
				&&\left. + \tau_i\varepsilon\sum_{h\neq i} \varrho(\xi_h) \sum_{s=1}^2 \partial_{(\xi_i)_s}G^g(\xi_i,\xi_h) y_s +\tau_i\varepsilon\sum_{s=1}^2 \partial_{(\xi_i)_s}\log V(\xi_i) y_s +{\mathcal{O}}( \tau_i^2\varepsilon^2 |y|^2+\varepsilon^{1+\alpha_0} )   \right)   \\
				&& \cdot \left( \frac 1 {\tau_i \varepsilon}\frac{ 4y_j}{1+|y|^2}+\frac {\varrho(\xi_i)} 2 \partial_{(\xi_i)_j} H^g(\xi_i,\xi_i) +{\mathcal{O}}(\varepsilon|y|+\varepsilon^{\alpha_0})\right) \, d y +{\mathcal{O}}(\varepsilon^2)\\
				&=&\frac 1 2 \varrho(\xi_i)^2  \partial_{(\xi_i)_j} H^g(\xi_i,\xi_i) + \frac 1 2 \varrho(\xi_i)^2 \partial_{(\xi_i)_j} H^g(\xi_i,\xi_i)\\
				&&+ \sum_{h\neq i}\varrho(\xi_i)\varrho(\xi_h)\partial_{(\xi_i)_j} G^g(\xi_i,\xi_h)+\varrho(\xi_i) \partial_{(\xi_i)_j} \log V(\xi_i)+o(1)\\
				&=&  \varrho(\xi_i)^2 \partial_{(\xi_i)_j} H^g(\xi_i,\xi_i)+ \sum_{h\neq i}\varrho(\xi_i)\varrho(\xi_h)\partial_{(\xi_i)_j} G^g(\xi_i,\xi_h)+\varrho(\xi_i) \partial_{(\xi_i)_j} \log V(\xi_i)+o(1),
			\end{eqnarray*}
			where we applied 
			$ \int_{\mathbb{R}^2 }\frac{1}{(1+|y|^2)^2} \, d y =\pi = 2\int_{\mathbb{R}^2} \frac{|y|^2}{(1+|y|^2)^3 } \, d y.$
			
			For any $h\neq i$,  analogue to the proof for $h=i$, we can obtain 	\begin{eqnarray*}
				\int_{U_{2r_0}(\xi_h)} \varepsilon^2 Ve^{\sum_{l=1}^{k+l} P U_l}\partial_{(\xi_i)_j}PU_i \, dv_g=\varrho(\xi_i)\varrho(\xi_h)\partial_{(\xi_i)_j} G^g(\xi_h,\xi_i)+o(1).
			\end{eqnarray*}
			Combining the estimates above, 
			\begin{eqnarray*}
				&&	\varepsilon^2 \int_{\Sigma} Ve^{\sum_{h=1}^{k+l} P U_h}\partial_{(\xi_i)_j}PU_i \, dv_g= \partial_{(\xi_i)_j}\cF^V_{k,l}(\xi)+o(1).
			\end{eqnarray*}
		\end{proof}
		\begin{lemma}~\label{lem13}
			Let  $i, h=1, \cdots, k+l.$. Then as $\varepsilon \rightarrow 0$,
			$$
			\left\|\varepsilon^{2}\chi_h {e}^{U_{h}}\left( \partial_{(\xi_i)_j}PU_i-\chi_{i}\partial_{(\xi_i)_j} U_i\right)\right\|_{p}\leq O\left(\varepsilon^{\frac{2(1-p)} p}\right).
			$$
		\end{lemma}
		\begin{proof}
			By Remark~\ref{rklemb3}, $ \partial_{(\xi_i)_j}PU_i-\chi_{i}\partial_{(\xi_i)_j} U_i=\mathcal{O}(1)$. Then, applying Lemma~\ref{lemb2},
			$$
			\left\|\varepsilon^{2} \chi_h {e}^{U_{h}}\left( \partial_{(\xi_i)_j}PU_i-\chi_i \partial_{(\xi_i)_j} U_i \right)\right\|_p\leq O\left(\left\|\varepsilon^{2}\chi_h {e}^{U_{h}}\right\|_{p}\right)=O\left(\varepsilon^{\frac{2(1-p)} p}\right).
			$$
		\end{proof}
		\begin{lemma}~\label{lem11}
			Given $\delta>0$ sufficiently small, let $\xi=(\xi_1,\cdots,\xi_{k+l})\in M_{\delta}$. 
			Let $ \phi\in K_{\xi}^{\perp} \text{ and }\|\phi\|\leq {\mathcal{O}}(\varepsilon^{\frac{2-p}{p}}|\log \varepsilon|)$, where $p\in (1,\frac 6 5)$. Then for $i=1,\cdots, k+l $ and $j=1,\cdots, \ii(\xi_i)$, as $ \varepsilon\rightarrow 0$, 
			\begin{equation}
				\left	\lan \sum_{h=1}^{k+l} PU_h +\phi -i^*(\varepsilon^2 Ve^{\sum_{h=1}^{k+l} PU_h+\phi}),\partial_{(\xi_i)_j}PU_i\right\ran = -\frac 1 2\frac{\partial  \cF^V_{k,l}}{\partial(\xi_i)_j } (\xi)+o(1), 
			\end{equation}
			which is uniformly convergent for $\xi$ in  $M_{\delta}$.
		\end{lemma}
		\begin{proof}
			For $y= y_{\xi_i}(x)$, 
			$	\partial_{(\xi_i)_j} |y_{\xi_i}(x)|^2 = -2y_j  + \mathcal{O}(|y|^3). 
			$
			Since $\|\phi\|=o(1)$ and $\lan P\Psi^i_j,\phi\ran=0$, we have  
			\begin{eqnarray}
				\label{eq:phi_in_PPsi}
				&&\quad\left\lan \phi, \partial_{(\xi_i)_j}PU_i\right\ran = \int_{\Sigma} \varepsilon^2e^{-\varphi_i} e^{U_i}\phi\partial_{(\xi_i)_j} \chi_i  \, dv_g + \int_{\Sigma} \varepsilon^2e^{-\varphi_i}  \chi_i  e^{U_i}\phi \partial_{(\xi_i)_j} U_i \, dv_g \\
				&&+\int_{\Sigma} \varepsilon^2 \chi_i e^{U_i}\phi\partial_{(\xi_i)_j} e^{-\varphi_i} \, dv_g\nonumber \\
				&=&\int_{\Sigma} \varepsilon^2\chi_ie^{-\varphi_i} e^{U_i}\phi\Psi^j_i \, dv_g + \mathcal{O}\left( 	\int_{B_{2r_0}^{\xi_i}} \frac{\tau_i^2\varepsilon^2\chi\left(\frac{|y|}{r_0}\right) ( \tau_i^2\varepsilon^2|y|^2+|y|^4+|y|^3) }{(\tau_i^2\varepsilon^2+|y|^2)^3}|\phi| \, d y \right)\nonumber\\
				&=& \lan\phi, P\Psi^j_i\ran +o(1)=o(1),\nonumber
			\end{eqnarray}
			for any $i=1,\cdots,m$ and $j=1,\cdots,\ii(\xi_i).$
			Considering that $\int_{\Sigma} \partial_{(\xi_i)_j}PU_i \, dv_g=0$ and $\chi_i \cdot\chi_h \equiv 0$ for any $i\neq h,$ we have 
			\begin{eqnarray*}
				&&\left\lan \sum_{h=1}^{k+l}  PU_h +\phi -i^*(\varepsilon^2 Ve^{\sum_{h=1}^{k+l} PU_h+\phi}),\partial_{(\xi_i)_j}PU_i\right\ran \\
				&=& \sum_{h=1}^{k+l} \lan PU_h, \partial_{(\xi_i)_j}PU_i\ran+ \lan\phi, \partial_{(\xi_i)_j}PU_i\ran-\varepsilon^2\int_{\Sigma} Ve^{\sum_{h=1}^{k+l} PU_h +\phi }\partial_{(\xi_i)_j}PU_i \, dv_g
				\\
				&\stackrel{\eqref{eq:phi_in_PPsi}}{=}&\sum_{h=1}^{k+l}\int_{\Sigma} \varepsilon^2   \chi_h e^{-\varphi_h} e^{U_h}\partial_{(\xi_i)_j}PU_i \, dv_g - \varepsilon^2 \int_{\Sigma} V
				e^{\sum_{h=1}^{k+l}  PU_h} (e^{\phi} -\phi-1)\partial_{(\xi_i)_j}PU_i \, dv_g\\ 
				&& -\varepsilon^2 \int_{\Sigma}\left(Ve^{\sum_{h=1}^{k+l} PU_h}-\sum_{h=1}^{k+l} \chi_h e^{U_h} \right) \phi \partial_{(\xi_i)_j} PU_i
				\, dv_g\\
				&&+\sum_{h\neq i} \varepsilon^2\int_{\Sigma} \chi_h e^{ U_h} \phi \chi_i(\partial_{(\xi_i)_j} U_i -\chi_i \partial_{(\xi_i)_j} U_i) \, dv_g	\\
				&& -\varepsilon^2 \int_{\Sigma} Ve^{\sum_{h=1}^{k+l} PU_h}\partial_{(\xi_i)_j}PU_i \, dv_g +o(1).
			\end{eqnarray*}
			By Lemma~\ref{lem4} and Lemma~\ref{lem6}, we have 
			\begin{eqnarray*}
				&&\left|  \varepsilon^2\int_{\Sigma}   	Ve^{\sum_{h=1}^{k+l}P U_h} (e^{\phi} -\phi-1)\partial_{(\xi_i)_j}PU_i  \, dv_g  \right|\leq  |\varepsilon^2 h	e^{\sum_{h=1}^{k+l} PU_h} (e^{\phi} -\phi-1)   |_{L^p(\Sigma)}  
				|\partial_{(\xi_i)_j}PU_i|_{L^q(\Sigma)} \\
				&\leq & c \|\phi\|^2 \varepsilon ^{\frac{2-2pr}{pr}} |\partial_{(\xi_i)_j}PU_i|_{{L}^q(\Sigma)} \leq  c  \|\phi\|^2 \varepsilon ^{\frac{2-3pr}{pr}},
			\end{eqnarray*}
			where 	 $q\geq 1$ with $\frac{1}{p}+\frac{1}{q}=1$ and for any $r>1$. 
			By Lemma~\ref{lem13}, 
			\begin{eqnarray*}
				\left|\varepsilon^2	\int_{\Sigma}\sum_{h=1}^{k+l} \chi_h e^{ U_h} \phi (\chi_{i} \partial_{(\xi_i)_j} U_i-\partial_{(\xi_i)_j}PU_i) \, dv_g  \right|&\leq& c\sum_{h=1}^{k+l}\|\phi\| | \varepsilon^2\chi_h  e^{ U_h} (\chi_{i} \partial_{(\xi_i)_j}U_i-\partial_{(\xi_i)_j}PU_i)|_{L^p(\Sigma)}  \\
				&&\leq  c \|\phi\| \varepsilon^{\frac{2(1-p)}{p }}.  
			\end{eqnarray*}
			By Lemma~\ref{lem5}, 
			\begin{eqnarray*}
				\left| \varepsilon^2 \int_{\Sigma} ( \sum_{h=1}^{k+l}  \chi_he^{U_h} -Ve^{\sum_{h=1}^{k+l}  PU_h})  \phi\partial_{(\xi_i)_j}PU_i \right|
				&\leq& c\varepsilon^2\|\phi\| \left|  \sum_{h=1}^{k+l}  \chi_he^{U_h} - Ve^{\sum_{h=1}^{k+l}  PU_h}\right|_{L^p(\Sigma)}  \|\partial_{(\xi_i)_j}PU_i\|\\
				&\leq & c \|\phi\| \varepsilon^{\frac{2-p}{p}-1} = c\|\phi\| \varepsilon^{\frac{2(1-p)}{p }}. 
			\end{eqnarray*}
			In view of $\partial_{(\xi_i)_j}|y_{\xi_i}(x)|^2 = -2y_{\xi_i}(x)_j+ \mathcal{O}(|y_{\xi_i}(x)|^3)$ as $x\rightarrow \xi_i$, as $\varepsilon\rightarrow 0$
			\begin{eqnarray*}
				&& \varepsilon^2\int_{\Sigma}\chi_i  e^{ U_i} \phi\chi_{i}  \partial_{(\xi_i)_j} U_i \, dv_g\\
				&=&\varepsilon^2 \int_{\Sigma} e^{ U_i} \phi\chi_{i} e^{-\varphi_i}(1+\mathcal{O}(|y_{\xi_i}|^2)) \left(P\Psi^j_i + \mathcal{O}\left( \frac{|y_{\xi_i}|^3 }{\tau_i^2\varepsilon^2+|y_{\xi_i}|^2}\right)\right) \, dv_g+\mathcal{O}(\varepsilon^2)\\
				&=& \lan\phi, P\Psi^i_j\ran + \mathcal{O}(\varepsilon)=o(1).
			\end{eqnarray*}
			On the other hand, applying Lemma \ref{lem9} and Lemma \ref{lem10}, we deduce that 
			\begin{eqnarray*}
				&&\sum_{h=1}^{k+l}	\varepsilon^2 \int_{\Sigma}\chi_h  e^{-\varphi_h}e^{ U_h}\partial_{(\xi_i)_j}PU_i -\varepsilon^2 \int_{\Sigma} Ve^{\sum_{h=1}^{k+l} PU_h}\partial_{(\xi_i)_j}PU_i\\
				& =& \sum_{h=1}^{k+l}	\varepsilon^2 \int_{\Sigma}\chi_h e^{ U_h}\partial_{(\xi_i)_j}PU_i -\varepsilon^2 \int_{\Sigma} Ve^{\sum_{h=1}^{k+l} PU_h}\partial_{(\xi_i)_j}PU_i+o(1)=-\frac 1 2  \partial_{(\xi_i)_j}\cF^V_{k,l}(\xi) +o(1). 
			\end{eqnarray*}
			For any  $p\in (1,\frac 6 5 )$, take $r>1$ close to 1 enough such that $\frac{4-2p}{p}+ \frac{2-3pr}{pr}>0$. 
			Hence, 
			we have  as $\varepsilon\rightarrow 0$
			\[ \left\lan \sum_{h=1}^{k+l}  PU_h +\phi -i^*(\varepsilon^2 Ve^{\sum_{h=1}^{k+l} PU_h+\phi}),\partial_{(\xi_i)_j}PU_i\right\ran = -\frac{ 1}{2} \partial_{(\xi_i)_j }\cF^V_{k,l}(\xi) +o(1). \]  
		\end{proof}

		\section{The partial invertibility of the linearized operator}\label{linAp}
		\begin{altproof}	{  Lemma~\ref{t1} }
		Assume the conclusion in Lemma~\ref{t1}  does not hold. Then there exists $\xi\in {M}_{\delta}\subset \Xi^\prime_{k,l}$ for some small $\delta>0$,  a sequence $\varepsilon_n\rightarrow 0$ and $\phi_n\in K_{\xi}^{\perp}$ with $\|\phi_n\|=1$ and $ \|L^{\varepsilon_n}_{\xi}(\phi)\|=o(\frac{1}{|\log \varepsilon_n|})$.
		To simplify the notations, we use $\varepsilon$ instead of $\varepsilon_n$ and $\phi$ instead of $\phi_n$.
		\begin{equation}\label{wpsi}
			\phi -i^*(\varepsilon^2 Ve^{\sum_{i=1}^{k+l} PU_i}\phi) =  \psi+w,
		\end{equation} 
		where $\psi\in K_{\xi}^{\perp}$ and $w\in K_{\xi}$. Then $\|\psi\|=o(\frac{1}{|\log \varepsilon|})\rightarrow 0 . $
		It is equivalent that $\phi$ solves the following problem in the weak sense, 
		\[	\left\{\begin{aligned}
			(-\Delta_g+\beta)\phi=& \varepsilon^{2} Ve^{\sum_{i=1}^{k+l} PU_i} \phi -\overline{ \varepsilon^2 Ve^{\sum_{i=1}^{k+l} PU_i}\phi} +(- \Delta_g+\beta)(\psi+w), &\text{ in } \intsigma,\\
			\partial_{\nu_g} \phi=& 0,&\text{ on } \partial \Sigma. 
		\end{aligned}\right.
		\] 
		{\it  {Step 1.  $ \|w\|=o(1) $.}} \par 
		Given that $w\in K_{\xi}$, we have $w= \sum_{i=1}^{k+l}\sum_{j=1}^{\ii(\xi_i)} c_{ij}^{\varepsilon} P\Psi^j_i$. 
		Consider the inner product of equation \eqref{wpsi} with $P\Psi^{j'}_{i'}$, leading to the following equation:
		\begin{eqnarray*}
			&&	\langle \phi,P\Psi^{j'}_{i'}\rangle-\int_{\Sigma}  P\Psi^{j'}_{i'}\left( \varepsilon^2 Ve^{\sum_{i=1}^{k+l} PU_i} \phi -\frac 1 {|\Sigma|_g}\int_{\Sigma} \varepsilon^2 Ve^{\sum_{i=1}^{k+l} PU_i}  \phi \, dv_g \right) \, dv_g \\
			&=& \la\psi, P\Psi^{j'}_{i'}\ra + \la w,P\Psi^{j'}_{i'}\ra.  
		\end{eqnarray*}
		Since $P\Psi^{j'}_{i'}\in \oH$ and $\phi\in K_{\xi}^{\perp}$, we have 
		$\int_{\Sigma} P\Psi^{j'}_{i'} \, dv_g
		=0\text{ and } \la\psi, P\Psi^{j'}_{i'}\ra=\la\phi, P\Psi^{j'}_{i'}\ra=0. $
		It follows 
		\begin{equation}~\label{pp}
			-\varepsilon^2 \int_{\Sigma} Ve^{\sum_{i=1}^{k+l} PU_i}\phi P\Psi^{j'}_{i'} \, dv_g =\sum_{i=1}^{k+l} \sum_{j=1}^{\ii(\xi_i)} c_{ij}^{\varepsilon}  \la P\Psi^j_i, P\Psi^{j'}_{i'}\ra .
		\end{equation}
		Applying Lemma \ref{lem4}, the right-hand side of the equation \eqref{pp} equals   $$\frac{ 8\varrho(\xi_{i'})D_1}{\pi \tau^2_{i'}\varepsilon^2} c^{\varepsilon}_{i'j'}+\mathcal{O}(\varepsilon^{\alpha_0-1}\sum_{i=1}^{k+l}\sum_{j=1}^{\ii(\xi_i)} |c_{ij}^{\varepsilon}|).$$ The left-hand side of equation \eqref{pp} can be expanded as follows:
		\begin{eqnarray*}
			&\int_{\Sigma}\varepsilon^2 \left (\sum_{i=1}^{k+l} \chi_i  e^{U_i} -Ve^{\sum_{i=1}^{k+l} PU_i}\right) P\Psi^{j'}_{i'} \phi \, dv_g 
			-\sum_{i=1}^{k+l} \int_{\Sigma} \varepsilon^2 \chi_i  e^{U_i} (P\Psi^{j'}_{i'}-\chi_{i'} \Psi^{j'}_{i'})\phi \, dv_g\\
			&-\varepsilon^2 \int_{\Sigma}\chi_{i'}^2 (-e^{-\varphi_{i'}}+1)e^{U_{i'}}\Psi^{j'}_{i'}\phi \, dv_g-\varepsilon^2 \int_{\Sigma}\chi_{i'}^2 e^{-\varphi_{i'}}e^{U_{i'}}\Psi^{j'}_{i'}\phi \, dv_g.
		\end{eqnarray*}
		Since $\|\phi\|=1$ and $\phi\in K_{\xi}^{\perp}$,  
		$\int_{\Sigma}  	\varepsilon^2 \chi_{i'}^2e^{-\varphi_{i'}} e^{U_{i'}} \Psi^{j'}_{i'}   \phi   =\mathcal{O}(\varepsilon^2)+ \la P\Psi^{j'}_{i'},\phi\ra=\mathcal{O}(\varepsilon^2)$.
		By calculation, we have
		\begin{eqnarray*}
			\left| \varepsilon^2 \int_{\Sigma} (e^{-\varphi_{i'}}-1) \chi_{i'} e^{U_{i'}} \Psi^{j'}_{i'} \, dv_g    \right| 
			&\leq & \mathcal{ O}\left( \int_{|y|\leq 2 r_0} \frac{\tau_i\varepsilon |y|^2 \, d y  }{(\tau_i^2\varepsilon^2 +|y|^2)^3} \right)=\mathcal{O}(\varepsilon). 
		\end{eqnarray*}
		Applying Lemma \ref{lem4} and Lemma \ref{lem5},
		\begin{eqnarray*}
			&&	\left|  \int_{\Sigma} \varepsilon^2(\sum_{i=1}^{k+l} \chi_{i} e^{U_i}- Ve^{\sum_{i=1}^{k+l} PU_i}) P\Psi^{j'}_{i'} \phi \, dv_g \right|\\
			&\leq & C \left| \varepsilon^2(\sum_{i=1}^{k+l} \chi_{i}e^{U_i}- Ve^{\sum_{i=1}^{k+l} PU_i}) \right|_{L^p(\Sigma)} |\phi|_{L^q(
				\Sigma)} \| P\Psi^{j'}_{i'} \|\\
			&\leq & \mathcal{O}(\varepsilon^{\frac{2(1-p)}{p}} ), 
		\end{eqnarray*}
		where $\frac{1}{p}+\frac{1}{q}<1$ and $C>0$ is a constant. 
		Further, Lemma \ref{lemb3} implies 
		$ P\Psi^{j'}_{i'} -\chi_{i'}\Psi^{j'}_{i'} =\mathcal{O}(1). $ And applying Lemma \ref{lemb2},  for any $i=1,\cdots, k+l $
		\begin{eqnarray*}
			\left| \varepsilon^2 \int_{\Sigma} \chi_{i} e^{U_i} \phi( P\Psi^{j'}_{i'} -\chi_{i'}\Psi^{j'}_{i'}) \, dv_g \right|
			&\leq& \mathcal{O}( |\varepsilon^2 \chi_{i}e^{U_i}|_{L^p(\Sigma)} |\phi|_{L^q(\Sigma)})\\
			&\leq&  \mathcal{O}( \varepsilon^{\frac{2(1-p)}{p}}). 
		\end{eqnarray*}
		Combining these estimates, we conclude as $\varepsilon\rightarrow 0$
		\[   \frac{ 8\varrho(\xi_{i'})D_{i'}}{\pi \tau^2_{i'}\varepsilon^2} c^{\varepsilon}_{i'j'}+\mathcal{O}\left(\varepsilon^{\alpha_0-1}\sum_{i=1}^{k+l} \sum_{j=1}^{\ii(\xi_i)}|c_{ij}^{\varepsilon}|\right)
		{=} \mathcal{O}(\varepsilon^{\frac{2(1-p)}{p}}) . \]
		Then $|c_{i'j'}|{=} \mathcal{O}(\varepsilon^{\frac{2}{p}}) $, where $p\in (1,2)$. So
		\begin{equation}\label{eq:est_c_ij}
			\sum_{i=1}^{k+l} \sum_{j=1}^{\ii(\xi_i)} |c_{ij}^{\varepsilon}|=\mathcal{ O}(\varepsilon^{\frac{2}{p}})
		\end{equation} by the arbitrariness of $i'$ and $j'$. 
		Lemma~\ref{lem4} and~\eqref{eq:est_c_ij} yield that 
		\[ \|w\|^2= \left\| \sum_{i=1}^{k+l} \sum_{j=1}^{\ii(\xi_i)} c_{ij}^{\varepsilon} P\Psi^j_i \right\|^2=\mathcal{ O}\left( \sum_{i=1}^{k+l} \sum_{j=1}^{\ii(\xi_i)}|c_{ij}^{\varepsilon}|^2\frac{1}{\varepsilon^2} + \mathcal{O}(\varepsilon^{\alpha_0-1})\right){\leq}\mathcal{ O}(\varepsilon^{\frac{4}{p}-2}). \]
		Hence, it follows that as $\varepsilon\rightarrow 0$, $\|w\|=\mathcal{O}(\varepsilon^{\frac{2-p}{p}})\rightarrow 0$ for {any} $p\in (1,2)$. \par
		{\it {Step 2.  $\langle\phi, P\Psi^0_i\rangle\rightarrow 0.$}}\par
		Following the construction in~\cite{Esposito2005} and~\cite{Esposito2014singular}, we define 
		\[ \omega_i(y)=\frac{4}{3\tau_i} \log(\tau_i^2\varepsilon^2 +|y|^2 ) \frac{\tau_i^2\varepsilon^2 -|y|^2}{ \tau_i^2\varepsilon^2 +|y|^2} +\frac{8}{3\tau_i} \frac{\tau_i^2\varepsilon^2}{ \tau_i^2\varepsilon^2 +|y|^2}, \]
		and 	\[ t_i(y)=  -2 \frac{\tau^2_i \varepsilon^2 }{ \tau^2_i\varepsilon^2+|y|^2}. \]
	It holds that 
		$$
		\int_{\mathbb{R}^2}\left|\nabla \omega_i\right|^{2}=M_i^{2}(1+o(1))(\log \varepsilon)^{2}, \quad \int_{\mathbb{R}^2}\left|\nabla t_i\right|^{2}=O(1), \text { as } \varepsilon \rightarrow 0
		$$
		with $M_i=\frac{32}{3 \tau_i}\left(\int_{\mathbb{R}^{2}} \frac{|y|^{2}}{\left(1+|y|^{2}\right)^{4}}\right)^{1 / 2}$. 
		Let \[ u_i(x)=\chi_{i}(x)\left( \omega_i(y_{\xi_i}(x))+\frac{2\varrho(\xi_i)}{3\tau_i} H^g(\xi_i,\xi_i) t_i(y_{\xi_i}(x))\right) , \text{ for all } x \in U_{2r_0}(\xi_i). \]
		The projection $Pu_i\in \oH$ from $u_i$ is given by
		\begin{equation}~\label{pui}
			\begin{cases}
				(-\Delta_g+\beta) P u_i=-\chi_{i} \Delta_g u_i(x)+ \overline{\chi_{i} \Delta_g u_i(x)}& x\in \intsigma \\
				\partial_{ \nu_g } Pu_i=0& x\in \partial \Sigma\\
				\int_{\Sigma} Pu_i =0
			\end{cases}. 
		\end{equation} 
		Let us consider 
		$\eta_i:= u_i-Pu_i+\frac{2\varrho(\xi_i)}{3\tau_i} H^g(x,\xi_i)$. The integral of $\eta_i$ over $\Sigma$ is given by 
		$
		\int_{\Sigma} \eta_i \, dv_g  
		= \mathcal{ O}(\varepsilon^2\log^2 \varepsilon). 
		$
		If $\xi_i\in\intsigma$, we have $\partial_{\nu_g}\eta_i\equiv 0$ in $\partial \Sigma$. 
		For $\xi_i \in \partial \Sigma$,
		$ |\partial_{\nu_g}\eta_i(x)|_{L^p(\partial\Sigma)}=\mathcal{O}(\varepsilon^{\frac 1 p}|\log\varepsilon|)$.
		In view of $\int_{\R^2}\frac{1-|y|^2}{(1+|y|^2)^3}\log(1+|y|^2)\, d y =-\frac {\pi}2,$ and $\int_{\R^2} \frac 2 {(1+|y|^2)^3} \, d y= \int_{\R^2}\frac{1}{(1+|y|^2)^2}\, d y=\pi, $
		$$\left|(-\Delta_{g}+\beta) \eta_i\right|_{L^p(\Sigma
			)}={ \mathcal{O}(\varepsilon^{\frac 1 p}|\log\varepsilon|)}.$$
		By the  $L^p$-theory in Lemma~\ref{lem:inh_LP}, 
	$
			\| \eta_i-\overline{ \eta_i}\|_{W^{2,p}(\Sigma)}\leq C \varepsilon^{\frac 1 p} |\log \varepsilon|,
	$
		for any $p>1$. 
		Applying Sobolev inequality,
		$ \left|\eta_i-\overline{\eta_i} \right|_{C^{\gamma}(\Sigma)}\leq C \varepsilon^{\frac 1 p} |\log \varepsilon|, $
		for any $\gamma \in (0, 2(1-\frac 1 p)). $
		Choosing $p\in (1,2]$, we deduce that 
		\begin{eqnarray}~\label{eq4}
			{	|\eta_i|\leq \mathcal{ O}(\varepsilon^{\frac 1 p}|\log \varepsilon|). }
		\end{eqnarray}
		Moreover, for any $x\in \Sigma\setminus\{\xi_i\}$, the following inequality holds:
		\begin{equation}~\label{eq3}
			\left|	Pu_i(x)- \frac{2\varrho(\xi_i)}{3\tau_i}G^g(x,\xi_i)\right|\leq \mathcal{O}(\varepsilon^{\frac 1 p}|\log\varepsilon|).
		\end{equation}
		Additionally, $\|Pu_i\|^2$ is computed directly as
		\begin{eqnarray*}
			\|Pu_i\|^2& =& \la Pu_i,Pu_i\ra=-\int_{\Sigma}\chi_i\left(u_i+\frac{2\varrho(\xi_i)}{3\tau_i}H^g_{\xi_i}+\mathcal{O}(\varepsilon^{\frac 1 p}|\log\varepsilon|)\right)\Delta_gu_i\\
			&=& \mathcal{O} (|\log\varepsilon|^2). 
		\end{eqnarray*}
		Thus  as $\varepsilon\rightarrow 0$
		\begin{eqnarray}~\label{eq2}
			\|Pu_i\|=\mathcal{O} (|\log \varepsilon |).
		\end{eqnarray}
		Applying $Pu_i$ as a test function for  \eqref{wpsi}, 
		\begin{eqnarray*}
			\la Pu_i, \phi \ra
			-\int_{\Sigma} \varepsilon^2 Ve^{\sum_{h=1}^{k+l}PU_h}\phi Pu_i \, dv_g = \la Pu_i, w+\psi\ra. 
		\end{eqnarray*}
		Considering 
		$|\la Pu_i, w+\psi\ra|\leq \|Pu_i\| (\|w\|+\|h\|)\leq \|Pu_i\|o\left(\frac{1}{|\log \varepsilon|}\right)=o(1),$
		we deduce that 
		\begin{eqnarray}~\label{eq5}
			\la Pu_i, \phi\ra 
			-\int_{\Sigma} \varepsilon^2V e^{\sum_{h=1}^{k+l} PU_h}\phi Pu_i \, dv_g(x) =o(1). 
		\end{eqnarray}
		By~\eqref{pui} and $\|\phi\|=1$ with the H\"{o}lder inequality, 
		\begin{eqnarray}~\label{puii}
	&&\la Pu_i,\phi\ra=\int_{\Sigma} (-\chi_i \Delta_g u_i+\overline{\chi_i \Delta_g u_i})\phi \, dv_g \\
			&=& \int_{\Sigma}\varepsilon^2  e^{U_i}u_i \, dv_g 
			+\int_{\Sigma}\frac{2\varrho(\xi_i)}{3\tau_i} H^g(x,\xi_i)\varepsilon^2\chi_i e^{U_i}\phi \, dv_g+\la P\Psi^0_i,\phi\ra
			+\mathcal{O}(\varepsilon^{\frac{2-q} q}),\nonumber
		\end{eqnarray}
		for any $q\in (1,2)$.
		On the other hand, \eqref{lemb2} and \eqref{eq4} with the H\"{o}lder inequality yield that 
		\begin{eqnarray}\label{puii2}
			\quad	\int_{\Sigma}\varepsilon^2 Ve^{\sum_{h=1}^{k+l} PU_h}\phi Pu_i \, dv_g &=& \int_{\Sigma}\varepsilon^2  e^{U_i}u_i \, dv_g 
			+\int_{\Sigma}\frac{2\varrho(\xi_i)}{3\tau_i} H^g(x,\xi_i)\varepsilon^2\chi_i e^{U_i}\phi \, dv_g\nonumber\\
			&&+\mathcal{O}(\varepsilon^{\frac 1 p +2\left(\frac 1 s-1\right)}),\nonumber
		\end{eqnarray}
		for any $s\in(1,2)$. 
		We choose $s,p$ sufficiently close to $1$ such that $\frac 1 p+ 2(\frac 1 s-1)>0$. Then~\eqref{puii} and~\eqref{puii2} imply that 
		$\la P\Psi^0_i,\phi\ra=o(1),$ as $\varepsilon\rightarrow 0.$
		\par
		{\it Step 3. Construct a contradiction.}\\
		Define the following space for $\xi=(\xi_1,\cdots, \xi_{k+l})\in M_{\delta}$. We denote that $\R_i=\R^2$ if $1\leq i\leq k$; $\R_i=\R^2_+:=\{ 
		y\in \R^2: y_2\geq 0\}$ if $ k+1\leq i\leq m.$
		Let $\pi_N$ be the stereographic projection through the north pole for the standard unit sphere in $\mathbb{R}^3$.
		We denote that	$S_i=\pi_{N}(\R_i)$ for $i=1,\cdots, k+l .$ We define 
		\[ L_i:=\left\{  \Psi: \left| \frac{\Psi}{1+|y|^2}\right|_{L^2(\R_i)} <+\infty\right\}, \]
		and 
		\[ H_i:=\left\{\Psi:  |\nabla \Psi|_{L^2(\R_i)} + \left| \frac{\Psi}{1+|y|^2} \right|_{ L^2(\R_i)   }<\infty\right\}\]
		The associated norms are defined as the following, 
		\[ \|\Psi\|_{L_i}:=\left| \frac{\Psi}{1+|y|^2}\right|_{L^2(\R_i)}\text{ and } 
		\|\Psi\|_{H_i}:=|\nabla \Psi|_{L^2(\R_i)} + \left| \frac{\Psi}{1+|y|^2} \right|_{L^2(\R_i)}.  \]
		The maps
		\begin{eqnarray}
			L_i\rightarrow L^2(S_i): 		\Psi\mapsto \Psi\circ \pi_N   
		\end{eqnarray}
		and 
		$  H_i\rightarrow H^1(S_i): 		\Psi\mapsto \Psi\circ \pi_N $
		are isometric. 
		Let $\Omega_i^{\varepsilon}:= \frac{1}{\tau_{i}\varepsilon} B^{\xi_{i}}_{2r_0}, \phi^{\varepsilon}_i(x)=\phi(y_{\xi_i}^{-1}( {\tau_i\varepsilon}{y}))$ and $\chi^{\varepsilon}_i(y)= \chi(\tau_i\varepsilon| y|)$. 
		Consider 
		\[ \tilde{\phi}^{\varepsilon}_i=\begin{cases}
			\phi_i^{\varepsilon} \chi_i^{\varepsilon}  & y\in\Omega_i^{\varepsilon} \\
			0&  y\in \R_i\setminus \Omega_i^{\varepsilon} 
		\end{cases} .\]
		By Lemma \ref{lem5} and H\"{o}lder inequality,  we have 
		\begin{eqnarray*}
			\sum_{h=1}^{k+l} 	\varepsilon^2\int_{\Sigma} e^{-\varphi_h}\chi_h e^{U_h} \phi^2 \, dv_g &=& \varepsilon^2\int_{\Sigma} Ve^{\sum_{h=1}^{k+l}PU_h} \phi^2 \, dv_g\\
			&& +\mathcal{O}\left( \int_{\Sigma} \varepsilon^2 |\sum_{h=1}^{k+l} e^{-\varphi_h}\chi_he^{U_h}-Ve^{\sum_{h=1}^{k+l} PU_h}| \phi^2 \, dv_g \right)\\
			&=& \varepsilon^2\int_{\Sigma} Ve^{\sum_{h=1}^{k+l}PU_h} \phi^2 \, dv_g+o(1),
		\end{eqnarray*}
		where $p\in (1,2)$ and $\frac{1}{p}+\frac{1}{q}=1$. 
		On the other hand,  we take the inner product of~\eqref{wpsi} with $\phi$, since $\|\phi\|=1$ and $\|\psi\|=o(\frac{1}{\log \varepsilon}), $
		\begin{equation*}
			\varepsilon^2\int_{\Sigma} Ve^{\sum_{h=1}^{k+l}PU_h} \phi^2 \, dv_g =\la \phi,\phi\ra -\la w+\psi,\phi\ra =1+ o(1).  
		\end{equation*} 
		By direct calculation, we have 
		\begin{eqnarray*}\sum_{i=1}^{k+l}\varepsilon^2\int_{\Sigma}e^{-\varphi_i}\chi_i e^{U_i}\phi^2 \, dv_g &=&\sum_{i=1}^{k+l} \int_{B^{\xi_i}_{2r_0}} \frac{8\tau_i^2\varepsilon^2\chi^2(|y|/r_0)}{(\tau_i^2\varepsilon^2+|y|^2)^2}(\phi\circ y_{\xi_i}^{-1}(\tau_i\varepsilon y))^2 \, d y +\mathcal{O}(\varepsilon^2) \\
			&=&8\sum_{i=1}^{k+l} \int_{\R_i}\frac{ |\tilde{\phi}^{\varepsilon}_i(y)|^2}{(1+|y|^2)^2} \, d y+\mathcal{O}(\varepsilon^2);\\
			\int_{\Omega_i^{\varepsilon}} |\nabla \tilde{\phi}^{\varepsilon}_i|^2 \, d y &=&\int_{\frac 1 {\tau_i\varepsilon}B_{2r_0}^{\xi_i}} |\chi_i^{\varepsilon} \nabla \phi^{\varepsilon}_i + \phi_i^{\varepsilon} \nabla \chi_i^{\varepsilon}|^2 \, d y\\
			&=& \mathcal{O}\left(\int_{\Sigma} |\nabla \phi|_g^2 \, dv_g + \int_{\Sigma} e^{-\varphi_i}|\phi(x)|^2 \, dv_g \right)= \mathcal{ O}\left( \|\phi\|\right)=\mathcal{ O}(1). 
		\end{eqnarray*}
		Hence $\tilde{\phi}^{\varepsilon}_i$ is bounded in $H_i$. We observe that  $H_i$  compactly embeds into $L_i$.  
		Up to a subsequence, as $\varepsilon\rightarrow 0$,
		$ \tilde{\phi}^{\varepsilon}_i\rightarrow  \tilde{\phi}_i^0$
		weakly in $H_i$ and strongly in $L_i$. 
		\begin{eqnarray}~\label{scaleL}
			\sum_{i=1}^{k+l}	\|\tilde{\phi}_i^0\|^2_{L_i}= \frac 1 {8}.
		\end{eqnarray}
		For any $h\in C_c^{\infty}(\mathbb{R}^2)$, assume that  $\text{ supp }h \subset B_{R_0}(0).$
		If $\tau_i \varepsilon <\frac {r_0}{R_0}$, then 
		$\text{ supp }\nabla \chi\left(\frac{|y|}{r_0}\right) \cap \text{{ supp }} h\left(\frac{1}{\tau_i\varepsilon} y\right) =\emptyset. $
		For any $\Phi\in \oH$, 
		\begin{eqnarray}
			~\label{eqp1}
			0&=&\int_{B_{2r_0}^{\xi_i}} \Phi\circ y_{\xi_i}^{-1}(y)\nabla \chi\left(\frac{|y|}{r_0}\right) \cdot \nabla  h\left(\frac{1}{\tau_i\varepsilon} y\right) \, d y\\
			&=&
			\int_{B_{2r_0}^{\xi_i}} h\left(\frac{1}{\tau_i\varepsilon} y\right)\nabla( \Phi\circ y_{\xi_i}^{-1}(y))\cdot \nabla \chi\left(\frac{|y|}{r_0}\right)   \, d y. \nonumber
		\end{eqnarray}
		In \eqref{eqp1}, we take $\Phi=\phi, w\text{ and } \psi$, respectively.\par 
	For any $\|h\|:=(\int_{\mathbb{R}^2} |\nabla h|^2+|h|^2) ^{\frac 1 2}\leq 1$ and $h\in C_c^{\infty}(\mathbb{R}^2)$, it holds
		\begin{eqnarray}~\label{eqp2}
			\int_{\Sigma} \chi_ih^2\left(\frac{1}{\tau_i\varepsilon} y_{\xi_i}(x) \right) \, dv_g(x)= O(\varepsilon^2) 
		\end{eqnarray}
		{and }
		\begin{eqnarray}\label{eqp3}
			\int_{\Sigma} \chi_i\left |\nabla h\left(\frac{1}{\tau_i\varepsilon} y_{\xi_i}(x) \right) \right|_g^2  \, dv_g(x)= \mathcal{ O}(1) . 
		\end{eqnarray}
	Combining the result in {\it Step 1} and $\|\psi\|=o(\frac 1 {|\log\varepsilon|})$, 
		\begin{equation}~\label{eqp5}
			\|w\|+\|\psi\|=o(1).
		\end{equation} 
		Assume that $0<\varepsilon< \frac{r_0}{\tau_i R_0}$, as $\varepsilon\rightarrow 0$
		\begin{eqnarray*}
			&&\int_{\R_i} \nabla\tilde{\phi}^{\varepsilon}_i \nabla h \, d y =
			\int_{B_{2r_0}^{\xi_i}} \nabla\left( \chi\left(\frac{|y|}{r_0}\right) \phi\circ y_{\xi_i}^{-1}(y)\right)\cdot\nabla h\left(\frac{1}{\tau_i\varepsilon} y\right) \, d y\\
			&\stackrel{\eqref{eqp1}}{=}&  \int_{B_{2r_0}^{\xi_i}} \nabla \phi\circ y_{\xi_i}^{-1}(y)\cdot\nabla\left( \chi\left(\frac{|y|}{r_0}\right)  h\left(\frac{1}{\tau_i\varepsilon} y\right) \right)\, d y \\
			&=& \int_{\Sigma} \left \la\nabla \phi, \nabla \left( \chi_i(x) h\left(\frac{1}{\tau_i\varepsilon} y_{\xi_i}(x)\right)\right) \right\ra_g \, dv_g \\
			&\stackrel{\eqref{wpsi}}{=}& -\beta \int_{\Sigma}\chi_i h\left(\frac{1}{\tau_i\varepsilon} y_{\xi_i}(x) \right) \phi \, dv_g +\int_{\Sigma} \varepsilon^2\chi_i Ve^{\sum_{h=1}^{k+l}PU_h} \phi h\left(\frac{1}{\tau_i\varepsilon} y_{\xi_i}(x) \right)\, dv_g(x) \\
			&&- \frac 1 {|\Sigma|_g}\int_{\Sigma} \varepsilon^2 Ve^{\sum_{h=1}^{k+l}PU_h} \phi \, dv_g(x) \int_{\Sigma}\chi_{i}  h\left(\frac{1}{\tau_i\varepsilon} y_{\xi_i}(x) \right)\, dv_g(x) \\
			&&+ \int_{\Sigma} \left \la\nabla (w+\psi), \nabla \left( \chi_i(x) h\left(\frac{1}{\tau_i\varepsilon} y_{\xi_i}(x)\right)\right) \right\ra_g \, dv_g \\
			&=&-\tau_i^2\varepsilon^2\beta\int_{\R_i}\tilde{\phi}^{\varepsilon}_i(y) h(y) \, d y +\int_{\Sigma} \varepsilon^2\chi_i Ve^{\sum_{h=1}^{k+l}PU_h} \phi h\left(\frac{1}{\tau_i\varepsilon} y_{\xi_i}(x) \right)\, dv_g(x)\\
			&&-\tau_i^2\varepsilon^2\int_{\Sigma} \varepsilon^2 Ve^{\sum_{h=1}^{k+l}PU_h} \phi \, dv_g(x) \int_{\R_i} \chi\left(\frac{\tau_i\varepsilon |y|}{r_0}\right)e^{\varphi_{i}(\tau_i\varepsilon y)}h(y) \, d y+o(1),
		\end{eqnarray*}
		for any $h\in C_c^{\infty}(\mathbb{R}^2)$ with $\|h\|\leq 1$.
		By the H\"{o}lder inequality and Lemma~\ref{lem5}, 
		\[\left|\int_{\R_i}\tilde{\phi}^{\varepsilon}_i(y) h(y) \, d y\right| 
		\leq \|\tilde{\phi}^{\varepsilon}_i(y) \|_{L_i}\left(\int_{\R_i} (1+|y|^2)^2|h(y)| \, d y\right)^{\frac 12 }\leq C \|\tilde{\phi}^{\varepsilon}_i(y) \|_{L_i}\|h\|, \]
		and 
		\begin{eqnarray*}
			&&\left|\varepsilon^2 \int_{\Sigma} \varepsilon^2 \chi_i Ve^{\sum_{h=1}^{k+l}PU_h} \phi \right|\leq C\left(\varepsilon^{\frac{2}{p}}\|\phi\|\right),
		\end{eqnarray*}
		where $C>0$ is a constant depending only on $R_0$,
		Applying Lemma~\ref{lemb1} and \eqref{tau},
		\begin{eqnarray*}
				& &\int_{\Sigma}  \varepsilon^2\chi_i V e^{\sum_{h=1}^{k+l} P U_h} \phi h\left(\frac{1}{\tau_i \varepsilon} y_{\xi_i}\right) \, dv_g \\
				&= & \int_{U_{2 r_0}\left(\xi_i\right)}	\frac{ 8 \tau_i^2\varepsilon^2\chi_i}{\left(\tau_i^2 \varepsilon^2+\left|y_{\xi_i}\right|^2\right)^2} \exp\{-\log(8\tau_i^2)+\varrho\left(\xi_i\right) H^g\left(\xi_i, \xi_i\right)\\
				&&+\sum_{h \neq i} \varrho\left(\xi_h\right) G^g\left(\xi_i, \xi_h\right)+\log V\left(\xi_i\right)+\mathcal{O}\left(\left|y_{\xi_i}\right|+\varepsilon^{1+\alpha_0}\right)\}
				\phi h\left(\frac{1}{\tau_i \varepsilon} y_{\xi_i}\right) \, dv_g \\
				&= &\int_{\R_i}\frac{8}{\left(1+|y|^2\right)^2} \tilde{\phi}^{\varepsilon}_i h(y) \, d y+o(1).
		\end{eqnarray*}
		Then $\tilde{\phi}^0_i$ is a distributional solution for  the equation 
		\begin{equation}~\label{eq8u}
			-\Delta U= \frac{8}{(1+|y|^2)^2 } U \text{ in } \R_i \text{ with}  \int_{\mathbb{R}^2} |\nabla U|^2 \, d y <\infty,
		\end{equation} 
		with boundary condition $\partial_{\nu_0} U=0$ on $\partial\R_i$, where $\nu_0$ is  the unit outward normal of 
		$\partial\R_i$. 
		By the regularity theory, $\tilde{\phi}_i^0$ is a smooth solution. 
		It is well-known that any solutions to problem~\eqref{eq8u} are in the following form,
		$\tilde{\phi}^0_i(y)=\frac{a^i_0(1-|y|^2)}{1+|y|^2} +\sum_{j=1}^{\ii(\xi_i)} \frac{a^i_j y_j}{1+|y|^2}, $
		where $a_j^i\in \mathbb{R}$ for $i=1,\cdots, k+l , j=0,\cdots,\ii(\xi_i)$ (see Lemma D.1. of \cite{Esposito2005}). 
		
		Applying the result from {\it Step 2.}, 
		\begin{eqnarray*}
			&&\frac{16}{\tau_i}	\int_{\R_i} \frac{|y|^2-1}{  (|y|^2+1)^3} \tilde{\phi}^0_i(y) \, d y = \lim_{\varepsilon\rightarrow 0}
			\frac{16}{\tau_i}	\int_{\Omega_i^{\varepsilon}} \frac{|y|^2-1}{  (|y|^2+1)^3}  \phi_i^{\varepsilon}\chi_i^{\varepsilon} \, d y \\
			&=& \lim_{\varepsilon\rightarrow 0}\int_{B_{2r_0}^{\xi_i}} \varepsilon^2 e^{u_{\tau_i,0}} \psi^0_{\tau_i,0} \phi\circ y^{-1}_{\xi_i}(y)\chi(|y|)  \, d y= \lim_{\varepsilon\rightarrow 0}\int_{\Sigma}  \varepsilon^2\chi_{i} e^{-\varphi_i} e^{U_i} \Psi^0_i   \phi \, dv_g \\
			&=& \lim_{\varepsilon\rightarrow 0}\la P\Psi^0_i, \phi\ra =0. 
		\end{eqnarray*}
		For any $i=1,\cdots, k+l $ and $j=1,\cdots, \ii(\xi_i)$, 
		\begin{eqnarray*}
			&&\frac{32}{\tau_i \varepsilon}	\int_{\R_i} \frac{y_j}{  (|y|^2+1)^3} \tilde{\phi}^0_i \, d y = \lim_{\varepsilon\rightarrow 0} \frac{32}{\tau_i\varepsilon}
			\int_{\Omega_i^{\varepsilon}} \frac{y_j}{  (|y|^2+1)^3}  \phi^{\varepsilon}_i\chi^{\varepsilon}_i \, d y \\
			&=& \lim_{\varepsilon\rightarrow 0} \int_{B_{2r_0}^{\xi_i}} \varepsilon^2\chi\left(\frac{|y|}{r_0}\right) e^{u_{\tau_i,0}}  \psi^j_{\tau_i,0} \phi\circ  y^{-1}_{\xi_i}(y) \, d y\\
			&=& \lim_{\varepsilon\rightarrow 0}\int_{U_{2r_0}(\xi_i)} \varepsilon^2\chi_{i} e^{U_i} e^{-\varphi_i} \Psi^j_i \phi(x) \, dv_g =\lim_{\varepsilon\rightarrow 0} \la P\Psi^j_i,\phi\ra=0. 
		\end{eqnarray*}
		Thus 	for any $i=1,\cdots, k+l , j=1,\cdots, \ii(\xi_i)$
		$
		\int_{\R_i} \frac{|y|^2-1}{(|y|^2+1)^{3}} \tilde{\phi}_i^0 \, d y  =  	\int_{\R_i } \frac{y_j}{(|y|^2+1)^{3}} \tilde{\phi}_i^0 \, d y =0.
		$
		It indicates that $\tilde{\phi}_i^0\equiv 0, $ which contradicts to \eqref{scaleL}.  
		\end{altproof}

	\end{appendices}
	

	\bibliographystyle{plain} 
	\bibliography{KS_2024} 
	\vspace{2mm}\noindent
	\noindent
	{\sc Mohameden Ahmedou, Thomas Bartsch}\\
	{\it Mathematisches Institut,
		Universit\"at Giessen}\\
	Arndtstr.\ 2, 
	35392 Giessen, Germany\\
	\textsf{\href{mailto:Mohameden.Ahmedou@math.uni-giessen.de}{Mohameden.Ahmedou@math.uni-giessen.de}}
	\\
	\textsf{\href{mailto:Thomas.Bartsch@math.uni-giessen.de}{Thomas.Bartsch@math.uni-giessen.de}}
	\\
	{\sc Zhengni Hu}\\
	{\it School of Mathematical Sciences,
		Shanghai Jiao Tong University}\\
	800 Dongchuan RD, Minhang District, 200240 Shanghai, China\\
	\textsf{\href{mailto: zhengni_hu2021@outlook.com}{ zhengni\_hu2021@outlook.com}}
	\\
\end{document}